%% file: main.tex
\documentclass[11pt,reqno,a4paper]{amsart}
\input{packages.tex}

\begin{document}

\newgeometry{top=3cm, bottom=3.5cm,left=3cm,right=3cm}

\title{Hyperuniformity of random measures on Euclidean and hyperbolic spaces}
\author{Michael Björklund}\author{Mattias Byl\'ehn}
\subjclass[2020]{Primary: 60G55. Secondary: 60G57, 43A90, }
\keywords{Hyperuniformity, spherical diffraction, asymptotic properties of spherical functions}

\pagestyle{plain}
\pagenumbering{arabic}

\input{Abstract}

\maketitle

\section{Introduction}
\input{Introduction}

\newpage

\begin{spacing}{0.1}
\tableofcontents
\end{spacing}

\section{Random measures on homogeneous spaces}
\label{Random measures on homogeneous spaces}
\input{Setup}

\section{Euclidean harmonic analysis}
\label{Euclidean harmonic analysis}
\input{EuclideanDiffraction}

\section{Beck's Theorem}
\label{Beck's Theorem}
\input{BecksThm}

\section{Spherical harmonic analysis on real hyperbolic spaces}
\label{Spherical harmonic analysis on real hyperbolic spaces}
\input{HyperbolicSpace}

\section{Hyperbolic diffraction measures}
\label{Hyperbolic diffraction measures}
\input{HyperbolicDiffraction}

\section{A hyperbolic analouge of Beck's theorem}
\label{A hyperbolic analouge of Beck's theorem}
\input{HyperbolicBeck}

\section{Spectral hyperuniformity and stealth}
\label{Spectral hyperuniformity and stealth}
\input{SpectralHyperuniformity}

\section{Random perturbed lattice orbits in hyperbolic spaces}
\label{Random perturbed lattice orbits in hyperbolic spaces}
\input{HyperbolicRandomPerturbedLatticeOrbits}

\printbibliography

{\small\textsc{Department of Mathematics, Chalmers and University of Gothenburg, Gothenburg, Sweden \\
 \textit{Email address}: {\tt micbjo@chalmers.se}}}

{\small\textsc{Department of Mathematics, Chalmers and University of Gothenburg, Gothenburg, Sweden \\
 \textit{Email address}: {\tt bylehn@chalmers.se}}}

\end{document}

%% file: packages.tex
\usepackage{setspace}
\usepackage[english]{babel}
\usepackage[utf8]{inputenc}
\usepackage{tikz}
\usetikzlibrary{cd}
\usepackage{amsfonts}  
\usepackage{amsmath} 
\usepackage{amsthm}
\numberwithin{equation}{section}
\usepackage{amssymb}
\usepackage{mathtools}
\usepackage{mathrsfs}
\usepackage[pdftex]{hyperref}
\usepackage{mathrsfs}
\usepackage{setspace}
\usepackage{csquotes}

\allowdisplaybreaks

\usepackage{units}
\usepackage{icomma}
\usepackage{graphicx}
\usepackage{tcolorbox}
\usepackage{listings}
\usepackage{relsize}
\usepackage{subcaption}
\usepackage{mdframed}
\usepackage{caption}
\usepackage{amsthm}
\usepackage{rotating}
\usepackage{tikz-cd}
\usepackage[margin=1in,footskip=.25in]{geometry}
\usepackage{parskip}
\usepackage{fancyhdr}

\patchcmd{\subsection}{-.5em}{.5em}{}{}
\patchcmd{\subsubsection}{-.5em}{.5em}{}{}

\makeatletter
\def\thm@space@setup{%
  \thm@preskip=\parskip \thm@postskip=0pt
}
\makeatother

\newcommand{\e}{\ensuremath{\mathrm{e}}} 

\newcommand{\Id}{\mathrm{Id}}

\newcommand{\N}{\mathbb{N}}
\newcommand{\Z}{\mathbb{Z}}
\newcommand{\Q}{\mathbb{Q}}
\newcommand{\R}{\mathbb{R}}
\newcommand{\C}{\mathbb{C}}

\renewcommand{\H}{\mathbb{H}}

\newcommand{\GL}{\mathrm{GL}}
\newcommand{\SO}{\mathrm{SO}}

\newcommand{\SL}{\mathrm{SL}}

\newcommand{\Isom}{\mathrm{Isom}}
\newcommand{\Stab}{\mathrm{Stab}}
\newcommand{\Vol}{\mathrm{Vol}}

\newcommand{\covol}{\mathrm{covol}}
\newcommand{\Prob}{\mathrm{Prob}}

\newcommand{\Var}{\mathrm{Var}}
\newcommand{\Cov}{\mathrm{Cov}}
\newcommand{\supp}{\mathrm{supp}}

\newcommand{\rad}{\mathrm{rad}}

\newcommand{\Emu}{\mathbb{E}_{\mu}}
\newcommand{\Covmu}{\mathrm{Cov}_{\mu}}
\newcommand{\Varmu}{\mathrm{Var}_{\mu}}
\newcommand{\Vmu}{\mathrm{NV}_{\mu}}
\newcommand{\NVmu}{\mathrm{NV}_{\mu}}
\newcommand{\NV}{\mathrm{NV}}

\newcommand{\Poi}{\mathrm{Poi}}

\newcommand{\arccosh}{\mathrm{arccosh}}

\newcommand{\arctanh}{\mathrm{arctanh}}

\renewcommand{\Re}{\mathrm{Re}}
\renewcommand{\Im}{\mathrm{Im}}

\newcommand{\Ccinfty}{C_c^{\infty}}

\newcommand{\Borelinfty}{\mathscr{L}^{\infty}}
\newcommand{\Borelbndinfty}{\mathscr{L}^{\infty}_{\mathrm{bnd}}}

\newcommand{\Radonplus}{\mathscr{M}_+}

\newcommand{\Av}{\mathrm{Av}}

\renewcommand{\hat}{\widehat}
\renewcommand{\tilde}{\widetilde}

\newcommand{\norm}[1]{\lVert#1\rVert}

\newcommand{\bE}{\mathbb{E}}

\newcommand{\bS}{\mathbb{S}}

\newcommand{\sF}{\mathscr{F}}

\newcommand{\sM}{\mathscr{M}}

\newcommand{\sS}{\mathscr{S}}

\newcommand{\rB}{\mathrm{B}}

\definecolor{lichtgrijs}{gray}{1.00}
\mdfdefinestyle{mystyle}{ %
    backgroundcolor=lichtgrijs, %
    linewidth=0pt, %
    innertopmargin=5pt, %
    innerbottommargin=10pt,%
}

\newtheorem{theorem}{Theorem}[section]
\newtheorem{corollary}[theorem]{Corollary}
\newtheorem{proposition}[theorem]{Proposition}
\newtheorem{lemma}[theorem]{Lemma}

\theoremstyle{definition}
\newtheorem{definition}[theorem]{Definition}
\newtheorem{remark}[theorem]{Remark}

\newtheorem{example}{Example}[section]

\tikzstyle{decision} = [diamond, draw, fill=blue!20, 
    text width=4.5em, text badly centered, node distance=3cm, inner sep=0pt]
\tikzstyle{block} = [rectangle, draw, fill=blue!20, 
    text width=5em, text centered, rounded corners, minimum height=2em]
\tikzstyle{line} = [draw, -latex']
\tikzstyle{cloud} = [draw, ellipse,fill=red!20, node distance=3cm,
    minimum height=2em]

\usepackage[bibstyle=numeric]{biblatex}

\renewbibmacro{in:}{}
\bibliography{Bibliography} 

\usepackage{hyperref}


%% file: Abstract.tex
\begin{abstract}
We investigate lower asymptotic bounds of number variances for invariant locally square-integrable random measures on Euclidean and real hyperbolic spaces. In the Euclidean case we show that there are subsequences of radii for which the number variance grows at least as fast as the volume of the boundary of Euclidean balls, generalizing a classical result of Beck. 
With regards to real hyperbolic spaces we prove that random measures are never geometrically hyperuniform and if the random measure admits non-trivial complementary series diffraction, then it is hyperfluctuating. Moreover, we define spectral hyperuniformity and stealth of random measures on real hyperbolic spaces in terms of vanishing of the complementary series diffraction and sub-Poissonian decay of the principal series diffraction around the Harish-Chandra $\Xi$-function. 
\end{abstract}

%% file: Introduction.tex
In the field of statistical mechanics for low-temperature states of matter, the notion of \emph{hyperuniformity}, or \emph{superhomogeneity}, as introduced by Stillinger and Torquato in the seminal paper \cite{StillingerTorquato} in 2003, quantifies long-range order of an invariant point process in terms of the degree of suppression of large-scale density fluctuations. Equivalently, hyperuniformity is characterized by the suppression of small frequencies in the diffraction picture. Hyperuniformity and the associated number variances of invariant point processes have been studied extensively in Euclidean spaces, examples including invariant random lattice shifts, i.i.d. perturbations of such, certain determinantal point processes, most mathematical quasicrystals and more. See \cite{Torquato2018HyperuniformSO} for an extensive survey and \cite{Bjorklund2022HyperuniformityAN} for details on mathematical quasicrystals.

In this article we investigate lower asymptotic bounds of number variances, i.e. the variance of the measure of asymptotically large balls, for invariant locally square-integrable random measures on Euclidean and real hyperbolic spaces. We prove for general invariant locally square-integrable random measures on Euclidean space that there are unbounded sequences of centered Euclidean balls such that the number variance grows at least as fast as the volume of the boundary of the balls, extending a Theorem of Beck in the case of a locally finite point set \cite[Theorem 2A]{Beck1987IrregularitiesOD}.
Our methods are Fourier theoretic and utilizes the \emph{diffraction measure} of a random measure, generalizing the notion of the \emph{structure factor} of a point process. 
 
As for random measures on real hyperbolic space, we investigate contending notions of hyperuniformity. This is also done in terms of the diffraction measure associated to such a random measure, which for hyperbolic spaces (and more generally, commutative spaces) is defined using the available \emph{spherical transform}, generalising the Hankel transform in Euclidean space. We show that, contrary to the Euclidean setting, the large-scale number variance of a random measure on real hyperbolic space grows at least as fast as the volume of large metric balls along some unbounded sequence of metric balls. Nevertheless, we are able to define a notion of spectral hyperuniformity on the principal part of the relevant frequency domain that is distinctly different from measuring large-scale density fluctuations. Instead, our definition favours the property of having suppressed ''small scale'' frequency fluctuations to characterize hyperuniformity. 

\subsection{Random measures}

Let $X$ denote the space $\R^n$ with the action of translations or real hyperbolic space $\H^n$ with the action of orientation-preserving isometries. By a \emph{random measure} on $X$ we mean a probability measure $\mu$ on the cone $\Radonplus(X)$ of positive locally finite measures on $X$, which can be thought of as the law of a $\Radonplus(X)$-valued random variable. One usually studies such random measures in terms of \emph{linear statistics} $Sf : \Radonplus(X) \rightarrow \C$ associated to bounded measurable functions $f : X \rightarrow \C$ with bounded support, given by
\begin{align*}
Sf(p) = \int_{X} f(x) dp(x) \, . 
\end{align*}
Given an invariant locally square-integrable random measure $\mu$ on $X$, we denote by $\Vmu : \R_{\geq 0} \rightarrow \R_{\geq 0}$ the \emph{number variance} with respect to $\mu$,
\begin{align}
\Vmu(r) =  \Varmu(S\chi_{B_r}) = \int_{\sM_+(X)} |p(B_r) - \iota_{\mu} \Vol_X(B_r)|^2 d\mu(p) 
\end{align}
where $B_r = B_r(o)$ is the metric ball of radius $r > 0$ centered at a fixed reference point $o \in X$ and $\iota_{\mu} > 0$ is the intensity of $\mu$, i.e. the mean measure of a unit volume Borel set. More generally, one can consider variances of the form $\Var(S\chi_{rB})$ for $r > 0$, where $B \subset X$ is any bounded Borel set with positive measure and $rB$ is the $r$-neighbourhood of $B$ with respect to the metric on $X$. We will in this paper only consider metric balls $B_r$.

\subsection{Number variances and diffraction measures in $\R^n$}

A uniform upper bound on the number variance of a random measure $\mu$ on $\R^n$ is obtained from an application of the mean ergodic theorem, yielding
\begin{align*}
\limsup_{R \rightarrow +\infty} \frac{\Vmu(R)}{\Vol_{\R^n}(B_R)^2} = 0 \, . 
\end{align*}
We provide a general lower bound of the number variance for maximizing sequences of radii, extending a celebrated result of Beck \cite[Theorem 2A]{Beck1987IrregularitiesOD}.
\begin{theorem}
\label{Theorem1A}
Let $\mu$ be an invariant locally square-integrable random measure on $\R^n$. Then 
\begin{align}
\limsup_{R \rightarrow +\infty} \frac{\Vmu(R)}{\Vol_{n-1}(\partial B_R)} > 0 \, . 
\end{align}
\end{theorem}
\begin{remark}
The question of whether there is a uniform lower bound on this quotient for a given random measure $\mu$ is in general unclear. In Subsections \ref{Dynamics for the invariant random Z^n-lattice} and \ref{Proof of Theorem Theorem1B} however, we show that the randomly shifted standard lattice in $\R^5$ satisfies 
%
\begin{align}
\label{EqVanishingoftheZ5NumberVariance}
\liminf_{R \rightarrow +\infty} \frac{\mathrm{NV}_{\mu_{\Z^5}}(R)}{\Vol_{4}(\partial B_R)} = 0 \, . 
\end{align}
\end{remark}

The proof of Theorem \ref{Theorem1A} relies on a spectral formulation of the number variances. More precisely, the \emph{diffraction measure} $\hat{\eta}_{\mu}$ of an invariant locally square-integrable random measure $\mu$ is the unique positive Radon measure on $\R^n$ satisfying
\begin{align*}
\Covmu(Sf_1, Sf_2) = \int_{\R^n} \hat{f}_1(\xi) \overline{\hat{f}_2(\xi)} d\hat{\eta}_{\mu}(\xi)
\end{align*}
for all bounded measurable functions $f_1, f_2$ on $\R^n$ with bounded support, where $\hat{f}$ denotes the Fourier transform of $f$. In particular, one computes the number variance to be 
\begin{align*}
\Vmu(r) = (2\pi)^{n} r^n \int_{\R^n} J_{\frac{n}{2}}(r \norm{\xi})^2 \frac{d\hat{\eta}_{\mu}(\xi)}{\norm{\xi}^{n}} \, , 
\end{align*}
where $J_{n/2}$ denotes a Bessel function of the first kind. Classical asymptotic expansions of such Bessel functions allow for a proof of Theorem \ref{Theorem1A}. 

\subsection{Hyperuniformity and stealth in $\R^n$}

An invariant locally square-integrable random measure $\mu$ on $\R^n$ is \emph{geometrically hyperuniform} if 
\begin{align*}
\limsup_{R \rightarrow +\infty} \frac{\Vmu(R)}{\Vol_{\R^n}(B_R)} = 0 \, . 
\end{align*}
The volume term in the denominator coincides with the variance of the invariant Poisson point process on $\R^n$ with unit intensity, and one can interpret a hyperuniform random measure $\mu$ as having ''sub-Poissonian mass fluctuations''. If the above upper limit is infinite, the random measure is \emph{hyperfluctuating}. In \cite{Bjorklund2022HyperuniformityAN}, the first author and Hartnick prove that a random measure $\mu$ is geometrically hyperuniform if and only if it is \emph{spectrally hyperuniform}, meaning that
\begin{align*}
\limsup_{\varepsilon \rightarrow 0^+} \frac{\hat{\eta}_{\mu}(B_{\varepsilon}(0))}{\Vol_{\R^n}(B_{\varepsilon}(0))} = 0 \, . 
\end{align*}
The denominator here can be interpreted as the diffraction measure of the invariant Poisson point process in $\R^n$ with unit intensity. 

A particularly rigid class of hyperuniform invariant random measures $\mu$ are \emph{stealthy} random measures, meaning random measures for which the diffraction measure $\hat{\eta}_{\mu}$ vanishes in a neighbourhood of the origin in $\R^n$. Examples include invariant random shifts of lattices, and in upcoming work with A. Fish we will provide new families of stealthy random measures and point processes on $\R^n$.

\subsection{Number variances and diffraction measures for real hyperbolic spaces}

For $n$-dimensional real hyperbolic space $\H^n = \SO^{\circ}(1, n)/\SO(n)$, we consider locally square-integrable random measures $\mu$ that are invariant under the action of the orientation-preserving isometry group $\SO^{\circ}(1, n)$ in analogy with isotropic (i.e. translation and rotation invariant) random measures on $\R^n$.  

On $\H^n$, there is a well-behaved generalization of the Fourier transform, the \emph{spherical transform}, available for bi-$\SO(n)$-invariant functions $\varphi : \SO^{\circ}(1, n) \rightarrow \C$. It is given by
\begin{align*}
\hat{\varphi}(\lambda) = \int_{\SO^{\circ}(1, n)} \varphi(g) \omega_{\lambda}^{(n)}(g) dg \, , 
\end{align*}
where $dg$ denotes a fixed Haar measure and $\omega^{(n)}_{\lambda}$ are $\SO(n)$-\emph{spherical functions} for $\SO^{\circ}(1, n)$, that is, the bi-$\SO(n)$-invariant functions on $\SO^{\circ}(1, n)$ given by
\begin{align*}
 \omega_{\lambda}^{(n)}(g) = \frac{2^{\frac{n-1}{2}}\Gamma(\frac{n}{2})}{\sqrt{\pi} \, \Gamma(\frac{n - 1}{2})} \sinh(d(g.o, o))^{2 - n}\int_0^{d(g.o, o)} (\cosh(d(g.o, o)) - \cosh(s))^{\frac{n-3}{2}} \cos(\lambda s) ds \, ,  
\end{align*}  
where $d$ denotes the hyperbolic metric. With respect to the spherical transform, the diffraction measure of an invariant locally square-integrable random measure $\mu$ is a pair of measures defined by the equation
\begin{align*}
\Covmu(S\varphi_1, S\varphi_2) = \int_{0}^{\infty} \hat{\varphi}_1(\lambda) \overline{\hat{\varphi}_2(\lambda)} d\hat{\eta}_{\mu}^{(p)}(\lambda) + \int_{0}^{\frac{n-1}{2}} \hat{\varphi}_1(i\lambda) \overline{\hat{\varphi}_2(i\lambda)} d\hat{\eta}_{\mu}^{(c)}(\lambda)  \, . 
\end{align*}
%
We refer to the measures $\hat{\eta}_{\mu}^{(p)}$ on $(0, +\infty)$ and $\hat{\eta}_{\mu}^{(c)}$ on $i[0, \frac{n-1}{2})$ as the \emph{principal-} and \emph{complementary diffraction measure} respectively. We leave out $\lambda = 0$ corresponding to the \emph{Harish-Chandra $\Xi$-function} $\omega_0$ from the support of the principal diffraction measure for reasons that will become clear from our results. The existence and uniqueness of such diffraction measures can be seen as part of important work by Krein, Gelfand-Vilenkin and later by Bopp in \cite[Part II, Theorem 3]{Bopp}. 

\begin{remark}
Diffraction measures can more generally be defined for random measures on a large class of \emph{commutative spaces} $X = G/K$ with $K$ compact, which includes higher rank symmetric spaces, regular trees, Bruhat-Tits buildings and products of such. These are positive measures on the $K$-spherical unitary dual $\hat{G}^K$ of the pair $(G, K)$. We investigate the general framework and prove the existence and uniqueness of diffraction measures in an upcoming article.
\end{remark}

In terms of the diffraction measure, the number variance can be written as
\begin{align*}
\Vmu(r) = \frac{\pi^{n}}{\Gamma(\frac{n}{2} + 1)^2} \, \sinh(r)^{2n} \Big( \int_0^{\infty} |\omega^{(n + 2)}_{\lambda}(a_r)|^2 d\hat{\eta}_{\mu}^{(p)}(\lambda) + \int_0^{\frac{n-1}{2}} |\omega^{(n + 2)}_{i\lambda}(a_r)|^2 d\hat{\eta}_{\mu}^{(c)}(\lambda) \Big) \, . 
\end{align*}
In this picture, it is crucial to understand the asymptotic behaviour of the spherical functions in order to understand the asymptotics of the number variance. 

Also, the asymptotics of the number variance of a random measure $\mu$ depends significantly on whether the complementary diffraction measure is trivial or not. 

A first naive definition of geometric hyperuniformity of a random measure $\mu$ on $\H^n$ is to require that
\begin{align*}
\limsup_{R \rightarrow +\infty} \frac{\Vmu(R)}{\Vol_{\H^n}(B_R)} = 0 \, . 
\end{align*}
However, since $\SO^{\circ}(1, n)$ is non-amenable then $\Vol(\partial B_R) \asymp \Vol_{\H^n}(B_R)$ as $R \rightarrow +\infty$, so it is a priori unclear whether this definition admits examples or not. We prove that no random measures on $\H^n$ are geometrically hyperuniform in the following hyperbolic analogue of Theorem \ref{Theorem1A}.

\begin{theorem}
\label{Theorem2}
Let $\mu$ be a locally square-integrable invariant random measure on $\H^n$. Then 
\begin{align}
        \limsup_{R \rightarrow +\infty} \frac{\Vmu(R)}{\Vol_{\H^n}(B_R)} > 0 \, . 
\end{align}
Moreover,
\begin{enumerate}
\item if $\hat{\eta}_{\mu}^{(c)}(\{0\}) > 0$, then 
\begin{align}
        \liminf_{R \rightarrow +\infty} \frac{\Vmu(R)}{R^2 \Vol_{\H^n}(B_R)} > 0 \, . 
\end{align}
\item if $\hat{\eta}_{\mu}^{(c)}([\delta\frac{n-1}{2}, \frac{n-1}{2})) > 0$ for some $\delta \in (0, 1)$, then
\begin{align}
        \liminf_{R \rightarrow +\infty} \frac{\Vmu(R)}{\Vol_{\H^n}(B_R)^{1 + \delta}} > 0 \, . 
\end{align}
\end{enumerate}
\end{theorem}
For the special case of a randomly shifted lattice orbit in $\H^n$, Theorem \ref{Theorem2} implies that there is a sequence of radii such that the number variance grows at least as fast as the hyperbolic volume of large metric balls. Moreover, when the complementary diffraction measure is non-trivial, the random measure is hyperfluctuating in analogy with the Euclidean terminology.

There are related works to the above result: In \cite[Theorem 2]{Hill2005TheVO}, Hill and Parnovski derive explicit asymptotics of the number variance for randomly shifted lattice orbits in $\H^n$. Theorem \ref{Theorem2} implies that the error term in their result is sharp. In \cite[Theorem 2]{MagyarAkosDiscrepancyOnHyperbolicSpaces}, Magyar provides a positive lower bound on the $L^{\infty}$-norm of the discrepancy of infinite point sets over balls in $\H^n$ that is proportional to the volume of the ball. For number variances (the $L^2$-norm of the discrepancy), we expect that there are examples of translation bounded point processes $\mu$ in $\H^n$ for which
\begin{align*}
\liminf_{R \rightarrow +\infty} \frac{\Vmu(R)}{\Vol_{\H^n}(B_R)} = 0 \, . 
\end{align*} 

\subsection{Hyperuniformity in real hyperbolic space}

In light of Theorem \ref{Theorem2}, there is no canonical ''geometric'' definition of hyperuniformity of a random measure $\mu$ on $\H^n$. We instead formulate a notion of \emph{spectral hyperuniformity} in terms of the diffraction measure of $\mu$ that is distinctly different from the geometric one. 

\begin{definition}
An invariant locally square-integrable random measure on $\H^n$ is \emph{spectrally hyperuniform} if $\hat{\eta}_{\mu}^{(c)} = 0$ and 
\begin{align*}
\limsup_{\varepsilon \rightarrow 0^+} \frac{\hat{\eta}_{\mu}^{(p)}((0, \varepsilon])}{\varepsilon^3} = 0 \, . 
\end{align*}
\end{definition}
The denominator $\varepsilon^3$ is asymptotic to the diffraction measure of the interval $(0, \varepsilon]$ for the unit intensity invariant Poisson point process as $\varepsilon \rightarrow 0^+$, so this definition of hyperuniformity is in that way consistent with spectral hyperuniformity on $\R^n$. It is worth noting that, with respect to this definition, it is no longer the case that all invariant randomly shifted lattice orbits are spectrally hyperuniform as there are lattices that admit complementary spectrum, even co-compact ones, see for example \cite[Section 1.2, Eq.B]{Jenni1984UeberDE}. With this notion of hyperuniformity in mind, we define what it means for a random measure to be stealthy.
\begin{definition}
An invariant locally square-integrable random measure on $\H^n$ is \emph{stealthy} if $\hat{\eta}_{\mu}^{(c)} = 0$ and there is a $\lambda_o > 0$ such that $\hat{\eta}_{\mu}^{(p)}((0, \lambda_o)) = 0$. 
\end{definition}
Stealthy random measures are spectrally hyperuniform. In \cite[Section 1.2, Eq.B]{Jenni1984UeberDE}, Jenni constructs a lattice $\Gamma < \SL_2(\R) \cong \SO^{\circ}(1, 2)$ that does not admit complementary spectrum nor principal spectrum around $\lambda = 0$ in the sense that the smallest non-zero eigenvalue of the Laplace operator on $\Gamma \backslash \H^2$ is strictly larger than $1/4$. The eigenvalue $1/4$ corresponds to $\lambda = 0$, so the resulting invariant random lattice orbit of $\Gamma$ in $\H^2$ is stealthy, and in particular spectrally hyperuniform.

The lack of other examples of spectrally hyperuniform point processes and random measures in $\H^n$ is however quite severe. In an upcoming article we will prove that determinantal point processes in $\H^n$ (and more generally in any commutative space associated with a non-amenable Gelfand pair) are not spectrally hyperuniform, in contrast to the Euclidean case where determinantal point processes associated to kernels which define orthogonal projections in $L^2(\R^n)$ are always hyperuniform. Here we show that locally square-integrable i.i.d. perturbations of an invariant random lattice orbit in $\H^n$ are not hyperuniform. We note the apparent contrast to the setting in $\R^n$, where i.i.d. perturbations of an invariant random lattice are always locally square-integrable and always preserves hyperuniformity, see \cite[Appendix B]{KlattPerturbedLattices} and Proposition \ref{PropEuclideanLatticePerturbationsPreserveHyperuniformity}. The following Proposition is proved in Section \ref{Random perturbed lattice orbits in hyperbolic spaces}. 
\begin{proposition}
\label{PropHyperbolicLatticePerturbationDestroysHyperuniformity}
Let $\Gamma < \SO^{\circ}(1, n)$ be a lattice and $d\nu(g) = \beta(g) dg$ for some positive measurable bi-$\SO(n)$-invariant function $\beta$ with $\int \beta(g) dg = 1$ such that
\begin{align*}
\beta(g) \leq  C \e^{-2(n - 1 + \varepsilon)d(g.o, o)}
\end{align*}
for all $g \in G$ and some $C > 0, \varepsilon > 0$. Then the invariant random lattice orbit of $\Gamma$ in $\H^n$ with independent $\nu$-distributed perturbations is locally square-integrable and not spectrally hyperuniform. 
\end{proposition}

\begin{remark}[Heat kernel hyperuniformity in Euclidean and hyperbolic spaces]
In future work, we will prove that spectral hyperuniformity is equivalent to what we refer to as \emph{heat kernel hyperuniformity}. Denote the heat kernel by $h_{\tau}(x)$, $\tau > 0$, the fundamental solution of the heat operator $\partial_{\tau} - \Delta$ in the respective geometries. In $\R^n$, we define $\mu$ to be heat kernel hyperuniform if
\begin{align*}
\limsup_{\tau \rightarrow +\infty} \,  \tau^{\frac{n}{2}} \, \Varmu(Sh_{\tau}) = 0 \, ,  
\end{align*}
and similarly in $\H^n$, if
\begin{align*}
\limsup_{\tau \rightarrow +\infty} \,  \tau^{\frac{3}{2}} \e^{\frac{(n - 1)^2}{2} \tau} \,  \Varmu(Sh_{\tau}) = 0 \, . 
\end{align*}
One perk of this formulation is that the smoothness of $h_{\tau}$ allows us to extend the notion of hyperuniformity to invariant random \emph{distributions} on these spaces as well, for which we expect that one can construct hyperuniform and stealthy examples. Moreover, by uniform estimates of the heat kernel due to Anker and Ostellari in \cite{Anker2003TheHK}, one should be able to formulate heat kernel hyperunifomity for a large class of non-compact symmetric spaces. 
\end{remark}

\textbf{Acknowledgement}: This paper is part of the second author's doctoral thesis at the University of Gothenburg under the supervision of the first author. The first author was supported by the grants 11253320 and 2023-03803 from the
Swedish Research Council. We thank Tobias Hartnick and Günter Last for providing valuable insights into the making of the paper. The second author is grateful to Tobias Hartnick for many useful discussions and hospitality during several visits to the Karlsruhe Institute of Technology, and also to Morten Risager for providing valuable information regarding number variances of lattices during a stay at the Mittag-Leffler Analytic Number Theory Program, spring 2024.

%% file: Setup.tex
We introduce invariant locally square-integrable random measures and their autocorrelation measures. The constructions work for a large family of proper homogeneous metric spaces, but we will only apply them to Euclidean and hyperbolic spaces in the subsequent sections. 

In Subsection \ref{Homogeneous metric spaces} we set up the metric spaces over which we will consider random measures. We proceed by defining invariant locally square-integrable random measures, related objects and provide examples in Subsection \ref{Random measures}. In Subsection \ref{Autocorrelation measures}, we define the autocorrelation measure of an invariant locally square-integrable random measure and compute it for the examples given in the previous subsection. 

\subsection{Homogeneous metric spaces}
\label{Homogeneous metric spaces}

Let $(X, d)$ be a proper non-compact metric space and fix a point $o \in X$. Suppose that there is a closed and unimodular subgroup $G$ of the isometry group $\Isom(X, d)$ whose action on $X$ is transitive. The action of $g \in G$ on $x \in X$ will be denoted by $g.x$. If we denote the $G$-stabilizer of the point $o$ by $K$ then the orbit map $\pi : G \rightarrow X$, $g \mapsto g.o$, descends to a bijection $G/K \rightarrow X$. We will assume that $K$ is \emph{compact}. 

We will frequently make use of the fact that there is a Borel measurable \emph{section} $\varsigma : X \rightarrow G$ of the orbit map $\pi$, that is, a Borel measurable map such that $\varsigma(x).o = x$ for all $x \in X$ and such that $\varsigma(B) \subset G$ is pre-compact for every bounded Borel set $B \subset X$, see \cite[Lemma 1.1]{MackeyInducedRepsI}. One shows that such a section satisfies the $\varsigma(g.o) \in gK$ for all $g \in G$. For Euclidean and real hyperbolic spaces, we will construct \emph{continuous} sections. 

Fix a Haar measure $m_G$ on $G$ and define a positive $G$-invariant measure $m_X$ on $X$ as the push-forward of $m_G$ along $\pi$. Explicitly,
\begin{align}
\label{ReferenceMeasure}
m_X(B) = m_G\big( \big\{ g \in G : g.o \in B \big\} \big)
\end{align}
for all bounded Borel sets $B \subset X$. Since the orbit map $\pi : G \rightarrow X$ is proper, $m_X$ defines a positive Radon measure on $X$. Also, since the Haar measure $m_G$ is unique up to scaling, then the measure $m_X$ is unique up to scaling. We will moreover denote by $m_K$ the unique Haar probability measure on $K$.

On the level of functions, if $M$ is a measurable space we denote by $\Borelinfty(M)$ the vector space of bounded measurable complex valued functions on $M$. With regards to the proper metric space $X$, we denote by $\Borelbndinfty(X) \subset \Borelinfty(X)$ the subspace of functions which vanish outside of a bounded set, or equivalently, vanish outside a compact subset. The Borel section $\varsigma$ then establishes a bijection between $f \in \Borelbndinfty(X)$ and right-$K$-invariant functions $\varphi_f \in \Borelbndinfty(G)$ by $f(x) = \varphi_f(\varsigma(x))$ for all $x \in X$. The vector space $\Borelbndinfty(G)$ defines an involutive algebra over the complex numbers under the convolution operation
\begin{align*}
(\varphi * \psi)(g) = \int_G \varphi(gh^{-1}) \psi(h) dm_G(h) \, , \quad \varphi, \psi \in \Borelbndinfty(G) \, , 
\end{align*}
and the involution $\varphi^*(g) = \overline{\varphi(g^{-1})}$. If $\varphi, \psi \in \Borelbndinfty(G)$ are right-$K$-invariant then $\varphi^*$ is left-$K$-invariant, and one checks that $\varphi^* * \psi$ is bi-$K$-invariant, so that the subspace $\Borelbndinfty(G, K) \subset \Borelbndinfty(G)$ of bi-$K$-invariant functions defines an involutive subalgebra. We can moreover identify $\Borelbndinfty(G, K)$ with the subspace $\Borelbndinfty(X)^K \subset \Borelbndinfty(X)$ of left-$K$-invariant functions on $X$ using the section $\varsigma$ as mentioned. We say that a function on $X$ is \emph{radial} if it is left-$K$-invariant. More generally, if $\nu, \eta$ are Radon measures on $G$ with $\nu$ finite, their convolution is given by
$$ (\nu * \eta)(Q) = \int_G \int_G \chi_Q(gh) d\nu(g) d\eta(h) $$
for Borel sets $Q \subset G$.
\subsection{Random measures}
\label{Random measures}
Let $\Radonplus(X)$ denote the set of positive locally finite Borel measures on $X$, endowed with the smallest $\sigma$-algebra such that for every $f \in \Borelbndinfty(X)$, the \emph{linear statistic} $S f : \Radonplus(X) \rightarrow \C$ given by 
\begin{align*}
S f (p) = p(f) = \int_X f(x) dp(x) \, , \quad p \in \Radonplus(X) 
\end{align*}
is measurable. Then $\Radonplus(X)$ is a Polish space when endowed with the vague topology with respect to the subspace $C_c(X) \subset \Borelbndinfty(X)$ of compactly supported continuous complex valued functions on $X$. By a \emph{random measure on $X$} we will mean a probability measure $\mu$ on $\Radonplus(X)$. We say that $\mu$ is a \emph{point process} if it is supported on the subspace 
$$ \Radonplus^{\mathrm{LF}}(X) = \Big\{ p \in \Radonplus(X) : \supp(p) \mbox{ is locally finite} \Big\} $$
of measures with locally finite support. A point process is \emph{simple} if it is supported on
$$ \Radonplus^{\mathrm{LFS}}(X) = \Big\{ \delta_{\Xi} = \sum_{x \in \Xi} \delta_x \in \Radonplus(X) : \Xi \subset X \mbox{ is locally finite} \Big\} $$ 
of locally finite measures with unit masses. The action of $G$ on $X$ lifts to an action of $G$ on $\Radonplus(X)$ by push-forward and leaves the mentioned subspaces invariant. Moreover, an additional lifting by push-forward yields an action of $G$ on the space of random measures on $X$. A random measure $\mu$ on $X$ is 
\begin{itemize}
    \item \emph{invariant} if $g_*\mu = \mu$ for all $g \in G$.
    \item \emph{locally $k$-integrable}, $k \in \N$, if for every bounded Borel set $B \subset X$,
    $$ \int_{\Radonplus(X)} p(B)^k d\mu(p) < + \infty \, . $$
    In terms of linear statistics, $S\chi_B \in L^k(\Radonplus(X), \mu)$ for all bounded Borel $B \subset X$. Note that a locally $k$-integrable random measure $\mu$ is locally $k'$-integrable for every $1 \leq k' \leq k$. 
\end{itemize}
We will from now on assume that all random measures are invariant and locally square-integrable, and we will simply refer to them as random measures, if not stated otherwise.

By local square-integrability, the linear statistics $Sf$, $f \in \Borelbndinfty(X)$, have a well-defined \emph{$\mu$-expectation}
$$ \bE_{\mu}(Sf) = \int_{\Radonplus(X)} p(f) d\mu(p) , $$
\emph{$\mu$-covariance},
\begin{align*}
\Covmu(Sf_1, Sf_2) &= \bE_{\mu}((Sf_1 - \bE_{\mu}(Sf_2))\overline{(Sf_2 - \Emu(Sf_2))}) \\
&= \Emu(Sf_1 \overline{Sf_2}) - \Emu(Sf_1) \overline{\Emu(Sf_2)} \, ,
\end{align*}   
and \emph{$\mu$-variance}
$$ \Var_{\mu}(Sf) = \Covmu(Sf, Sf) = \bE_{\mu}(|Sf - \bE_{\mu}(Sf)|^2) = \bE_{\mu}(|Sf|^2) - |\bE_{\mu}(Sf)|^2 \, . $$
By invariance of $\mu$, the linear functional $f \mapsto \bE_{\mu}(Sf)$ on $\Borelbndinfty(X)$ is $G$-invariant, and since $m_X$ is the only $G$-invariant positive Radon measure on $X$ up to scaling there is a unique scalar $\iota_{\mu} > 0$, the \emph{intensity} of $\mu$, such that 
$$\bE_{\mu}(Sf) = \iota_{\mu} m_X(f) \, . $$
The $\mu$-variance of $Sf$ can thus be written as 
$$\Var_{\mu}(Sf) = \bE_{\mu}(|Sf|^2) - \iota_{\mu}^2 |m_X(f)|^2  \, . $$
In the case that $f = \chi_{B_r}$ is the indicator function of the metric ball $B_r$ centered at $o \in X$, we arrive at the number variance. 
\begin{definition}
Let $\mu$ be an invariant locally square-integrable random measure on $X$. The \emph{number variance} of $\mu$ is the function $\Vmu : \R_{\geq 0} \rightarrow \R_{\geq 0}$ given by
\begin{align*}
\NVmu(r) = \Varmu(S\chi_{B_r}) \, . 
\end{align*}
\end{definition}
\begin{example}[Invariant Poisson point processes]
The \emph{$m_X$-Poisson point process} is the random measure $\mu$ on $X$ satisfying 
\begin{enumerate}
    \item for every bounded Borel set $B \subset X$, the linear statistic $S\chi_B$ is Poisson distributed with intensity $m_X(B)$. In other words, for every non-negative integer $j$,
    \begin{align*}
        \mu\Big(\Big\{ p \in \Radonplus(X) : p(B) = j \Big\}\Big) = \e^{-m_X(B)} \frac{m_X(B)^j}{j!} \, .
    \end{align*}
    \item for every finite collection $B_1, ..., B_N$ of disjoint bounded Borel sets in $X$, the linear statistics $S\chi_{B_1}, ..., S\chi_{B_N}$ are $\mu$-independent.
\end{enumerate}
One shows that the two conditions above ensure the existence and uniqueness of $\mu$ up to equivalence and that $\mu$ defines a locally square-integrable simple point process \cite[Theorem 3.6 + Prop. 3.2]{Penrose2017LecturesOT}. Moreover, since $m_X$ is $G$-invariant, $\mu$ is invariant. We will write $\mu_{\Poi}$ for the $m_X$-Poisson point process and $\Var_{\Poi}$ for the associated variance.
\end{example}

\begin{example}[Random lattice orbits]
\label{ExampleRandomLatticeOrbits}
Let $\Gamma < G$ be a lattice, i.e. a discrete subgroup such that there is a (necessarily unique) $G$-invariant probability measure $m_{G / \Gamma}$ on $G / \Gamma$. Since $\Gamma$ is a subgroup it is also uniformly discrete in $G$ in the sense that there is a (symmetric) open neighbourhood $U \subset G$ of the identity such that $|\Gamma \cap gU| \leq 1$ for all $g \in G$. In particular, if $Q \subset G$ is compact then a standard covering argument for $Q$ using left translates of $U$ yields 
$$ \sup_{g \in G} |g \Gamma \cap Q| = \sup_{g \in G} |\Gamma \cap g Q| < + \infty \, .   $$
We let $\Gamma_o = \Stab_{\Gamma}(o) = \Gamma \cap K$ be the stabilizer of $o \in X$ under the action of $\Gamma$, which is finite since $\Gamma$ is discrete and $K$ compact. The map $G / \Gamma \rightarrow \Radonplus^{\mathrm{LFS}}(X)$ sending a coset $g \Gamma$ to the measure 
\begin{align*} 
\delta_{g \Gamma.o} = \sum_{\gamma\Gamma_{o} \in \Gamma/\Gamma_{o}} \delta_{g \gamma.o} = \frac{1}{|\Gamma_{o}|} \sum_{\gamma \in \Gamma} \delta_{g \gamma.o}
\end{align*}
induces a simple point process $\mu_{\Gamma}$ on $X$ as the pushforward of $m_{G / \Gamma}$ via said map, in other words
\begin{align*}
\int_{\Radonplus(X)} \Psi(p) d\mu_{\Gamma}(p) = \int_{G / \Gamma} \Psi(\delta_{g \Gamma.o}) dm_{G / \Gamma}(g \Gamma) 
\end{align*}
for all $\Psi \in \Borelinfty(\Radonplus(X))$. This point process will be invariant as $m_{G / \Gamma}$ is a $G$-invariant measure and it is locally $k$-integrable for every $k \in \N$ as
\begin{align*}
\int_{\Radonplus(X)} p(B)^k d\mu_{\Gamma}(p) = \int_{G / \Gamma} |g \Gamma.o \cap B|^k dm_{G / \Gamma}(g \Gamma) \leq \sup_{g \in G} |g \Gamma.o \cap B|^k < + \infty 
\end{align*}
by uniform discreteness of $\Gamma$. We say that $\mu_{\Gamma}$ is a \emph{random lattice orbit} in $X$. In the case that $X = G$ is a proper metric group then we say that $\mu_{\Gamma}$ is a \emph{random lattice}. 

The description of the probability measure $m_{G / \Gamma}$ can be made explicit. A subset $F_{\Gamma} \subset G$ is a (left) \emph{fundamental domain} for $\Gamma$ if it is measurable and 
\begin{align*}
G = \bigsqcup_{\gamma \in \Gamma} F_{\Gamma} \gamma \, . 
\end{align*}
If $F_{\Gamma}$ is pre-compact then we say that $\Gamma$ is \emph{cocompact}. It follows from uniform discreteness of $\Gamma$ that $Q \cap F_{\Gamma}\gamma = \varnothing$ for all but finitely many $\gamma \in \Gamma$ whenever $Q \subset G$ is compact. Given a fundamental domain $F_{\Gamma}$ of $\Gamma$ in $G$ we define the \emph{covolume} of $\Gamma$ to be 
$$\covol(\Gamma) = m_G(F_{\Gamma}) \, . $$
The $G$-invariant probability measure on $G/\Gamma$ is then given by 
$$ m_{G / \Gamma} = \covol(\Gamma)^{-1}(\pi_{\Gamma})_*m_G|_{F_{\Gamma}} \, , $$
where $\pi_{\Gamma} : G \rightarrow G/\Gamma$ is the canonical quotient map. The above definitions can be shown to be independent of the choice of fundamental domain. 
\end{example}
Next we construct new invariant point processes by applying random i.i.d. perturbations to each point in the lattice orbit. 
\begin{example}[Random perturbed lattice orbits]
\label{ExamplePerturbedLatticePointProcess}
Let $\Gamma < G$ be a lattice and let $\nu \in \Prob(G)$ a probability measure.

Consider the space $Z = G^{\Gamma}$ of $\Gamma$-indexed sequences $z = (z_{\gamma})_{\gamma \in \Gamma}$ in $G$, endowed with the product topology. The group $\Gamma$ acts on $Z$ by right translation, $\gamma.(z_{\gamma'})_{\gamma' \in \Gamma} = (z_{\gamma'\gamma})_{\gamma \in \Gamma}$ and it also acts on $G$ canonically from the right by $g \mapsto g\gamma^{-1}$ for each $\gamma \in \Gamma$. The action of $G \times \Gamma$ on $G \times Z$ given by $(g, \gamma).(z, h) = (g h \gamma^{-1}, \gamma.z)$ for $g, h \in G$, $\gamma \in \Gamma$ and $z \in Z$ is well-defined and continuous, and we let $Y = (G \times Z)/\Gamma$ be the space of $\Gamma$-orbits. Explicitly, each $y \in Y$ is of the form 
\begin{align*}
y = (g, z)\Gamma = \big\{ (g \gamma, \gamma.z) \in G \times G^{\Gamma} : \gamma \in \Gamma \big\} 
\end{align*}
for some $g \in G$, $z \in G^{\Gamma}$. On $Y$ we define a $G$-invariant positive probability measure $m_{Y}$ by 
\begin{align*}
m_Y(F) = \int_{G / \Gamma} \Big( \int_{Z} F((g, z)\Gamma) d\nu^{\otimes \Gamma}(z) \Big) dm_{G / \Gamma}(g \Gamma)
\end{align*}
for all $F \in \Borelinfty(Y)$, where $\nu^{\otimes \Gamma}$ is the product probability measure on $Z$.

To construct a $\nu$-perturbed random $\Gamma$-orbit in $X$ we first define a $\nu$-perturbed random $\Gamma$-orbit in $G$ by considering the measures
\begin{align*}
\delta_{(g, z)\Gamma} = \sum_{\gamma \in \Gamma} \delta_{g \gamma z_{\gamma}} \, . 
\end{align*}
\begin{lemma}
\label{LemmaSimplePerturbedLatticeInGroups}
Let $\Gamma < G$ be a lattice and $\nu \in \Prob(G)$. Then
\begin{enumerate}
    \item $\delta_{(g, z)\Gamma}$ is locally finite $m_{G / \Gamma} \otimes \nu^{\otimes \Gamma}$-almost everywhere. 
    \item If $\nu$ is absolutely continuous with respect to the Haar measure $m_G$, then $\delta_{(g, z)\Gamma}$ is simple $m_{G / \Gamma} \otimes \nu^{\otimes \Gamma}$-almost everywhere.
\end{enumerate}
\end{lemma}
\begin{proof}
For (1), let $Q \subset G$ be compact. We want to show that 
\begin{align*}
\nu^{\otimes \Gamma}\Big( \Big\{ z \in Z :  \gamma z_{\gamma} \in Q \mbox{ for infinitely many } \gamma \in \Gamma \Big\} \Big) = 0 \, . 
\end{align*}
By the Borel-Cantelli Lemma, it suffices to show that
\begin{align*}
\sum_{\gamma \in \Gamma} \nu^{\otimes \Gamma}\Big( \Big\{ z \in Z : \gamma z_{\gamma} \in Q \Big\} \Big) = \sum_{\gamma \in \Gamma} \nu(\gamma^{-1} Q) < +\infty \, ,
\end{align*}
which follows from a standard covering argument by picking $\Gamma$-injective open sets $U_1, ..., U_N \subset G$ that cover $Q$ and using that $\nu$ is a probability measure.

For (2) it is enough to show that the event $\gamma_1 z_{\gamma_1} = \gamma_2 z_{\gamma_2} $ is a null set in $Z$ with respect to $\nu^{\otimes \Gamma}$ for all $\gamma_1 \neq \gamma_2$. Explicitly, if we set $B_{\gamma_1, \gamma_2} = \{z \in Z : \gamma_1 z_{\gamma_1}  = \gamma_2 z_{\gamma_2}  \}$ then we aim to show that $\nu^{\otimes \Gamma}(B_{\gamma_1, \gamma_2}) = 0$ for $\gamma_1 \neq \gamma_2$ in $\Gamma$. Since $\Gamma$ is countable it will follow by $\sigma$-subadditivity that
\begin{align*}
\nu^{\otimes \Gamma}\Big(\Big\{ z \in Z : \gamma_1 z_{\gamma_1} = \gamma_2 z_{\gamma_2}  \,\, \mbox{for some} \, \gamma_1 \neq \gamma_2 \Big\}\Big)  \leq \sum_{\gamma_1 \neq \gamma_2} \nu^{\otimes \Gamma}(B_{\gamma_1, \gamma_2}) = 0\, . 
\end{align*}
Using that $\nu^{\otimes \Gamma}$ is a product measure, we get 
\begin{align*}
\nu^{\otimes \Gamma}(B_{\gamma_1, \gamma_2}) &= (\nu \otimes \nu)\Big( \Big\{ (z_1, z_2) \in G \times G :  \gamma_1 z_{1} = \gamma_2 z_{2}   \Big\} \Big) \\
&= (\nu \otimes \nu)\Big( \Big\{ (z_1, \gamma_2^{-1}\gamma_1 z_1) \in G \times G :  z_1 \in G  \Big\} \Big) \, . 
\end{align*}
By Fubini's Theorem, 
\begin{align*}
\nu^{\otimes \Gamma}(B_{\gamma_1, \gamma_2}) &= \int_G \nu(\{ \gamma_2^{-1}\gamma_1 z_1 \}) d\nu(z_1) \, , 
\end{align*}
so if $\nu$ is absolutely continuous with respect to the Haar measure $m_G$ then $\nu(\{ \gamma_2^{-1}\gamma_1 z_1 \}) = 0$ for all $z_1 \in G$ and $\nu^{\otimes \Gamma}(B_{\gamma_1, \gamma_2}) = 0$.
\end{proof}
We now assume that $\nu \in \Prob(G)$ is absolutely continuous with respect to the Haar measure $m_G$. By Lemma \ref{LemmaSimplePerturbedLatticeInGroups} the invariant random measure $\tilde{\mu}_{\Gamma, \nu}$ obtained as the push-forward of $m_{Y}$ via the map $(g, z)\Gamma \mapsto \delta_{(g, z)\Gamma}$ is an invariant simple point process on $G$, explicitly given by
\begin{align*}
\int_{\Radonplus(G)} \Psi(q) d\tilde{\mu}_{\Gamma, \nu}(q) = \int_{G/\Gamma} \Big( \int_{Z} \Psi(\delta_{(g, z)\Gamma}) d\nu^{\otimes \Gamma}(z) \Big) dm_{G / \Gamma}(g\Gamma) 
\end{align*}
for all $\Psi \in \Borelinfty(\Radonplus(G))$. In order to construct a locally finite simple $\nu$-perturbed random $\Gamma$-orbit in $X$ we consider the measures 
$$ p_{(g, z)\Gamma} = \pi_*\delta_{(g, z)\Gamma} = \sum_{\gamma \in \Gamma} \delta_{g\gamma z_{\gamma}.o} \, ,  $$
where $\pi : G \rightarrow X$ is the orbit map $g \mapsto g.o$. The following Proposition provides conditions for when these measures are almost everywhere locally finite and simple with respect to $m_{G/\Gamma} \otimes \nu^{\otimes \Gamma}$.

\begin{proposition}
\label{PropSimplicityofRandomPerturbedLatticeOrbits}
Suppose that $\nu \in \Prob(G)$ is right-$K$-invariant and absolutely continuous with respect to the Haar measure $m_G$. Moreover, assume that $m_G(K) = 0$. Then the measures $p_{(g, z)\Gamma}$ are locally finite and simple $m_{G/\Gamma} \otimes \nu^{\otimes \Gamma}$-almost everywhere.
\end{proposition}
Under these assumptions we get an invariant simple point process in $X$ as the push-forward of $\tilde{\mu}_{\Gamma, \nu}$ along the map $\Radonplus(G) \ni q \mapsto \pi_*q \in \Radonplus(X)$.
\begin{definition}[Random perturbed lattice orbit]
Let $\Gamma < G$ be a lattice and $\nu \in \Prob(G)$ a right-$K$-invariant probability measure absolutely continuous with respect to the Haar measure $m_G$. The \emph{$\nu$-perturbed $\Gamma$-orbit} in $X = G/K$ is the simple point process $\mu_{\Gamma, \nu}$ given by
\begin{align*}
\int_{\Radonplus(X)} \Psi(p) d\mu_{\Gamma, \nu}(p) = \int_{G/\Gamma} \Big( \int_{Z} \Psi(p_{(g, z)\Gamma}) d\nu^{\otimes \Gamma}(z) \Big) dm_{G / \Gamma}(g\Gamma) 
\end{align*}
for all $\Psi \in \Borelinfty(\Radonplus(X))$.
\end{definition}
\begin{remark}
If $K$ is compact and open in $G$, for example when $X$ is the homogeneous $p$-regular tree and $G = \mathrm{PGL}_2(\Q_p)$, $K = \mathrm{PGL}_2(\Z_p)$, then $m_G(K) > 0$. In this case there are lattices $\Gamma$ and probability measures $\nu$ absolutely continuous with respect to the Haar measure on $G$ for which the associated random perturbed lattice orbit is not a simple point process.
\end{remark}
\begin{proof}[Proof of Proposition \ref{PropSimplicityofRandomPerturbedLatticeOrbits}]
Similarly to the proof of item (2) in Lemma \ref{LemmaSimplePerturbedLatticeInGroups}, it is enough to show that $\nu^{\otimes \Gamma}(B_{\gamma_1, \gamma_2}) = 0$, where
\begin{align*}
B_{\gamma_1, \gamma_2} &= \Big\{ z \in Z : \gamma_1 z_{\gamma_1} . o = \gamma_2 z_{\gamma_2} . o  \Big\} = \Big\{ z \in Z :  \gamma_2 z_{\gamma_2} \in \gamma_1 z_{\gamma_1} K \Big\}  \, . 
\end{align*}
Using that $\nu^{\otimes \Gamma}$ is a product measure, we get 
\begin{align*}
\nu^{\otimes \Gamma}(B_{\gamma_1, \gamma_2}) &= (\nu \otimes \nu)\Big( \Big\{ (z_1, z_2) \in G \times G :  \gamma_2 z_{2} \in \gamma_1 z_{1} K   \Big\} \Big) \\
&= (\nu \otimes \nu)\Big( \Big\{ (z_1, \gamma_2^{-1}\gamma_1 z_1 k) \in G \times G :  z_1 \in G \, , \, k \in K  \Big\} \Big) \, . 
\end{align*}
By Fubini's Theorem, 
\begin{align*}
\nu^{\otimes \Gamma}(B_{\gamma_1, \gamma_2}) &= \int_G \nu(\gamma_2^{-1}\gamma_1 z_1 K) d\nu(z_1) \, . 
\end{align*}
Now, if $\nu$ is absolutely continuous with respect to the Haar measure $m_G$ then $\nu(\gamma_2^{-1}\gamma_1 z_1 K)$ vanishes for all $z_1 \in G$ if and only if $m_G(K) = 0$, so $\nu^{\otimes \Gamma}(B_{\gamma_1, \gamma_2}) = 0$.
\end{proof}
Next we address local square-integrability of $\mu_{\Gamma, \nu}$, which in the general setting requires the following conditions on either $\Gamma$ or $\nu$. 
\begin{lemma}
\label{LemmaPerturbedLatticeisLocallySquareIntegrable}
Let $\Gamma < G$ be a lattice with fundamental domain $F_{\Gamma}$ and $\nu$ be a right-$K$-invariant probability measure on $G$. If $\Gamma$ is cocompact or the support $\supp(\nu)$ of $\nu$ is compact then the random perturbed lattice orbit $\mu_{\Gamma, \nu}$ on $X$ is locally square-integrable.  
\end{lemma}
\begin{proof}
Let $B \subset \R^n$ be a bounded Borel set. Then by definition of $\mu_{\Gamma, \nu}$,
\begin{align*}
\bE_{\mu_{\Gamma, \nu}}(|S\chi_B|^2) = \int_{G / \Gamma} \Big( \sum_{\gamma_1, \gamma_2 \in \Gamma} \int_{Z} &\chi_B (g \gamma_1 z_{\gamma_1}.o) \chi_B(g \gamma_2 z_{\gamma_2}.o) d\nu^{\otimes \Gamma}(z) \Big) dm_{G / \Gamma}(g\Gamma) \, . 
\end{align*}
Note how
\begin{align*}
\int_{Z} \chi_B (g \gamma_1 z_{\gamma_1}.o) &\chi_B(g \gamma_2 z_{\gamma_2}.o) d\nu^{\otimes \Gamma}(z) 
%
=
\begin{dcases}
\nu(\gamma^{-1} g^{-1} Q) &\mbox{ if } \gamma_1 = \gamma_2 = \gamma \\
\nu(\gamma_1^{-1} g^{-1} Q)\nu(\gamma_2^{-1} g^{-1} Q) &\mbox{ if } \gamma_1 \neq \gamma_2  
\end{dcases}
\end{align*}
where $Q = \pi^{-1}(B) = \varsigma(B)K$ is measurable and pre-compact. Splitting into these two cases we have that
\begin{align*}
\bE_{\mu_{\Gamma, \nu}}(|S\chi_B|^2) &= \int_{G / \Gamma} \Big( \sum_{\gamma \in \Gamma} \nu(\gamma^{-1} g^{-1} Q) \Big) dm_{G / \Gamma}(g\Gamma) \\
&+ \int_{G / \Gamma} \Big( \sum_{\gamma_1 \neq \gamma_2} \nu(\gamma_1^{-1} g^{-1} Q) \nu(\gamma_2^{-1} g^{-1} Q) \Big) dm_{G / \Gamma}(g\Gamma) \, . 
\end{align*}
The first term can be computed using Fubini and the $G$-invariance of $m_{G / \Gamma}$ as
\begin{align*}
\int_{G / \Gamma} \Big( \sum_{\gamma \in \Gamma} \nu(\gamma^{-1} g^{-1} Q) \Big) dm_{G / \Gamma}(g\Gamma) 
= \frac{m_G(Q)}{\covol(\Gamma)} < +\infty \, . 
\end{align*}
We postpone the details of this computation to Example \ref{ExamplePerturbedLatticeAutocorrelation}.

For the second term, it suffices to show that
\begin{align}
\label{EqLemma2point5SecondTerm}
I_{\Gamma, \nu} :=  \int_{G / \Gamma} \Big( \sum_{\gamma \in \Gamma} \nu(\gamma^{-1} g^{-1} Q) \Big)^2 dm_{G / \Gamma}(g\Gamma) < +\infty \, . 
\end{align}
First, if $\Gamma$ is cocompact with fundamental domain $F_{\Gamma}$ then we use Fubini to write
\begin{align*}
I_{\Gamma, \nu} = \frac{1}{\covol(\Gamma)} \sum_{\gamma_1, \gamma_2} \int_{G} \int_G \Big( \int_{F_{\Gamma}}\chi_Q (g \gamma_1 h_1) \chi_Q(g \gamma_2 h_2) dm_G(g) \Big)  d\nu(h_1) d\nu(h_2) \, . 
\end{align*}
If we let $Q' = F_{\Gamma}^{-1} Q$, then $Q'$ is measurable and pre-compact in $G$ and $\chi_Q(g \gamma h) \leq \chi_{Q'}(\gamma h)$ for all $g \in F_{\Gamma}, h \in G, \gamma \in \Gamma$, so
\begin{align}
\label{EqLemma2point5IndicatorIntegralBound}
\frac{1}{\covol(\Gamma)} \int_{F_{\Gamma}}\chi_Q (g \gamma_1 h_1) \chi_Q(g\gamma_2 h_2) dm_G(g) \leq \chi_{Q'}(\gamma_1 h_1) \chi_{Q'}(\gamma_2 h_2) \, . 
\end{align}
Thus the expression in Equation \ref{EqLemma2point5SecondTerm} is bounded by
\begin{align*}
I_{\Gamma, \nu} \leq \sum_{\gamma_1, \gamma_2} \int_{G} \int_G \chi_{Q'}(\gamma_1 h_1) \chi_{Q'}(\gamma_2 h_2)  d\nu(h_1) d\nu(h_2) = \Big( \sum_{\gamma \in \Gamma} \nu(\gamma^{-1} Q') \Big)^2 \, , 
\end{align*}
which is finite by the same argument as in the proof of item (1) in Lemma \ref{LemmaSimplePerturbedLatticeInGroups}.

Secondly, if $\supp(\nu)$ is compact and we let $Q'' = Q \, \supp(\nu)^{-1}$, then $Q''$ is measurable and pre-compact in $G$ and by the same argument as for Equation \ref{EqLemma2point5IndicatorIntegralBound} we have 
\begin{align*}
\int_{G} \int_G \chi_Q (g \gamma_1 h_1) \chi_Q(g \gamma_2 h_2) d\nu(h_1) d\nu(h_2) \leq \chi_{Q''}(g \gamma_1) \chi_{Q''}(g \gamma_2) \, .
\end{align*}
This bound along with a standard computation yields
\begin{align*}
I_{\Gamma, \nu} &\leq \frac{1}{\covol(\Gamma)} \int_{F_{\Gamma}} \sum_{\gamma_1, \gamma_2} \chi_{Q''}(g \gamma_1) \chi_{Q''}(g \gamma_2) dm_G(g) \\
&= \frac{1}{\covol(\Gamma)} \sum_{\gamma_1 \in \Gamma} \int_{G} \chi_{Q''}(g \gamma_1) \chi_{Q''}(g ) dm_G(g) = \frac{1}{\covol(\Gamma)} \sum_{\gamma_1 \in \Gamma} m_G(Q''\gamma_1^{-1} \cap Q'') \, . 
\end{align*}
Since $Q''$ is pre-compact, the latter series is a sum over finitely many $\gamma_1 \in \Gamma$ and hence finite.
\end{proof}
We will give more examples of perturbations $\nu$ for which random perturbed lattice orbits on real hyperbolic spaces are locally square-integrable in section \ref{Random perturbed lattice orbits in hyperbolic spaces}. 
\end{example}

\subsection{Autocorrelation measures}
\label{Autocorrelation measures}

The autocorrelation measure of a random measure is a signed Radon measure on $G$ that uniquely determines the $\mu$-covariance 
of linear statistics,
\begin{align*}
\Cov_{\mu}(Sf_1, Sf_2) &= \int_{\Radonplus(X)} p(f_1) \overline{p(f_2)} d\mu(p) - \iota_{\mu}^2 m_X(f_1) \overline{m_X(f_2)}\, , \quad f_1, f_2 \in \Borelbndinfty(X) \, .
\end{align*}
The definition is the following.
\begin{definition}[Autocorrelation measures]
\label{DefAutocorrelation}
Let $b \in \Borelbndinfty(X)$ be a positive function with $m_X(b) = 1$. The \emph{autocorrelation measure} of an invariant locally square-integrable random measure $\mu$ on $X$ is the bi-$K$-invariant signed Radon measure $\eta_{\mu}$ on $G$ defined by
\begin{align*}
\eta_{\mu}(\varphi) = \int_{\Radonplus(X)} \Big( \int_{X} \int_{X} \varphi(\varsigma(x)^{-1}\varsigma(y)) b(x) dp(x) dp(y) \Big) d\mu(p) - \iota_{\mu}^2 \int_G \varphi(g) dm_G(g) \, . 
\end{align*}
for all $\varphi \in \Borelbndinfty(G)$.
\end{definition}
The truncation function $b \in \Borelbndinfty(X)$ is necessary for the definition as the bivariante functions $(x, y) \mapsto \varphi(\varsigma(x)^{-1}\varsigma(y))$ do not have bounded support. However, the definition does not depend on the choice of $b$ by the following Lemma.
\begin{lemma}
\label{LemmaAutocorrelationRelationtoCovariance}
The autocorrelation measure $\eta_{\mu}$ of an invariant locally square-integrable random measure $\mu$ satisfies
\begin{align*}
\eta_{\mu}(\psi^* * \varphi) = \Cov_{\mu}(S(\varphi \circ \varsigma), S(\psi \circ \varsigma)) 
\end{align*}
for all right-$K$-invariant $\varphi, \psi \in \Borelbndinfty(G)$. In particular, the definition of $\eta_{\mu}$ is independent of the choice of $b \in \Borelbndinfty(X)$.
\end{lemma}
\begin{proof}
Assume without loss of generality that $m_G(\varphi) = m_G(\psi) = 0$, so that $m_G(\psi^* * \varphi) = \overline{m_G(\psi)} m_G(\varphi) = 0$. Then it suffices to show that
\begin{align*}
\eta_{\mu}(\psi^* * \varphi) = \int_{\Radonplus(X)} \Big( \int_{X} \int_{X} \overline{\psi(\varsigma(x))} \varphi(\varsigma(y))  dp(x) dp(y) \Big) d\mu(p) \, . 
\end{align*}
Given this identity, the subspace of functions of the form $\psi^* * \varphi$ are enough to uniquely determine $\eta_{\mu}$, so $\eta_{\mu}$ will be independent of the choice of $b$.

First, we write 
\begin{align*}
(\psi^* * \varphi)(\varsigma(x)^{-1}\varsigma(y)) = \int_G\overline{\psi(h \varsigma(x))} \varphi(h\varsigma(y)) dm_G(h) 
\end{align*}
and using Fubini, we get that
\begin{align*}
\eta_{\mu}(\psi^* * \varphi) = \int_G \int_{\Radonplus(X)} \Big( \int_{X} \int_{X} \overline{\psi(h \varsigma(x))} \varphi(h\varsigma(y))  b(x) dp(x) dp(y) \Big) d\mu(p) dm_G(h) \, . 
\end{align*}
Since $h\varsigma(x).o = h.x = \varsigma(h.x).o$ for all $x \in X$ then $\psi(h\varsigma(x)) = \psi(\varsigma(h.x))$ and $\varphi(h\varsigma(y)) = \varphi(\varsigma(h.y))$ by right-$K$-invariance. Thus
\begin{align*}
\int_{X} \int_{X} \overline{\psi(h \varsigma(x))} \varphi(h\varsigma(y)) b(x)  dp(x) dp(y) = \int_{X} \int_{X} \overline{\psi(\varsigma(x))} \varphi(\varsigma(y)) b(h^{-1}.x) dh_*p(x) dh_*p(y)  
\end{align*}
and by the $G$-invariance of $\mu$, we see that
\begin{align*}
\eta_{\mu}(\psi^* * \varphi) &= \int_G \int_{\Radonplus(X)} \Big( \int_{X} \int_{X} \overline{\psi(\varsigma(x))} \varphi(\varsigma(y))  b(h^{-1}.x) dp(x) dp(y) \Big) d\mu(p) dm_G(h) \\
&= \int_{\Radonplus(X)} \Big( \int_{X} \int_{X} \Big( \int_G b(h^{-1}.x) dm_G(h) \Big) \overline{\psi(\varsigma(x))} \varphi(\varsigma(y))  dp(x) dp(y) \Big) d\mu(p) \\
&= m_X(b) \int_{\Radonplus(X)} \Big( \int_{X} \int_{X} \overline{\psi(\varsigma(x))} \varphi(\varsigma(y))  dp(x) dp(y) \Big) d\mu(p) \, .
\end{align*}
Since $b \in \Borelbndinfty(X)$ was assumed to satisfy $m_X(b) = 1$ then we are done. 
\end{proof}
\begin{example}[Autocorrelation of the invariant Poisson point process]
\label{PoissonAutocorrelation}
Let $f = \sum_{j = 1}^N a_j \chi_{B_j}$ be a simple function for disjoint and bounded Borel sets $B_1, ..., B_N$ in $X$. Since the linear statistics $S\chi_{B_1}, ..., S\chi_{B_N}$ are Poisson distributed and independent, the variance of the linear statistic $Sf$ with respect to the invariant Poisson point process $\mu_{\Poi}$ is 
\begin{align*}
\Var_{\Poi}(Sf) &= \sum_{i = 1}^N \sum_{j = 1}^N a_i \overline{a_j} \Cov_{\Poi}(S\chi_{B_i}, S\chi_{B_j}) = \sum_{j = 1}^N |a_j|^2 \Var_{\Poi}(S\chi_{B_j}) \\
&= \sum_{j = 1}^N |a_j|^2 m_X(B_j) = \int_X |f(x)|^2 dm_X(x) \, . 
\end{align*}
The latter can be rewritten as $ \int_X |f(x)|^2 dm_X(x) = \int_G |\varphi_f(g)|^2 dm_G(g) = (\varphi_f^* * \varphi_f)(e) $, so that 
\begin{align*}
\Var_{\Poi}(Sf) = \delta_e(\varphi_f^* * \varphi_f)  \, . 
\end{align*}
This identity extends to every $f \in \Borelbndinfty(X)$ and by polarization we find that the autocorrelation measure of the $m_X$-Poisson point process is $\eta_{\Poi} = \delta_e$ as a linear functional on $\Borelbndinfty(G, K)$. In particular, the number variance of the process is 
\begin{align*}
\NV_{\Poi}(R) = \int_X |\chi_{B_R}(x)|^2 dm_X(x) = \Vol_X(B_R) \,  . 
\end{align*}
\end{example}
\begin{example}[Intensity and autocorrelation of a random lattice orbit]
\label{ExampleLatticeAutocorrelation}
Using the explicit description for the probability measure $m_{G/\Gamma}$ from Example \ref{ExampleRandomLatticeOrbits}, we now compute the autocorrelation $\eta_{\Gamma} := \eta_{\mu_{\Gamma}}$. The expectation of linear statistics can be computed using a fundamental domain $F_{\Gamma} \subset G$ of $\Gamma$ to be
\begin{align*}
\bE_{\mu_{\Gamma}}(S f) &= |\Gamma_o|^{-1}\covol(\Gamma)^{-1} m_G(\varphi_f)  
\end{align*}
for every $f \in \Borelbndinfty(X)$, where $\varphi_f(g) = f(g.o)$. In particular, the intensity is $\iota_{\Gamma} := \iota_{\mu_{\Gamma}} = |\Gamma_o|^{-1} \covol(\Gamma)^{-1}$. The second correlation is computed in a similar way to be
\begin{align*}
\bE_{\mu_{\Gamma}}(Sf_1 \overline{Sf}_2) &= |\Gamma_o|^{-2}\covol(\Gamma)^{-1} \sum_{\gamma \in \Gamma}   \int_{G}  f_1(g\gamma.o) \overline{f_2(g.o)} dm_G(g) \\
&=  |\Gamma_o|^{-2}\covol(\Gamma)^{-1} \sum_{\gamma \in \Gamma} (\varphi_{f_2}^* * \varphi_{f_1})(\gamma)
\end{align*}
for all $f_1, f_2 \in \Borelbndinfty(X)$. With this the autocorrelation measure can be identified as
\begin{align*}
\eta_{\Gamma}(\varphi^* * \varphi) &= \frac{1}{|\Gamma_o|\covol(\Gamma)} \Big( \frac{1}{|\Gamma_o|}\sum_{\gamma \in \Gamma} (\varphi^* * \varphi)(\gamma)  - \frac{1}{|\Gamma_o|\covol(\Gamma)} \Big| \int_G \varphi(g) dm_G(g) \Big|^2  \Big)
\end{align*}
for all right-$K$-invariant $\varphi \in \Borelbndinfty(G)$. In shorter notation, 
\begin{align} 
\label{EqRandomLatticeOrbitAutocorrelation}
\eta_{\Gamma}(\varphi) = \iota_{\Gamma} \Big(\frac{1}{|\Gamma_o|}\delta_{\Gamma}(\varphi) - \iota_{\Gamma} m_G(\varphi) \Big)
\end{align}
for all bi-$K$-invariant functions $\varphi \in \Borelbndinfty(G, K)$.
\end{example}
\begin{example}[Intensity and autocorrelation of a random perturbed lattice orbit]
\label{ExamplePerturbedLatticeAutocorrelation}
Let $G$ be such that $m_G(K) = 0$, $\Gamma < G$ a lattice and $\nu \in \Prob(G)$ a right-$K$-invariant measure on $G$ absolutely continuous with respect to the Haar measure $m_G$. The expected value of a linear statistic $Sf$, $f \in \Borelbndinfty(X)$, with respect to the random perturbed lattice orbit $\mu_{\Gamma, \nu}$ can be computed using the $\Gamma$-invariance of $\nu^{\otimes \Gamma}$ and the same computation as for the random lattice orbit. Fubini yields 
\begin{align*}
\bE_{\mu_{\Gamma, \nu}}(Sf) &= \int_{G / \Gamma} \Big( \int_Z \sum_{\gamma \in \Gamma} f(g\gamma z_{\gamma}.o) d\nu^{\otimes \Gamma}(z) \Big) dm_{G / \Gamma}(g\Gamma) \\
&= \int_{G / \Gamma} \sum_{\gamma \in \Gamma} \Big( \int_G f(g\gamma h.o) d\nu(h) \Big) dm_{G / \Gamma}(g\Gamma) \\
&=  \int_{G} \Big( \int_{G / \Gamma} \sum_{\gamma \in \Gamma} f(g\gamma h.o)  dm_{G / \Gamma}(g\Gamma) \Big) d\nu(h) \, .
\end{align*}
By the computation for the intensity of random lattice orbits and the unimodularity of $G$ we get
\begin{align*}
\bE_{\mu_{\Gamma, \nu}}(Sf) = \frac{1}{\covol(\Gamma)} \int_G \Big( \int_G f(gh.o) dm_G(g) \Big) d\nu(h) = \frac{m_G(\varphi_f)}{\covol(\Gamma)} \, . 
\end{align*}
Thus the intensity of $\mu_{\Gamma, \nu}$ is $\iota_{\Gamma, \nu} = \covol(\Gamma)^{-1}$, in particular independent of $\nu$.

For the second moment of the process $\mu_{\Gamma, \nu}$, we assume that $\Gamma$ is cocompact or $\supp(\nu)$ is compact, so that $\mu_{\Gamma, \nu}$ is locally square-integrable by Lemma \ref{LemmaPerturbedLatticeisLocallySquareIntegrable}. Using Fubini,
\begin{align*}
\bE_{\mu_{\Gamma, \nu}}(Sf_1 \overline{Sf_2}) = \int_{G / \Gamma} \Big( \sum_{\gamma_1, \gamma_2 \in \Gamma} \int_{Z} f_1 (g\gamma_1 z_{\gamma_1}.o) \overline{f_2(g\gamma_2 z_{\gamma_2}.o)} d\nu^{\otimes \Gamma}(z) \Big) dm_{G / \Gamma}(g\Gamma) \, . 
\end{align*}
Note how
\begin{align*}
\int_{Z} f_1 (g\gamma_1 z_{\gamma_1}.o) \overline{f_2(g\gamma_2 z_{\gamma_2}.o)} d\nu^{\otimes \Gamma}(z) = 
\begin{dcases}
(\varphi_{f_1}\overline{\varphi}_{f_2} * \check{\nu})(g\gamma) &\mbox{ if } \gamma_1 = \gamma_2 = \gamma \\
(\varphi_{f_1} * \check{\nu})(g\gamma_1) \overline{(\varphi_{f_2} * \check{\nu})(g\gamma_2)} &\mbox{ if } \gamma_1 \neq \gamma_2 \, , 
\end{dcases}
\end{align*}
where $\varphi_f(g) = f(g.o)$ as before and $\check{\nu}$ is the left-$K$-invariant probability measure on $G$ defined by $\check{\nu}(Q) = \nu(Q^{-1})$ for Borel sets $Q \subset G$. From this we get that
\begin{align*}
\bE_{\mu_{\Gamma, \nu}}(Sf_1 \overline{Sf_2}) &=  \int_{G / \Gamma} \Big( \sum_{\gamma_1 \neq \gamma_2}  (\varphi_{f_1} * \check{\nu})(g\gamma_1) \overline{(\varphi_{f_2} * \check{\nu})(g\gamma_2)} \Big) dm_{G / \Gamma}(g\Gamma) \\
& \quad\quad\quad + \int_{G / \Gamma} \Big( \sum_{\gamma \in \Gamma} (\varphi_{f_1}\overline{\varphi}_{f_2} * \check{\nu})(g\gamma) \Big) dm_{G / \Gamma}(g\Gamma)  \\
&=  \int_{G / \Gamma} \Big( \sum_{\gamma_1, \gamma_2 \in \Gamma}  (\varphi_{f_1} * \check{\nu})(g\gamma_1) \overline{(\varphi_{f_2} * \check{\nu})(g\gamma_2)} \Big) dm_{G / \Gamma}(g\Gamma) \\
& \quad\quad\quad - \int_{G / \Gamma} \Big( \sum_{\gamma \in \Gamma}  (\varphi_{f_1} * \check{\nu})(g\gamma) \overline{(\varphi_{f_2} * \check{\nu})(g\gamma)} \Big) dm_{G / \Gamma}(g\Gamma) \\
& \quad\quad\quad + \int_{G / \Gamma} \Big( \sum_{\gamma \in \Gamma} (\varphi_{f_1}\overline{\varphi}_{f_2} * \check{\nu})(g\gamma) \Big) dm_{G / \Gamma}(g\Gamma)  \, . 
\end{align*}
From the computations made for the random lattice orbits, 
\begin{align*}
\int_{G / \Gamma} \Big( \sum_{\gamma_1, \gamma_2 \in \Gamma}  (\varphi_{f_1} * \check{\nu})(g\gamma_1) &\overline{(\varphi_{f_2} * \check{\nu})(g\gamma_2)} \Big) dm_{G/\Gamma}(g\Gamma) = \\
&= \frac{1}{\covol(\Gamma)} \sum_{\gamma \in \Gamma} (\nu * \varphi_{f_2}^* * \varphi_{f_1} * \check{\nu})(\gamma) \, , \\
\int_{G / \Gamma} \Big( \sum_{\gamma \in \Gamma}  (\varphi_{f_1} * \check{\nu})(g\gamma) &\overline{(\varphi_{f_2} * \check{\nu})(g\gamma)} \Big) dm_{G/\Gamma}(g\Gamma) = \\
&= \frac{1}{\covol(\Gamma)} \int_G  (\varphi_{f_1} * \check{\nu})(g) \overline{(\varphi_{f_2} * \check{\nu})(g)} dm_G(g) \, ,
\end{align*}
and by unimodularity of $G$,
\begin{align*}
\int_{G / \Gamma} \Big( \sum_{\gamma \in \Gamma} (\varphi_{f_1}\overline{\varphi}_{f_2} * \check{\nu})(g\gamma) \Big) dm_{G / \Gamma}(g\Gamma) &= \frac{1}{\covol(\Gamma)} \int_G  (\varphi_{f_1}\overline{\varphi}_{f_2} * \check{\nu})(g)  dm_G(g)\\
&= \frac{1}{\covol(\Gamma)} \int_G \Big( \int_G \varphi_{f_1}(gh) \overline{\varphi_{f_2}(gh)} dm_G(g) \Big) d\nu(h) \\
&= \frac{1}{\covol(\Gamma)} \int_G \varphi_{f_1}(g) \overline{\varphi_{f_2}(g)} dm_G(g) \, . 
\end{align*}
Thus
\begin{align*}
\bE_{\mu_{\Gamma, \nu}}(Sf_1 \overline{Sf_2}) = \frac{1}{\covol(\Gamma)} \Big( &\sum_{\gamma \in \Gamma} (\nu * \varphi_{f_2}^* * \varphi_{f_1} * \check{\nu})(\gamma) \\
&- \int_G  (\varphi_{f_1} * \check{\nu})(g) \overline{(\varphi_{f_2} * \check{\nu})(g)} dm_G(g) \\
&+ \int_G \varphi_{f_1}(g) \overline{\varphi_{f_2}(g)} dm_G(g) \Big) \, , 
\end{align*}
so the autocorrelation measure $\eta_{\Gamma, \nu}$ of $\mu_{\Gamma, \nu}$ is 
\begin{align*}
\eta_{\Gamma, \nu}(\varphi^* * \varphi) &= \frac{1}{\covol(\Gamma)} \Big( \sum_{\gamma \in \Gamma} (\nu * \varphi^* * \varphi * \check{\nu})(\gamma)  -  \int_G  |(\varphi * \check{\nu})(g)|^2 dm_G(g) \\
&\quad\quad\quad\quad\quad\quad\quad + \int_G |\varphi(g)|^2  dm_G(g) - \frac{1}{\covol(\Gamma)} \Big| \int_G \varphi(g) dm_G(g) \Big|^2 \Big) \\
&= |\Gamma_o|^2 \Var_{\mu_{\Gamma}}(S((\varphi * \check{\nu}) \circ \varsigma ) + \covol(\Gamma)^{-1} \big( m_G(|\varphi|^2) -m_G(|\varphi * \check{\nu}|^2) \big) \, . 
\end{align*}
In short hand notation, considering the autocorrelation measure as a functional on the sub-algebra $\Borelbndinfty(G, K)$ of bi-$K$-invariant functions on $G$, we have 
\begin{align*}
\eta_{\Gamma, \nu} = |\Gamma_o|^2 \,  \check{\nu} * \eta_{\Gamma} * \nu + \frac{1}{\covol(\Gamma)} \big( \delta_e - \check{\nu} * \nu \big) \, , 
\end{align*}
where $\eta_{\Gamma}$ is the random lattice orbit autocorrelation of $\Gamma$ from equation \ref{EqRandomLatticeOrbitAutocorrelation}. In particular, if $\nu = m_K$ then $\eta_{\Gamma, m_K} = \eta_{\Gamma}$.
\end{example}

%% file: EuclideanDiffraction.tex
In Subsection \ref{The spherical transform on R^n} we introduce the Fourier transform and compute the Fourier transform of the indicator function on a centered Euclidean ball in terms of Bessel functions of the first kind. An asymptotic expansion and an integral formula for the relevant Bessel functions are surveyed in Subsection \ref{Some properties of Bessel functions}.

\subsection{Euclidean space}
\label{Euclidean space}

Let $n \in \N$ and consider $\R^n$ with the metric induced by the Euclidean norm $\norm{\cdot}$ and the origin $o \in \R^n$ as a reference point. For $r > 0$ we let $B_r$ denote the Euclidean ball centered at $o$. We will write $dx = dm_{\R^n}(x)$ for the Lebesgue measure on $\R^n$. If we denote by $\sigma_{n-1}$ the canonical surface measure on the $(n-1)$-dimensional unit sphere $\bS^{n-1}$ with mass $\sigma_{n - 1}(\bS^{n - 1}) = 2 \pi^{n/2}/\Gamma( n / 2)$, then integration against the Lebesgue measure can be written as
\begin{align*}
m_{\R^n}(f) = \int_0^{\infty} \int_{\bS^{n-1}}  f(tu)  \, d\sigma_{n-1}(u) \, t^{n-1} dt  \, , \quad f \in \Borelbndinfty(\R^n) \, .
\end{align*}

With this measure we have that $\Vol_{\R^n}(B_r) = \pi^{n/2} \Gamma(\frac{n}{2} + 1)^{-1} r^n$. 

We consider the transitive action of $\R^n$ on itself by translation and emphasize that invariance of a random measure $\mu$ on $\R^n$ refers to invariance under this action, which corresponds to $G = \R^n$ and $K = \{0\}$ in the setting of the previous section. 


\subsection{The Fourier transform on $\R^n$}
\label{The spherical transform on R^n}
We take the Fourier transform of a  function $f \in \Borelbndinfty(\R^n)$ to mean the function $\hat{f} \in C_0(\R^n)$ given by  
\begin{align*}
\hat{f}(\xi) = \int_{\R^n} f(x) \e^{-i \langle x, \xi \rangle} dx   \, .
\end{align*}
In particular, $\hat{f}_1 = \hat{f}_2$ if and only if $f_1 = f_2$ Lebesgue almost everywhere. The classical Plancherel formula reads as 
\begin{align*}
\int_{\R^n} |f(x)|^2 dx = (2\pi)^{-n} \int_{\R^n} |\hat{f}(\xi)|^2 d\xi \, . 
\end{align*}
In particular, the Fourier transform extends to a unitary map
$$L^2(\R^n) \longrightarrow L^2(\R^n, (2\pi)^{-n} d\xi) \, . $$
\begin{example}[Fourier transform of the indicator function on centered Euclidean balls]
\label{ExIndicatorFT}
Let us compute the spherical transform of the indicator function $\chi_{B_r}$. Since $\chi_{B_r}$ is radial, averaging over the unit sphere $\bS^{n-1} \subset \R^n$ yields
\begin{align*}
\hat{\chi}_{B_r}(\xi) = \frac{2 \pi^{\frac{n}{2}}}{\Gamma(\frac{n}{2})} \int_{B_r} \omega^{(n)}_{\norm{\xi}}(x) \,  dx
\end{align*}
where 
\begin{align*}
\omega_{\norm{\xi}}^{(n)}(x) &= \frac{\Gamma(\frac{n}{2})}{2 \pi^{\frac{n}{2}}} \int_{\bS^{n-1}} \e^{- i \norm{\xi} \langle v, x \rangle} d\sigma_{n-1}(v) \, . 
\end{align*}
These are the $\SO(n)$-\emph{spherical functions} of the Euclidean group $\SO(n) \ltimes \R^n$ of rotations and translations, and they can be expressed in terms of \emph{Bessel functions of the first kind},
\begin{align*}
J_{\alpha}(z) = \frac{2^{1 - \alpha} z^{\alpha}}{ \pi^{\frac{1}{2}}\Gamma(\alpha + \frac{1}{2})} \int_0^1 (1 - s^2)^{\alpha - \frac{1}{2}} \cos(z s) ds \, , \quad \alpha, z \in \C \, , \,\, \Re(\alpha) > -\frac{1}{2} \, .
\end{align*}
The following formulas can be derived from \cite[Lemma 4.13 + Theorem 4.15, p. 170-171]{SteinWeissFourierAnalysis}.
\begin{lemma}
\label{LemmaFTofIndicatorFunction}
Let $\xi \in \R^n$ and $x \in \R^n$. Then
\begin{align*}
\omega^{(n)}_{\norm{\xi}}(x) &= \frac{2^{\frac{n - 2}{2}} \Gamma(\frac{n}{2})}{(\norm{\xi} \norm{x})^{\frac{n-2}{2}}} J_{\frac{n - 2}{2}}(\norm{\xi} \norm{x}) \quad \mbox{and} \quad  \hat{\chi}_{B_r}(\xi) = (2\pi)^{\frac{n}{2}} \Big( \frac{r}{\norm{\xi}} \Big)^{\frac{n}{2}} J_{\frac{n}{2}}(r\norm{\xi}) \, . 
\end{align*}
\end{lemma}
In particular, the Fourier transform satisfies the homogeneity $\hat{\chi}_{B_r}(\xi) = r^n \hat{\chi}_{B_1}(r \xi)$ as expected, and $\hat{\chi}_{B_r}(0) = \Vol_{\R^n}(B_r)$.
\end{example}


\subsection{Some properties of Bessel functions}
\label{Some properties of Bessel functions}

A crucial tool for the following sections will be a first order asymptotic approximation of the (half-)integer Bessel functions $J_{\frac{n}{2}}$, and hence for the spherical functions $\omega_{\lambda}^{(n)}$.

\begin{lemma}[\cite{Watson1923BesselF}, p.199]
\label{LemmaBesselAsymptoticExpansion}
Let $\alpha > -1/2$ and let $\phi_{\alpha} = (2\alpha + 1)\pi/4$. As $t \rightarrow +\infty$, the Bessel function $J_{\alpha}(t)$ admits the asymptotic expansion
\begin{align*}
J_{\alpha}(t) \sim \Big(\frac{2}{\pi t}\Big)^{1/2} \Big( \cos(t - \phi_{\alpha}) \sum_{j = 0}^{\infty} \frac{(-1)^j A_{2j}}{t^{2j}}  -  \sin(t - \phi_{\alpha}) \sum_{j = 0}^{\infty} \frac{(-1)^j A_{2j+1}}{t^{2j+1}}  \Big) \, , 
\end{align*}
where
\begin{align*}
A_j = \frac{\Gamma(\alpha + j + \frac{1}{2})}{2^j j! \Gamma(\alpha - j + \frac{1}{2})} \, , \quad j \in \Z_{\geq 0} \, . 
\end{align*}
In particular, there is a constant $\kappa_{\alpha} > 0$ such that
\begin{align*}
\Big| J_{\alpha}(t) - \Big(\frac{2}{\pi t}\Big)^{1/2}\cos(t - \phi_{\alpha}) \Big| \leq \frac{\kappa_{\alpha}}{t^{3/2}}  \, , \quad \forall \, t \geq 1 \, .
\end{align*}
\end{lemma}

\begin{corollary}
\label{CorollaryBesselUpperBound}
Let $\alpha \geq 1/2$. Then there is a constant $C_{\alpha} > 0$ such that 
\begin{align*}
|J_{\alpha}(t)| \leq C_{\alpha} (1 + t)^{-1/2} \, , \quad \forall \, t \geq 0 \, . 
\end{align*}
\end{corollary}

\begin{proof}
If $0 \leq t \leq 1$, then $t^{\alpha} \leq 2(1 + t)^{-1/2}$, and we bound $J_{\alpha}$ from above and below by
\begin{align*}
|J_{\alpha}(t)| = \frac{2^{1 - \alpha} t^{\alpha}}{ \pi^{\frac{1}{2}}\Gamma(\alpha + \frac{1}{2})} \Big| \int_0^1 (1 - s^2)^{\alpha - \frac{1}{2}} \cos(t s) ds \Big| \leq \frac{2^{1 - \alpha} t^{\alpha}}{ \pi^{\frac{1}{2}}\Gamma(\alpha + \frac{1}{2})} \leq  \frac{2^{2 - \alpha}}{ \pi^{\frac{1}{2}}\Gamma(\alpha + \frac{1}{2})(1 + t)^{1/2}} \, . 
\end{align*}
If $t \geq 1$ then we let $M_{\alpha}(t) = (\frac{2}{\pi t})^{1/2} \cos(t - \phi_{\alpha})$. We use that $t^{-1/2} \leq 2(1 + t)^{-1/2}$ for all $t \geq 1$ and Lemma \ref{LemmaBesselAsymptoticExpansion} to get the bound
\begin{align*}
|J_{\alpha}(t)| \leq |M_{\alpha}(t)| + |J_{\alpha}(t) - M_{\alpha}(t)| \leq \Big(\frac{2}{\pi t}\Big)^{1/2} + \frac{\kappa_{\alpha}}{t^{3/2}} \leq \frac{\big(\frac{2}{\pi}\big)^{1/2} + \kappa_{\alpha} }{t^{1/2}} \leq 2\, \frac{\big(\frac{2}{\pi}\big)^{1/2} + \kappa_{\alpha} }{(1 + t)^{1/2}} \, .
\end{align*}
To finish the proof, we take 
\begin{align*}
C_{\alpha} = \max\Big( \frac{2^{2 - \alpha}}{ \pi^{\frac{1}{2}}\Gamma(\alpha + \frac{1}{2})} , \, 2 \Big( \Big(\frac{2}{\pi}\Big)^{1/2} + \kappa_{\alpha} \Big) \Big) \, . 
\end{align*}
\end{proof}

In proving Theorem \ref{Theorem1A}, we will encounter an integral of the following type. 

\begin{lemma}
\label{LemmaBesselIntegral}
Let $\alpha \geq 1/2$ and $\beta \geq 1$. Then for every $T \geq 1$ there is a $\delta_T = \delta_{T}(\alpha, \beta)$ with $\delta_{T} \rightarrow 0$ as $T \rightarrow +\infty$ such that
\begin{align*}
\int_0^T t^{\beta} J_{\alpha}(t)^2 dt = \Big( \frac{1}{\pi\beta} + \delta_T \Big)  T^{\beta}   \, . 
\end{align*}
\end{lemma}

\begin{proof}[Proof of Lemma \ref{LemmaBesselIntegral}]
Let $M_{\alpha}(t) = (\frac{2}{\pi t})^{1/2} \cos(t - \phi_{\alpha})$, so that $|J_{\alpha}(t) - M_{\alpha}(t)| \leq \kappa_{\alpha} t^{-3/2}$ for all $t \geq 1$ as in Lemma \ref{LemmaBesselAsymptoticExpansion}. We bound the integral in question from above and below by
\begin{align*}
\int_0^T t^{\beta} M_{\alpha}(t)^2 dt - \int_0^T t^{\beta} |J_{\alpha}(t)^2 &- M_{\alpha}(t)^2| dt \\
&\leq \int_0^T t^{\beta} J_{\alpha}(t)^2 dt \\
&\leq \int_0^T t^{\beta} M_{\alpha}(t)^2 dt + \int_0^T t^{\beta} |J_{\alpha}(t)^2 - M_{\alpha}(t)^2| dt \, . 
\end{align*}
First, 
\begin{align*}
\int_0^T t^{\beta} M_{\alpha}(t)^2 dt &= \frac{2}{\pi} \int_0^T t^{\beta - 1} \cos^2(t - \phi_{\alpha}) dt \\
&=\frac{1}{\pi} \Big( \int_0^T t^{\beta - 1} dt + \int_0^T t^{\beta - 1} \cos(2(t - \phi_{\alpha})) dt \Big) \\
&= \frac{1}{\pi} \Big( \frac{T^{\beta}}{\beta} + T^{\beta} \int_0^1 s^{\beta - 1} \cos(2(Ts - \phi_{\alpha})) ds \Big) \, . 
\end{align*}
Note that $\int_0^1 s^{\beta - 1} \cos(2(Ts - \phi_{\alpha})) ds \rightarrow 0$ as $T \rightarrow +\infty$ by the Riemann-Lebesgue Lemma. 

Secondly, by Lemma \ref{LemmaBesselAsymptoticExpansion} and Corollary \ref{CorollaryBesselUpperBound} we have
\begin{align}
\label{EqBesselSquareError}
|J_{\alpha}(t)^2 - M_{\alpha}(t)^2| &\leq |J_{\alpha}(t)|^2 + |M_{\alpha}(t)|^2 \leq \frac{C_{\alpha}^2}{1 + t} + \frac{2}{\pi t} \leq \frac{C_{\alpha}^2 + 1}{t} \quad \mbox{ if } \, 0\leq t \leq 1 \, ,  \\
|J_{\alpha}(t)^2 - M_{\alpha}(t)^2| &= |J_{\alpha}(t) + M_{\alpha}(t)||J_{\alpha}(t) - M_{\alpha}(t)| \\
&\leq \Big( \frac{C_{\alpha}}{(1 + t)^{1/2}} + \Big( \frac{2}{\pi t} \Big)^{1/2} \Big) \frac{\kappa_{\alpha}}{t^{3/2}} \leq \frac{(C_{\alpha} + 1) \kappa_{\alpha}}{t^2} \quad \mbox{ if } \,  t \geq 1 \, ,
\end{align}
so that 
\begin{align*}
\int_0^T t^{\beta} |J_{\alpha}(t)^2 - M_{\alpha}(t)^2| dt &= \int_0^{1} t^{\beta} |J_{\alpha}(t)^2 - M_{\alpha}(t)^2| dt + \int_{1}^T t^{\beta} |J_{\alpha}(t)^2 - M_{\alpha}(t)^2| dt \\
&\leq (C^2_{\alpha} + 1) \int_0^{1} t^{\beta - 1} dt + (C_{\alpha} + 1) \kappa_{\alpha} \int_{1}^T t^{\beta - 2}  dt \\
&= \frac{C_{\alpha}^2 + 1}{\beta} + (C_{\alpha} + 1)  \kappa_{\alpha} \frac{T^{\beta - 1} - 1}{\beta - 1}  \, , \quad T \geq 1 \, .  
\end{align*}
Note that the right hand side remains finite in the limiting case of $\beta = 1^+$. Putting everything together, we define
\begin{align*}
\delta_T = T^{-\beta} \int_0^T t^{\beta} J_{\alpha}(t)^2 dt - \frac{1}{\pi \beta} \, , 
\end{align*}
so that
\begin{align*}
\int_0^T t^{\beta} J_{\alpha}(t)^2 dt = \Big(\frac{1}{\pi\beta} + \delta_T \Big) T^{\beta} \, , \quad T \geq 1 
\end{align*}
and
\begin{align*}
|\delta_T| &\leq T^{-\beta} \Big( \int_0^T t^{\beta} M_{\alpha}(t)^2 dt - \frac{T^{\beta}}{\pi \beta} + \int_0^T t^{\beta} |J_{\alpha}(t)^2 - M_{\alpha}(t)^2| dt \Big) \\
&\leq \frac{1}{\pi} \Big|\int_0^1 s^{\beta - 1} \cos(2(Ts - \phi_{\alpha})) ds \Big| + \frac{C_{\alpha}^2 + 1}{\beta T^{\beta}} +  (C_{\alpha} + 1)\kappa_{\alpha}\frac{T^{\beta - 1} - 1}{(\beta - 1) T^{\beta}} \longrightarrow 0
\end{align*}
as $T \rightarrow +\infty$. We also note that this limit is preserved in the limiting case $\beta = 1^+$. 
\end{proof}
%

\section{Euclidean diffraction measures}

We define the diffraction measure of an invariant locally square-integrable random measure on $\R^n$ and recall a known upper bound for its asymptotic volume growth. 

The autocorrelation measure $\eta_{\mu}$ of an invariant locally square-integrable random measure $\mu$ on $\R^n$ is \emph{positive-definite} in the sense of measures, that is 
\begin{align*}
\eta_{\mu}(f^* * f) \geq 0 \, , \quad \forall \,  f \in \Borelbndinfty(\R^n)\, . 
\end{align*}
The positive-definiteness of $\eta_{\mu}$ is the crucial property that will allow us to define the diffraction measure of $\mu$. Recall that a Radon measure $\eta$ is \emph{tempered} if it extends to a continuous linear functional on the Schwartz space $\sS(\R^n)$ of smooth functions $f : \R^n \rightarrow \C$ with sub-polynomial decay, that is,
\begin{align*}
\sup_{x \in \R^n} (1 + \norm{x})^{\alpha} |\partial^{\beta} f(x)| < +\infty
\end{align*}
for all $\alpha \in \Z_{\geq 0}$ and all multi-indices $\beta \in \Z_{\geq 0}^n$. The following classical result establishes the existence and uniqueness of a measure-theoretic Fourier transform of positive-definite measures.
\begin{theorem}[Bochner-Schwartz]
\label{ThmBochnerSchwartz}
Let $\eta$ be a positive-definite signed Radon measure on $\R^n$. Then there is a unique positive tempered Radon measure $\hat{\eta}$ on $\R^n$ such that 
$$\eta(f) = \int_{\R^n} \hat{f}(\xi) d\hat{\eta}(\xi) $$
for all functions $f \in \sS(\R^n)$. In particular, $\eta$ is tempered.
\end{theorem}
We refer to $\hat{\eta}$ as the \emph{Fourier transform} of $\eta$. 
\begin{definition}[Euclidean diffraction measures] 
The \emph{diffraction measure} of an invariant locally square-integrable random measure $\mu$ on $\R^n$ is the unique Fourier transform $\hat{\eta}_{\mu} \in \Radonplus(\R^n)$ of the autocorrelation measure $\eta_{\mu}$ on $\R^n$. 
\end{definition}
In particular, the $\mu$-variance of linear statistics $Sf$ with $f \in \sS(\R^n)$ can written in terms of the diffraction measure as
\begin{align*}
\Var_{\mu}(Sf) = \eta_{\mu}(f^* * f) = \int_{\R^n} |\hat{f}(\xi)|^2 d\hat{\eta}_{\mu}(\xi) \, .
\end{align*}
The linear span of the functions $f^* * f \in \sS(\R^n)$ are enough to uniquely determine the autocorrelation measure $\eta_{\mu}$, so the above formula for the $\mu$-variance uniquely determines the diffraction measure $\hat{\eta}_{\mu}$. Moreover, one can make use of the mean ergodic theorem to see that the diffraction measure satisfies $\hat{\eta}_{\mu}(\{0\}) = 0$.

We now justify the extension of this formula to all functions $f \in \Borelbndinfty(\R^n)$, in particular for indicator functions $\chi_{B_r}$ with $r > 0$. 

\begin{lemma}
\label{LemmaExtensionOfDiffractionFormulaToMeasurableFunctions}
Let $\mu$ be an invariant locally square-integrable random measure on $\R^n$ and $f \in \Borelbndinfty(\R^n)$. Then
\begin{align*}
\eta_{\mu}(f^* * f) = \int_{\R^n} |\hat{f}(\xi)|^2 d\hat{\eta}_{\mu}(\xi) \, .
\end{align*}
In particular, $\hat{f} \in L^2(\R^n, \hat{\eta}_{\mu})$.
\end{lemma}
\begin{proof}
For every $\varepsilon > 0$, let $\beta_{\varepsilon} \in \sS(\R^n)$ be the function satisfying $\hat{\beta}_{\varepsilon}(\xi) = \e^{-\varepsilon^2 \norm{\xi}^2}$ for all $\xi \in \R^n$. Setting $f_{\varepsilon} = \beta_{\varepsilon} * f \in \sS(\R^n)$ then a standard argument shows that $f_{\varepsilon}$ converges to $f$ Lebesgue-almost everywhere. Since $f$ is measurable and bounded with bounded support, the convolution $f^* * f$ defines a continuous compactly supported function on $\R^n$. It follows that $f_{\varepsilon}^* * f_{\varepsilon}$ converges \emph{uniformly} to $f^* * f$. In particular, since $\eta_{\mu}$ is a Radon measure then
\begin{align*}
\lim_{\varepsilon \rightarrow 0^+} \eta_{\mu}(f_{\varepsilon}^* * f_{\varepsilon}) = \eta_{\mu}(f^* * f) \, . 
\end{align*}
On the other hand, by definition of $\beta_{\varepsilon}$ and the diffraction measure we have that
\begin{align*}
\eta_{\mu}(f_{\varepsilon}^* * f_{\varepsilon}) = \int_{\R^n} |\hat{f}(\xi)|^2 \e^{-2\varepsilon^2\norm{\xi}^2} d\hat{\eta}_{\mu}(\xi) \leq \int_{\R^n} |\hat{f}(\xi)|^2 d\hat{\eta}_{\mu}(\xi) \, . 
\end{align*}
Since $\e^{-2\varepsilon^2\norm{\xi}^2}$ is monotonely increasing as $\varepsilon \rightarrow 0^+$ with limit $1$ then the  monotone convergence Theorem tells us that
\begin{align*}
\eta_{\mu}(f^* * f) = \lim_{\varepsilon \rightarrow 0^+}\int_{\R^n} |\hat{f}(\xi)|^2 \e^{-2\varepsilon^2\norm{\xi}^2} d\hat{\eta}_{\mu}(\xi) = \int_{\R^n} |\hat{f}(\xi)|^2 d\hat{\eta}_{\mu}(\xi) \, .
\end{align*}
\end{proof}
The above Lemma imposes a growth condition on the diffraction measure $\hat{\eta}_{\mu}$, which one can give a more specific description of.
\begin{lemma}[\cite{Berg1975PotentialTO}, Proposition 4.9, p.21]
\label{LemmaEuclideanBergFrost}
Let $\mu$ be a random measure on $\R^n$. Then the diffraction measure of $\mu$ satisfies
\begin{align*}
\hat{\eta}_{\mu}(B_L) \ll_{n, \mu} \Vol_{\R^n}(B_L) 
\end{align*} 
for all sufficiently large $L > 0$.
\end{lemma}
\begin{example}[Poisson diffraction]
\label{PoissonDiffraction}
In Example \ref{PoissonAutocorrelation} we computed the variance of linear statistics for the unit intensity invariant Poisson point process, which in the Euclidean case reads as
$$  \Var_{\Poi}(Sf) = (f^* * f)(0) = \int_{\R^n} |f(x)|^2 dx \, . $$
By the Plancherel formula,
\begin{align*}
\Var_{\Poi}(Sf) = \int_{\R^n} |f(x)|^2 dx = (2\pi)^{-n} \int_{\R^n} |\hat{f}(\xi)|^2 d\xi\, ,
\end{align*}
so the diffraction measure of the Poisson process is $d\hat{\eta}_{\Poi}(\xi) = (2\pi)^{-n} d\xi$.  
\end{example}

\begin{example}[Diffraction of a random lattice]
\label{ExampleEuclideanLatticeDiffraction}
Let $\Gamma < \R^n$ be a lattice. Then there is a matrix $g_{\Gamma} \in \GL_n(\R)$ such that $\Gamma = g_{\Gamma} \Z^n$. The covolume of $\Gamma$ can then be written as $\covol(\Gamma) = \det(g_{\Gamma})$ and the \emph{dual lattice of $\Gamma$} is the lattice
\begin{align*}
\Gamma^{\perp} = \Big\{ \xi \in \R^n : \langle \gamma, \xi \rangle \in 2\pi \Z  \quad \forall \, \gamma \in \Gamma \Big\} \, . 
\end{align*}
The \emph{Poisson summation formula} states that
\begin{align*}
\sum_{\gamma \in \Gamma} f(\gamma) = \frac{1}{\covol(\Gamma)} \sum_{\xi \in \Gamma^{\perp}} \hat{f}(\xi) 
\end{align*}
for $f \in \sS(\R^n)$. By the formula for the autocorrelation measure in Example \ref{ExampleLatticeAutocorrelation} and the Poisson summation formula,
\begin{align*}
 \eta_{\Gamma}(f) = \frac{1}{\covol(\Gamma)} \sum_{\gamma \in \Gamma} f(\gamma) - \frac{1}{\covol(\Gamma)^2} \int_{\R^n} f(x) dx = \frac{1}{\covol(\Gamma)^2} \sum_{\xi \in \Gamma^{\perp} \backslash \{0\}} \hat{f}(\xi)
\end{align*}
for all $f \in \sS(\R^n)$. Thus the diffraction measure of the random lattice $\mu_{\Gamma}$ is 
$$ \hat{\eta}_{\Gamma} = \frac{1}{\covol(\Gamma)^2} \sum_{\xi \in \Gamma^{\perp} \backslash \{0\}} \delta_{\xi} \, .  $$
\end{example}

\begin{example}[Diffraction of a random perturbed lattice]
\label{ExampleEuclideanPerturbedLatticeDiffraction}
Since all lattices in Euclidean spaces are cocompact, random perturbed lattices are locally square-integrable by Lemma \ref{LemmaPerturbedLatticeisLocallySquareIntegrable} with a well-defined diffraction measure.

If $\nu \in \Prob(\R^n)$ we define its Fourier transform $\hat{\nu} : \R^n \rightarrow \C$ to be
\begin{align*}
\hat{\nu}(\xi) = \int_{\R^n} \e^{-i \langle x, \xi \rangle} d\nu(x) \, . 
\end{align*}
Then $\hat{\nu}$ is a continuous function, but does not necessarily vanish at infinity as in the absolutely continuous case. Moreover, one computes $\hat{\check{\nu} * \nu} = |\hat{\nu}|^2$. With this Fourier transform in mind and the formula for the autocorrelation of a random perturbed lattice $(\Gamma, \nu)$ from Example \ref{ExamplePerturbedLatticeAutocorrelation},
\begin{align*}
\eta_{\Gamma, \nu} = \check{\nu} * \eta_{\Gamma} * \nu + \iota_{\Gamma}(\delta_0 - \check{\nu} * \nu) = \check{\nu} * \nu * \eta_{\Gamma} + \iota_{\Gamma}(\delta_0 - \check{\nu} * \nu)  
\end{align*}
where $\iota_{\Gamma} = \covol(\Gamma)^{-1}$, the corresponding diffraction measure is
\begin{align*}
d\hat{\eta}_{\Gamma, \nu}(\xi) = |\hat{\nu}(\xi)|^2 d\hat{\eta}_{\Gamma}(\xi) + \iota_{\Gamma} (1 - |\hat{\nu}(\xi)|^2) \frac{d\xi}{(2\pi)^n} \, . 
\end{align*}

\begin{proposition}
\label{PropEuclideanLatticePerturbationsPreserveHyperuniformity}
Let $\Gamma < \R^n$ be a lattice and $\nu \in \Prob(\R^n)$ a probability measure. Then the perturbed lattice point process $\mu_{\Gamma, \nu}$ is hyperuniform.
\end{proposition}
\begin{proof}
By the formula for the diffraction measure we have 
\begin{align*}
\hat{\eta}_{\Gamma, \nu}(B_{\varepsilon}) = \int_{B_{\varepsilon}} |\hat{\nu}(\xi)|^2 d\hat{\eta}_{\Gamma}(\xi) + \frac{\iota_{\Gamma}}{(2\pi)^n} \int_{B_{\varepsilon}}  (1 - |\hat{\nu}(\xi)|^2) d\xi \, . 
\end{align*}
Note that $\hat{\nu}(0) = \nu(\R^n) = 1$, so $1 - |\hat{\nu}(\xi)|^2 = o(1)$ as $\xi \rightarrow 0$. Moreover, since $\mu_{\Gamma}$ is stealthy, then $\hat{\eta}_{\Gamma}|_{B_{\varepsilon}} = 0$ for sufficiently small $\varepsilon > 0$. Finally, as $\varepsilon \rightarrow 0$ we then have that
\begin{align*}
\hat{\eta}_{\Gamma, \nu}(B_{\varepsilon}) = \frac{\iota_{\Gamma}}{(2\pi)^n} \int_{B_{\varepsilon}}  (1 - |\hat{\nu}(\xi)|^2) d\xi = \frac{\iota_{\Gamma}}{(2\pi)^n} o(\varepsilon^n) \, , 
\end{align*}
so $\mu_{\Gamma, \nu}$ is spectrally hyperuniform and hence hyperuniform.
\end{proof}
\end{example}

%% file: BecksThm.tex
We will in this section investigate lower bounds for number variance $\NVmu(r)$ as $R \rightarrow +\infty$ for invariant locally square-integrable random measures $\mu$ on $\R^n$. In Subsection \ref{Proof of Theorem Theorem1A} we refine and prove Theorem \ref{Theorem1A}, and in Subsections \ref{Dynamics for the invariant random Z^n-lattice} and \ref{Proof of Theorem Theorem1B} we arrive at a proof of Equation \ref{EqVanishingoftheZ5NumberVariance}. Finally, we formulate the contents of Subsection \ref{Proof of Theorem Theorem1A} for isotropic random measures in terms of the associated powder diffraction measure in Subsection \ref{Isotropic random measures on Rn}.


\subsection{Proof of Theorem \ref{Theorem1A}}
\label{Proof of Theorem Theorem1A}

To prove Theorem \ref{Theorem1A}, we will prove a lower bound on the average of the number variance that is similar to the inequality that Beck and Chen prove in \cite[p.4]{Beck1987IrregularitiesOD}. 

\begin{theorem}
\label{ThmBeck}
Let $\mu$ be an invariant locally square-integrable random measure on Euclidean space $\R^n$. Then there is a constant $C = C(n) > 0$ such that for every $R_o > 0$,
\begin{align*}
\frac{1}{R} \int_0^R \NVmu(r) dr \geq C \Big( \int_{\norm{\xi} \geq R_o^{-1}} \frac{d\hat{\eta}_{\mu}(\xi)}{\norm{\xi}^{n+1}} \Big) R^{n-1} \, , \quad \forall \, R \geq R_o \, .
\end{align*}
\end{theorem}

\begin{proof}
From Lemma \ref{LemmaFTofIndicatorFunction} we have that 
\begin{align*}
\hat{\chi}_{B_r}(\xi) = (2\pi)^{\frac{n}{2}} \Big( \frac{r}{\norm{\xi}} \Big)^{\frac{n}{2}} J_{\frac{n}{2}}(r\norm{\xi}) \, . 
\end{align*}
By Lemma \ref{LemmaBesselIntegral} with $\alpha = n/2$ and $\beta = n$, there is for every $T \geq 1$ a $\delta_{T} \in \R$ with $\delta_T \rightarrow 0$ as $T \rightarrow +\infty$ such that
\begin{align}
\label{IndicatorFTIntegralLowerBound}
\int_0^R |\hat{\chi}_{B_r}(\xi)|^2 dr &= \frac{(2\pi)^n}{\norm{\xi}^n} \int_0^R r^{n} J_{\frac{n}{2}}(r\norm{\xi})^2 dr = \frac{b_n^2}{\norm{\xi}^{2n + 1}} \int_0^{R\norm{\xi}} t^n J_{\frac{n}{2}}(t)^2 dt \\
&= \frac{(2\pi)^n}{\norm{\xi}^{n+1}} \Big(\frac{1}{\pi n} + \delta_{R\norm{\xi}} \Big) R^{n} > 0 
\end{align}
for every $R \geq R_o$ and every $\norm{\xi} \geq R_o^{-1}$. Define
\begin{align*} 
C = (2\pi)^n \inf_{T \geq 1} \Big(\frac{1}{\pi n} + \delta_T\Big) = (2\pi)^n \inf_{T \geq 1} T^{-n} \int_0^T t^n J_{\frac{n}{2}}(t)^2 dt \, ,   
\end{align*}
which is positive since 
\begin{align*}
\int_0^T t^{n} J_{\frac{n}{2}}(t)^2 dt > 0 \quad \forall \, T \geq 1 \quad \mbox{ and } \quad \lim_{T \rightarrow +\infty} T^{-n} \int_0^T t^n J_{\frac{n}{2}}(t)^2 dt = \frac{1}{\pi n} > 0
\end{align*}
by Lemma \ref{LemmaBesselIntegral}. An application of Fubini provides the lower bound
\begin{align*}
\frac{1}{R} \int_0^R \NVmu(r) dr  &= \int_{\R^n} \Big( \frac{1}{R} \int_0^R |\hat{\chi}_{B_r}(\xi)|^2 dr \Big) d\hat{\eta}_{\mu}(\xi)\\
&\geq \int_{\norm{\xi} \geq R_o^{-1}} \Big( \frac{1}{R} \int_0^R |\hat{\chi}_{B_r}(\xi)|^2 dr \Big) d\hat{\eta}_{\mu}(\lambda) \geq R^{n-1} C \int_{\norm{\xi} \geq R_o^{-1}} \frac{d\hat{\eta}_{\mu}(\xi)}{\norm{\xi}^{n+1}} 
\end{align*}
as desired.
\end{proof}
\begin{remark}
\label{RemarkEuclideanDiffractionIntegrability}
Since $\hat{\eta}_{\mu}$ satisfies $\hat{\eta}_{\mu}(B_L) \ll_{n,\mu} L^n$ for sufficiently large $L > 0$ by Lemma \ref{LemmaEuclideanBergFrost}, then 
\begin{align*}
\int_{\norm{\xi} \geq R_o^{-1}} \frac{d\hat{\eta}_{\mu}(\xi)}{\norm{\xi}^{n+1}} \ll_n \int_{R_o^{-1}}^{\infty} \frac{\hat{\eta}_{\mu}(B_u)}{u^{n + 2}} du \ll_{n, \mu, R_o} \int_{R^{-1}_o}^{\infty} \frac{du}{u^{2}} = R_o < + \infty \, . 
\end{align*}
Thus the right hand side in Theorem \ref{ThmBeck} is indeed finite.
\end{remark}
Theorem \ref{Theorem1A} now follows form the fact that 
\begin{align*}
\limsup_{R \rightarrow +\infty} \frac{\NVmu(R)}{R^{n - 1}} \geq \limsup_{R \rightarrow +\infty} \frac{1}{R^n} \int_0^R \NVmu(r) dr \geq C \int_{\R^n} \frac{d\hat{\eta}_{\mu}(\xi)}{\norm{\xi}^{n+1}} > 0 \, ,
\end{align*}
which may or may not be infinite.

Before moving on we emphasize that if the diffraction measure vanishes in a neighbourhood around $\xi = 0$, then the limit as $R \rightarrow + \infty$ in Theorem \ref{ThmBeck} can be normalized and computed.

\begin{definition}[Euclidean stealthy random measures]
An invariant locally square-integrable random measure $\mu$ on $\R^n$ is \emph{stealthy} if there is a $\lambda_o > 0$ such that $\hat{\eta}_{\mu}(B_{\lambda_o}) = 0$.
\end{definition}
\begin{corollary}
\label{CorollaryEuclideanStealthyVarianceAverage}
Let $\mu$ be a stealthy random measure on $\R^n$. Then 
\begin{align*}
\lim_{R \rightarrow +\infty} \frac{1}{R} \int_0^R \frac{\NVmu(r)}{r^{n-1}} dr = 2^n \pi^{n-1} \int_{\R^n} \frac{d\hat{\eta}_{\mu}(\xi)}{\norm{\xi}^{n+1}}  \, . 
\end{align*}
\end{corollary}
\begin{proof}
Let $\lambda_o > 0$ such that $\hat{\eta}_{\mu}(B_{\lambda_o}) = 0$. Similarly to the proof of Theorem \ref{ThmBeck} we have 
\begin{align*}
\frac{1}{R} \int_0^R \frac{\NVmu(r)}{r^{n-1}} dr &= \int_{\norm{\xi} \geq \lambda_o} \Big( \frac{1}{R} \int_0^R \frac{|\hat{\chi}_{B_r}(\xi)|^2}{r^{n-1}} dr \Big) d\hat{\eta}_{\mu}(\xi) \\
&= (2\pi)^n \int_{\norm{\xi} \geq \lambda_o} \Big( \frac{1}{R} \int_0^R r J_{\frac{n}{2}}(r \norm{\xi})^2 dr \Big) \frac{d\hat{\eta}_{\mu}(\xi)}{\norm{\xi}^n} \, . 
\end{align*}
If $R \geq \lambda_o^{-1}$, then Lemma \ref{LemmaBesselIntegral} with $\alpha = n/2$ and $\beta = 1$ gives us 
\begin{align*}
\frac{1}{R} \int_0^R r J_{\frac{n}{2}}(r\norm{\xi})^2 dr = \frac{1}{R \norm{\xi}^2} \int_0^{R\norm{\xi}} t J_{\frac{n}{2}}(t)^2 dt \longrightarrow \frac{1}{\norm{\xi} \pi}
\end{align*}
as $R \rightarrow +\infty$. By Remark \ref{RemarkEuclideanDiffractionIntegrability} we know that $\int_{\norm{\xi} \geq \lambda_o} \norm{\xi}^{-(n+1)} d\hat{\eta}_{\mu}(\xi)$ is finite, so by dominated convergence we see that
\begin{align*}
\lim_{R \rightarrow +\infty} (2\pi)^n \int_{\norm{\xi} \geq \lambda_o} \Big( \frac{1}{R} \int_0^R r J_{\frac{n}{2}}(r\norm{\xi})^2 dr \Big) \frac{d\hat{\eta}_{\mu}(\xi)}{\norm{\xi}^n} = 2^n \pi^{n-1} \int_{\norm{\xi} \geq \lambda_o} \frac{d\hat{\eta}_{\mu}(\xi)}{\norm{\xi}^{n+1}} \, . 
\end{align*}
\end{proof}


\subsection{Dynamics for the invariant random $\Z^n$-lattice}
\label{Dynamics for the invariant random Z^n-lattice}

We consider the invariant random $\Z^n$-lattice and compute its diffraction measure. We also provide dynamical results on square roots of natural numbers that will be a key ingredient in proving the lower limit
\begin{align*}
\liminf_{R \rightarrow +\infty} \frac{\NV_{\mu_{\Z^5}}(R)}{R^4} = 0 \, . 
\end{align*}

\subsubsection{Number variance of the invariant random $\Z^n$-lattice}

Consider the standard lattice $\Z^n < \R^n$ and the random lattice $\mu_{\Z^n}$. Its dual lattice is $(\Z^n)^{\perp} = 2\pi\Z^n$, so by Example \ref{ExampleEuclideanLatticeDiffraction}, the diffraction measure $\hat{\eta}_{\Z^n}$ of the random lattice $\mu_{\Z^n}$ is 
\begin{align*}
\hat{\eta}_{\Z^n}(\hat{f}) = \sum_{\gamma \in \Z^n \backslash \{0\}} \hat{f}(2\pi\gamma)  
\end{align*}
for any $f \in \Borelbndinfty(\R^n)$. In particular, $\mu_{\Z^n}$ is stealthy. The number variance of the linear statistic $S\chi_{B_R}$ can then be written using Example \ref{ExIndicatorFT} as 
\begin{align}
\label{EqIntegerLatticeNumberVariance}
\frac{\NV_{\mu_{\Z^n}}(R)}{R^{n-1}} = \frac{1}{R^{n-1}} \sum_{\gamma \in \Z^n \backslash \{ 0 \}} |\hat{\chi}_{B_R}(2\pi\gamma)|^2 =  (2\pi)^n R  \sum_{\ell = 1}^{\infty} \frac{r_n(\ell)}{\ell^{n/2}} J_{\frac{n}{2}}(2 \pi R\sqrt{\ell})^2  
\end{align}
where $r_n(\ell) = \#\{ \gamma \in \Z^n : \norm{\gamma}^2 = \ell \}$ is the sum-of-squares function. By Corollary \ref{CorollaryEuclideanStealthyVarianceAverage},
\begin{align*}
\lim_{R \rightarrow +\infty} \frac{1}{R} \int_0^R \frac{\NV_{\mu_{\Z^n}}(r)}{r^{n-1}} dr = 2^n \pi^{n-1} \sum_{\ell = 1}^{\infty} \frac{r_n(\ell)}{\ell^{\frac{n+1}{2}}} \, . 
\end{align*}
Moreover, Remark \ref{RemarkEuclideanDiffractionIntegrability} implies that the right hand side is finite. We now show the slightly stronger statement that the map $R \mapsto R^{-(n-1)} \NV_{\mu_{\Z^n}}(R)$ is uniformly bounded in $R \geq 0$. 

\begin{lemma}
\label{LemmaLatticeDirichletSeries}
Let $n \geq 1$. Then the Dirichlet series 
\begin{align*}
\mathfrak{D}(s ; r_n) = \sum_{\ell = 1}^{\infty} \frac{r_n(\ell)}{\ell^{s}} 
\end{align*}
is absolutely convergent whenever $\Re(s) > \frac{n}{2}$.
\end{lemma}
\begin{proof}
Without loss of generality, we assume that $s$ is real. Rewriting $\mathfrak{D}(s, r_n)$ as a series over the integer lattice and partitioning into unit annuli, we get 
\begin{align*}
\mathfrak{D}(s ; r_n) = \sum_{\gamma \in \Z^n \backslash \{0\}} \frac{1}{\norm{\gamma}^{2s}} = \sum_{k = 1}^{\infty} \Big( \sum_{k \leq \norm{\gamma} < k + 1} \frac{1}{\norm{\gamma}^{2s}} \Big) \, . 
\end{align*}
An upper bound is then
\begin{align*}
\mathfrak{D}(s ; r_n) \leq \sum_{k = 1}^{\infty} \frac{\#(\Z^n \cap B_{k + 1} \backslash B_k)}{k^{2s}} \, . 
\end{align*}
The number of lattice points in a ball can be bounded from above and below using a standard box covering argument,
\begin{align*}
\frac{\pi^{n/2}}{\Gamma(\frac{n}{2} + 1)} (k - \sqrt{n})^n \leq \#(\Z^n \cap B_k) \leq  \frac{\pi^{n/2}}{\Gamma(\frac{n}{2} + 1)} (k + \sqrt{n})^n
\end{align*}
for all sufficiently large $k \geq 1$. Thus the number of lattice points in the annulus $B_{k + 1} \backslash B_k$ is bounded as 
\begin{align*}
\#(\Z^n \cap B_{k + 1} \backslash B_k) &= \#(\Z^n \cap B_{k+1}) - \#(\Z^n \cap B_k) \\
&\leq \frac{\pi^{n/2}}{\Gamma(\frac{n}{2} + 1)} \Big((k + 1 + \sqrt{n})^n - (k - \sqrt{n})^n \Big) \ll_n k^{n - 1} 
\end{align*}
for all sufficiently large $k \geq 1$, and we extend this to a uniform bound $\#(\Z^n \cap B_{k + 1} \backslash B_k) \ll_n k^{n-1}$ for all $k \geq 1$. This means that
\begin{align*}
\mathfrak{D}(s ; r_n) \ll_n \sum_{k = 1}^{\infty} \frac{1}{k^{2s - (n - 1)}} \, . 
\end{align*}
The latter series is convergent if and only if $s > \frac{n}{2}$. 
\end{proof}

\begin{corollary}
\label{CorollaryLatticeNumberVarianceUpperBound}
The number variance of the invariant random $\Z^n$-lattice satisfies 
\begin{align*}
\NV_{\mu_{\Z^n}}(R) \ll_n R^{n - 1}
\end{align*}
for all $R \geq 0$. 
\end{corollary}
\begin{proof}
By Corollary \ref{CorollaryBesselUpperBound}, there is a constant $C_{n/2} > 0$ such that $J_{\frac{n}{2}}(t) \leq C_{n/2}(1 + t)^{-1/2}$ for all $t \geq 0$. Thus 
\begin{align*}
\NV_{\mu_{\Z^n}}(R) &= (2\pi)^n  R^n  \sum_{\ell = 1}^{\infty} \frac{r_n(\ell)}{\ell^{n/2}} J_{\frac{n}{2}}(2 \pi R\sqrt{\ell})^2 \\
&\leq C_{n/2}^2 (2\pi)^n R^n \sum_{\ell = 1}^{\infty} \frac{r_n(\ell)}{\ell^{n/2}(1 + 2\pi R \sqrt{\ell})} \leq C_{n/2}^2 (2\pi)^{n-1} R^{n-1} \sum_{\ell = 1}^{\infty} \frac{r_n(\ell)}{\ell^{\frac{n+1}{2}}} \, . 
\end{align*}
By Lemma \ref{LemmaLatticeDirichletSeries}, the right hand side is convergent and so we're done. 
\end{proof}

\begin{example}[Vanishing of the number variance for random shifts of $\Z$] 
\label{ExampleDimension1Theorem4B} 
When $n = 1$ we have $r_1(\ell) = 2$ if $\ell = m^2$ for some integer $m \geq 1$ and $0$ else. Moreover, $\hat{\chi}_{[-R, R]}(2\pi m) = \frac{\sin(2\pi R m)}{\pi m} $, so 
\begin{align*}
\NV_{\mu_{\Z}}(R) =  \frac{2}{\pi^2} \sum_{m = 1}^{\infty} \frac{\sin^2(2\pi R m)}{m^2} \, . 
\end{align*}
Thus the number variance is identically zero along the sequence $R = 0, 1/2, 1, 3/2, 2, ...$ of half-integers. In dimension $n = 5$ we can not find simultaneous zeroes of the Bessel function $J_{5/2}(2\pi R \sqrt{\ell})$ for all relevant positive integers $\ell$, but we will provide a sequence of radii for which the number variance is asymptotically dominated by $R^4$. 
\end{example}

\subsubsection{Rational independence of square roots and ergodicity}

The support of the diffraction measure $\hat{\eta}_{\Z^n}$ is $\{ 2\pi\xi : \xi \in \Z^n \} \subset \R^n$, and in light of Example \ref{ExampleDimension1Theorem4B} we will show that one can scale norms $\norm{\xi_1}, ..., \norm{\xi_N}$ for any finite collection $\xi_1, ..., \xi_N \in \Z^n$ such that they are simultaneously close to the integers. We will make crucial use of the following is a classical result due to Besicovitch. 
\begin{lemma}[\cite{Besicovitch1940OnTL}, Theorem 2, p.4]
\label{LemmaBesicovitch}
Let $\ell_1, ..., \ell_J$ be distinct positive square-free integers. Then the roots $\sqrt{\ell_1}, ..., \sqrt{\ell_J}$ are linearly independent over $\Q$.
\end{lemma}
A simultaneous consideration of the fractional parts of such square roots give rise to the following uniquely ergodic dynamical system.
\begin{lemma}
\label{LemmaTranslationErgodicity}
Let $\alpha_1, ..., \alpha_J \in \R$ be linearly independent over $\Q$. Then the translation action 
\begin{align*} 
 v + \Z^J \longmapsto (v + \alpha) + \Z^J \, , \quad \alpha = (\alpha_1, ..., \alpha_J) \in \R^J
\end{align*}
on the torus $\R^J / \Z^J$ is uniquely ergodic with the Lebesgue measure as the invariant measure. In particular, every orbit of the action is dense. 
\end{lemma}

\begin{proof}
Suppose that $\theta$ is an invariant probability measure on $\R^J / \Z^J$ with respect to the prescribed action. It suffices to show that $\theta$ is the Lebesgue measure. By invariance, the Fourier transform $\hat{\theta} : \Z^J \rightarrow \C$ of $\theta$ satisfies
\begin{align*}
\hat{\theta}(\gamma) = \int_{\R^J / \Z^J} \e^{-2\pi i \langle \gamma, x \rangle} d\theta(x) = \int_{\R^J / \Z^J} \e^{-2\pi i \langle \gamma, x + \alpha \rangle} d\theta(x) = \e^{- 2\pi i \langle \gamma, \alpha \rangle} \hat{\theta}(\gamma)
\end{align*}
for all $\gamma \in \Z^J$. This forces $\hat{\theta}(\gamma) = 0$ for all $\gamma \neq 0$, since otherwise there would be a non-zero $\gamma$ for which $\langle \gamma, \alpha \rangle \in \Z$, contradicting the assumed $\Q$-linear independence of $\alpha_1, ..., \alpha_J$. Thus $\hat{\theta}$ is the unit mass on $0 \in \Z^J$ and hence $\theta$ is the Lebesgue measure on $\R^J / \Z^J$ by injectivity of the Fourier transform. 
\end{proof}

\begin{corollary}
\label{CorollarySequenceOfRootsConvergingtoZero}
Let $N \in \N$. Then there is a positive sequence $R_k \rightarrow +\infty$ such that 
\begin{align*}
R_k.(1, \sqrt{2}, ..., \sqrt{N}) + \Z^N \longrightarrow 0 + \Z^N
\end{align*}
as $k \rightarrow +\infty$ in $ \R^N / \Z^N $. 
\end{corollary}

\begin{proof}
Take distinct positive square-free integers $\ell_1, ..., \ell_J$ and finite subsets $A_1, ..., A_J \subset \N$ such that 
\begin{align}
\label{SqrtPartition}
\big\{ 1, \sqrt{2}, ..., \sqrt{N}  \big\} = \bigsqcup_{j = 1}^J \big\{ k \sqrt{\ell_j} : k \in A_j \big\} \, . 
\end{align}
By Lemma \ref{LemmaBesicovitch}, $\sqrt{\ell_1}, ..., \sqrt{\ell_J}$ are linearly independent over $\Q$, and by Lemma \ref{LemmaTranslationErgodicity} the action of translation by $\alpha = (\sqrt{\ell_1}, ..., \sqrt{\ell_J})$ on the torus $\R^J/\Z^J$ has dense orbits. Thus we may take a sequence $R_k \rightarrow +\infty$ such that $R_k.\alpha + \Z^J \rightarrow 0 + \Z^J$ in $\R^J/\Z^J$. It follows by the decomposition in Equation \eqref{SqrtPartition} that $R_k \sqrt{\ell} + \Z^N \rightarrow 0 + \Z^N$ for all $\ell = 1, ..., N$. 
\end{proof}

\subsection{Proof of Equation \ref{EqVanishingoftheZ5NumberVariance}}
\label{Proof of Theorem Theorem1B}

From Equation \eqref{EqIntegerLatticeNumberVariance} for the number variance of $\mu_{\Z^n}$, we have  
\begin{align*}
\frac{\NV_{\mu_{\Z^n}}(R)}{R^{n-1}} = (2\pi)^n  R  \sum_{\ell = 1}^{\infty} \frac{r_n(\ell)}{\ell^{n/2}} J_{\frac{n}{2}}(2 \pi R\sqrt{\ell})^2 \, . 
\end{align*}
This is uniformly bounded in $R \geq 0$ by Corollary \ref{CorollaryLatticeNumberVarianceUpperBound}, so we fix $\varepsilon > 0$ and take $N \geq 1$ to be an integer such that
\begin{align*}
R  \sum_{\ell = N + 1}^{\infty} \frac{r_n(\ell)}{\ell^{n/2}} J_{\frac{n}{2}}(2 \pi R\sqrt{\ell})^2 <  \varepsilon  \, , \quad \forall \, R \geq 0\, . 
\end{align*}
It remains to find a suitable sequence of radii such that the first $N$ terms of the series tend to $0$. To do this we make use of the approximation $M_{\frac{n}{2}}(t) = (\frac{2}{\pi t})^{1/2} \cos(t - \phi_{n/2})$ of $J_{n/2}(t)$ in Lemma \ref{LemmaBesselAsymptoticExpansion} along with the bound in Equation \eqref{EqBesselSquareError} for $R \geq 1$ to find that 
\begin{align*}
R\sum_{\ell = 1}^{N} \frac{r_n(\ell)}{\ell^{\frac{n}{2}}} J_{\frac{n}{2}}(2\pi R \sqrt{\ell})^2 \leq R &\sum_{\ell = 1}^{N} \frac{r_n(\ell)}{\ell^{\frac{n}{2}}} M_{\frac{n}{2}}(2 \pi R \sqrt{\ell})^2 \\
&+ R \sum_{\ell = 1}^{N} \frac{r_n(\ell)}{\ell^{\frac{n}{2}}} |J_{\frac{n}{2}}(2 \pi R \sqrt{\ell})^2 - M_{\frac{n}{2}}(2 \pi R\sqrt{\ell})^2| \\
\leq \frac{1}{\pi^2}&\sum_{\ell = 1}^{N} \frac{r_n(\ell)}{\ell^{\frac{n + 1}{2}}} \cos^2(2 \pi R\sqrt{\ell} - \phi_{\frac{n}{2}}) + \frac{2 \kappa_{\frac{n}{2}}}{R^{\frac{1}{2}}} \sum_{\ell = 1}^{N} \frac{r_n(\ell)}{\ell^{\frac{n}{2}}} \, ,  
\end{align*}
where again $\phi_{n/2} = \frac{n+1}{4}\pi$. Fix $\varepsilon > 0$ and take $R_o \geq 1$ such that 
\begin{align*}
\frac{2 \kappa_{\frac{n}{2}}}{R^{\frac{1}{2}}} \sum_{\ell = 1}^{N} \frac{r_n(\ell)}{\ell^{\frac{n}{2}}} < \varepsilon \quad \quad \forall \, R \geq R_o \, . 
\end{align*}
Then
\begin{align*}
\frac{\NV_{\mu_{\Z^n}}(R)}{R^{n-1}} &=  (2\pi)^n \Big( R \sum_{\ell = 1}^{N} \frac{r_n(\ell)}{\ell^{\frac{n}{2}}} J_{\frac{n}{2}}(2 \pi R \sqrt{\ell})^2 + R \sum_{\ell = N + 1}^{\infty} \frac{r_n(\ell)}{\ell^{\frac{n}{2}}} J_{\frac{n}{2}}(2 \pi R \sqrt{\ell})^2 \Big) \\
&< \frac{(2\pi)^n}{\pi^2}\sum_{\ell = 1}^{N} \frac{r_n(\ell)}{\ell^{\frac{n + 1}{2}}} \cos^2(2 \pi R\sqrt{\ell} - \phi_{\frac{n}{2}}) + 2 (2\pi)^n \varepsilon \quad \quad \forall \, R \geq R_o \, . 
\end{align*}
For the case $n = 5$ we get 
\begin{align*}
\frac{\Var_{\mu_{\Z^5}}(S\chi_{B_R})}{R^4} < 32 \pi^3 \sum_{\ell = 1}^{N} \frac{r_5(\ell)}{\ell^{3}} \sin^2(2 \pi R\sqrt{\ell}) + 64 \pi^5 \varepsilon \quad \quad \forall \, R \geq R_o \, .
\end{align*}
By Corollary \ref{CorollarySequenceOfRootsConvergingtoZero} there is an sequence $(R_k)_{k \geq 1}$ with $R_k \rightarrow +\infty$ as $k \rightarrow +\infty$ and such that $R_k \sqrt{\ell} \rightarrow 0$ in $\R^N/\Z^N$ for every $1 \leq \ell \leq N$, hence $\sin^2(R_k \sqrt{\ell}) \rightarrow 0$ for every $1 \leq \ell \leq N$. This means that
\begin{align*}
\limsup_{k \rightarrow +\infty} \frac{\NV_{\mu_{\Z^5}}(R_k)}{R_k^4}\leq 64 \pi^5 \varepsilon \, ,  
\end{align*}
and since $\varepsilon > 0$ was chosen arbitrarily we are done. 

\begin{remark}
The key feature of dimension $n = 5$ is that $\phi_{5/2} = \frac{3\pi}{2}$ is an odd multiple of $\frac{\pi}{2}$, so that $\cos^2(t - \phi_{5/2}) = \sin^2(t)$ for all $t \geq 0$. Thus the same proof works for dimensions $n = 4m + 1$, $m \geq 0$.  
\end{remark}


\subsection{Isotropic random measures on $\R^n$}
\label{Isotropic random measures on Rn}

Given a random measure $\mu$ on $\R^n$, one can require invariance under other subgroups of the Euclidean isometries rather than just translations. The special Euclidean group $G_n = K_n \ltimes \R^n$ with $K_n = \SO(n)$ is the group of rotations and translations, and contains the subgroup of translations considered in the previous subsections. Random measures that are invariant under $G_n$ are commonly referred to as \emph{isotropic}. 

Given a translation invariant random measure $\mu$ on $\R^n$, one can construct an isotropic random measure $\mu^{\mathrm{iso}}$ as the push-forward of $\mu$ along the $K_n$-averaging map $\Av_{K_n} : \Radonplus(\R^n) \rightarrow \Radonplus(\R^n)^{K_n}$ defined by
$$ \Av_{K_n}(p)(B) = \int_{K_n} p(k.B) dm_{K_n}(k) $$
for every bounded Borel set $B \subset \R^n$, where $m_{K_n}$ is the Haar probability measure on $K_n$. Note however that this specific procedure does not produce isotropic point processes from translation invariant point processes.

If $\mu$ is an isotropic and locally square-integrable random measure on $\R^n$, then the autocorrelation measure $\eta_{\mu}$ is a bi-$K_n$-invariant measure on $G_n$ which canonically descends to a radial measure on $\R^n$. Thus it is uniquely determined by its restriction to the subalgebra $\Borelbndinfty(\R^n)^{\rad} \subset \Borelbndinfty(\R^n)$ of radial functions. On this subalgebra the Fourier transform restricts to what we call the \emph{Hankel transform}. Given $f \in \Borelbndinfty(\R^n)^{\rad}$, its Hankel transform is the function $\tilde{f} : \R_{\geq 0} \rightarrow \C$ given by
\begin{align*}
\tilde{f}(\lambda) = \int_{\R^n} f(x) \omega_{\lambda}^{(n)}(x) dx \, , 
\end{align*}
where $\omega_{\lambda}^{(n)}$ are the $\SO(n)$-spherical functions from Lemma \ref{LemmaFTofIndicatorFunction},
\begin{align*}
\omega^{(n)}_{\lambda}(x) &= \frac{2^{\frac{n - 2}{2}} \Gamma(\frac{n}{2})}{(\lambda \norm{x})^{\frac{n-2}{2}}} J_{\frac{n - 2}{2}}(\lambda \norm{x}) \, . 
\end{align*}
By construction, $\hat{f}(\xi) = \tilde{f}(\norm{\xi})$ for all $\xi \in \R^n$ and dualizing this to measures, we are able to define the Hankel transform of a positive-definite Radon measure on $\R^n$. 
\begin{definition}
Let $\mu$ be an isotropic locally square-integrable random measure on $\R^n$. The \emph{powder diffraction measure} of $\mu$ is the Hankel transform $\tilde{\eta}_{\mu} \in \Radonplus(\R_{\geq 0})$ of the autocorrelation measure $\eta_{\mu}$. 
\end{definition} 
By the existence and uniqueness of the diffraction measure in Theorem \ref{ThmBochnerSchwartz}, $\tilde{\eta}_{\mu}$ exists and is unique. In particular,
\begin{align*}
\Var_{\mu}(Sf) = \frac{2 \pi^{\frac{n}{2}}}{\Gamma(\frac{n}{2})}\int_0^{\infty} |\tilde{f}(\lambda)|^2 d\tilde{\eta}_{\mu}(\lambda)  
\end{align*}
for all $f \in \Borelbndinfty(\R^n)^{\rad}$. With this notation, the lower bound in Theorem \ref{ThmBeck} is
\begin{align*}
\frac{1}{R} \int_0^R \NVmu(r) dr \geq C \Big( \int_{R_o^{-1}}^{\infty} \frac{d\tilde{\eta}_{\mu}(\lambda)}{\lambda^{n+1}} \Big) R^{n-1} \, , \quad \forall \, R \geq R_o 
\end{align*}
for $R_o > 0$ and some $C = C(n, R_o) > 0$. If $\mu$ is in addition stealthy, then the limit formula in Corollary \ref{CorollaryEuclideanStealthyVarianceAverage} becomes
\begin{align*}
\lim_{R \rightarrow +\infty} \frac{1}{R} \int_0^R \frac{\NVmu(r)}{r^{n-1}} dr = \frac{2^{n+1} \pi^{\frac{3n}{2} - 1}}{\Gamma(\frac{n}{2})} \int_{0}^{\infty} \frac{d\tilde{\eta}_{\mu}(\lambda)}{\lambda^{n+1}} \, .
\end{align*}

%% file: HyperbolicSpace.tex
We introduce real hyperbolic spaces, Cartan coordinates and invariant reference measures in Subsection \ref{Cartan decomposition and invariant reference measures} and the spherical functions in Subsection \ref{Spherical functions on H^n}. We give a careful treatment of the spherical functions, providing a trigonometric expansion in Subsection \ref{A trigonometric expansion of spherical functions} and as a consequence some bounds and an asymptotic mean in Subsection \ref{Bounds and asymptotics of spherical functions}, which are crucial ingredients in the proof of Theorem \ref{Theorem2}. Lastly, we survey some properties of the spherical transform in Subsection \ref{The spherical transform on H^n}.

\subsection{Cartan decomposition and invariant reference measures}
\label{Cartan decomposition and invariant reference measures}

Consider the bilinear form $[x, y] = x_0y_0 - x_1y_1 - ... - x_{n}y_n$ for $x, y \in \R^{1 + n}$ and the one-sheeted paraboloid $\H^n = \{ x \in \R^{1 + n} : [x, x] = 1 , \, x_0 > 0 \}$ with the hyperbolic metric $d(x, y) = \arccosh([x, y])$. On $\R^{1 + n}$, the group $G_n = \SO^{\circ}(1, n)$ of orientation preserving linear maps preserving $[\cdot, \cdot]$ restricts to a transitive isometric action on $(\H^n, d)$ and fixing the reference point $o = (1, 0, ..., 0) \in \H^n$, its stabilizer is 
$$ K_n = \Big\{ \left(\begin{array}{@{}c|c@{}}
  1 & 0\\
\hline
 0 &  k
\end{array}\right) : k \in \SO(n) \Big\} < G_n \, .  $$
With this we obtain a proper pointed metric space $(\H^n, d, o)$ as in Subsection \ref{Homogeneous metric spaces}. 

\subsubsection{Cartan decomposition}

A polar decomposition for $\H^n$ is available via the length map $\ell : G_n \rightarrow \R_{\geq 0}$, $\ell(g) = d(g.o, o)$, which is a bi-$K_n$-invariant map and descends to a bijection $K_n \backslash G_n / K_n \rightarrow \R_{\geq 0}$. If we consider the one parameter subgroup 
\begin{align*}
A_n = \Big\{ a_t = \left(\begin{array}{@{}c|c|c@{}}
  \cosh(t) & 0 & \sinh(t) \\
\hline
0  &  \Id_{\R^{n-1}} & 0 \\
\hline 
 \sinh(t) & 0 & \cosh(t) 
\end{array}\right) : t \in \R \Big\} < G_n
\end{align*}
then $\ell(a_t) = |t|$, and we obtain the \emph{Cartan decomposition} $G_n = K_n A_n^+ K_n$, where $A_n^+ = \{a_t : t \geq 0\}$. From this decomposition we get for every $R > 0$ an open subset 
$$B_R = \ell^{-1}([0, R)) = \{ k_1 a_t k_2 \in G_n : t \in [0, R) , k_1, k_2 \in K_n \} \, , $$
which we refer to as the centered open Cartan ball of radius $R > 0$. Moreover, we note that the orbit of the Cartan ball $B_R \subset G_n$ in $\H^n$ coincides with the open metric ball, $B_R.o = B_{R}(o)$. We will interchangeably use the notation $B_R$ for the centered Cartan ball in $G_n$ and the metric ball $B_R(o)$ in $\H^n$.

\subsubsection{Continuous sections}

Note that if $g \in G_n$ with Cartan decomposition $g = k_1 a_t k_2$ then 
$$g.o = k_1 a_t.o = (\cosh(t), \sinh(t)u)$$
for some $u \in \bS^{n-1} \subset \R^n$. The sphere $\bS^{n-1}$ is a homogeneous space for $K_n$, and a standard argument using the Gram-Schmidt orthogonalization process produces a continuous section $\kappa : \bS^{n-1} \rightarrow K_n$, so $\varsigma(\cosh(t), \sinh(t)u) = \kappa(u) a_t$ is a continuous section from $\H^n$ into $G_n$.  

\subsubsection{Invariant reference measures}

If we denote by $m_{K_n}$ the unique Haar probability measure on $K_n$ then a Haar measure on $G_n$ written in Cartan coordinates is
\begin{align*}
    m_{G_n}(\varphi) = \int_{K_n} \int_0^{\infty} \int_{K_n} \varphi(k_1 a_t k_2) \, dm_{K_n}(k_1)  \sinh(t)^{n-1} dt \, dm_{K_n}(k_2) \, , \quad \varphi \in \Borelbndinfty(G_n) \, , 
\end{align*}
and a $G_n$-invariant reference measure on $\H^n$ is
\begin{align*}
m_{\H^n}(f) = \int_0^{\infty} \int_{\bS^{n-1}} f(\cosh(t), \sinh(t)u) \, d\sigma_{n - 1}(u) \sinh(t)^{n-1} dt \, , \quad f \in \Borelbndinfty(\H^n) \, ,
\end{align*}
where $\sigma_{n - 1}$ is the canonical surface measure on $\bS^{n-1}$. The relation between the two measures is  
\begin{align*}
m_{\H^n}(f) = \frac{2 \pi^{\frac{n}{2}}}{\Gamma(\frac{n}{2})} \int_{G_n} f(g.o) dm_{G_n}(g)  \, . 
\end{align*}
In particular, the volume growth of centered Cartan balls in $\H^n$ is 
\begin{align}
\label{EqHyperbolicVolumeGrowth}
m_{\H^n}(B_R.o) = \frac{2 \pi^{\frac{n}{2}}}{\Gamma(\frac{n}{2})}  \int_0^R \sinh(t)^{n-1} dt \asymp_n \cosh(R)^{n-1}  
\end{align}
for all sufficiently large $R > 0$.

\subsubsection{The algebra of bi-$K_n$-invariant functions}

As in the case of isotropic random measures on Euclidean spaces we will restrict our attention to linear statistics and autocorrelation measures on the subspace of radial functions in $\Borelbndinfty(\H^n)$, which we identify with the space $\Borelbndinfty(G_n, K_n)$ of bi-$K_n$-invariant functions using the section $\varsigma$. A key feature in defining the spherical functions and the spherical transform is that the latter space defines a commutative algebra with respect to convolution.
\begin{lemma}
\label{LemmaGelfandsTrick}
Let $\eta, \nu$ be bi-$K_n$-invariant Radon measures on $G_n$ and assume that $\nu$ is finite. Then
$$ \eta * \nu = \nu * \eta \, .  $$
In particular, the subalgebra $\Borelbndinfty(G_n, K_n) \subset \Borelbndinfty(G_n)$ is a commutative involutive subalgebra.
\end{lemma}
With this property, we say that $(G_n, K_n)$ is a \emph{Gelfand pair}.

\subsection{Spherical functions on $\H^n$}
\label{Spherical functions on H^n}

A \emph{complex hyperbolic wave} will for us be a function $\xi_{\lambda, u} : \H^n \rightarrow \C$ of the form $\xi_{\lambda, u}(x) = [x, (1, u)]^{-i\lambda - \frac{n-1}{2}}$, where $u \in \bS^{n-1}$ and $\lambda \in \C$. A \emph{spherical function} for hyperbolic space $\H^n$ will, similarly to the Bessel functions appearing in the Euclidean setting, be a function $\omega_{\lambda} : G_n \rightarrow \C$ given by
\begin{align*}
\omega_{\lambda}(g) = \int_{K_n} \xi_{\lambda, u}(k^{-1}g.o) dm_{K_n}(k) = \frac{\Gamma(\frac{n}{2})}{2\pi^{\frac{n}{2}}} \int_{\bS^{n-1}} [g.o, (1, v)]^{-i\lambda - \frac{n-1}{2}} d\sigma_{n-1}(v)   
\end{align*}
for any choice of $u \in \bS^{n-1}$. To emphasize the dimension $n$, we write $\omega_{\lambda}^{(n)}$ in place of $\omega_{\lambda}$. These functions are continuous and bi-$K_n$-invariant, and can be rewritten as an integral over $\R$ by a formula that we state next. For the statement, we remind ourselves of the Beta function
\begin{align*}
\rB(z, w) = \frac{\Gamma(z)\Gamma(w)}{\Gamma(z + w)} = \int_0^1 s^{z - 1} (1 - s)^{w - 1} ds \, , \quad z, w \in \C \backslash\Z_{\leq 0} \, . 
\end{align*}

\begin{lemma}
\label{LemmaSphericalFunctionIntegralFormula}
The spherical function $\omega_{\lambda}^{(n)}$, $\lambda \in \C$, can be written as
\begin{align*}
\omega_{\lambda}^{(n)}(g) = a_n \sinh(\ell(g))^{2 - n}\int_0^{\ell(g)} (\cosh(\ell(g)) - \cosh(s))^{\frac{n-3}{2}} \cos(\lambda s) ds \, , \quad  a_n = \frac{2^{\frac{n-1}{2}}}{\rB(\frac{n-1}{2}, \frac{1}{2})} \, . 
\end{align*}
\end{lemma}
\begin{proof} 
From the Cartan decomposition, every $g \in G_n$ can be written as $g = k_1 a_{\ell(g)} k_2$ for some $k_1, k_2 \in K_n$, where $\ell(g) = d(g.o, o)$. Using that spherical functions are bi-$K_n$-invariant and introducing spherical coordinates on $\bS^{n-1}$ yields
\begin{align*}
\omega_{\lambda}^{(n)}(g) &= \omega_{\lambda}^{(n)}(a_{\ell(g)}) =  \frac{\Gamma(\frac{n}{2})}{2\pi^{\frac{n}{2}}}  \int_{\bS^{n-1}} [a_{\ell(g)}.o, (1, v)]^{-i\lambda - \frac{n-1}{2}} d\sigma_{n-1}(v) \\
&= \frac{\Gamma(\frac{n}{2})}{\pi^{\frac{1}{2}} \Gamma(\frac{n-1}{2})}  \int_{0}^{\pi} (\cosh(\ell(g)) - \sinh(\ell(g))\cos(\alpha))^{-i\lambda - \frac{n-1}{2}} \sin(\alpha)^{n-2} d\alpha \, . 
\end{align*}
Making the substitution $\e^{s} = \cosh(\ell(g)) - \sinh(\ell(g))\cos(\alpha)$ we get
\begin{align*}
\omega_{\lambda}^{(n)}(g) &= \frac{\Gamma(\frac{n}{2})}{\pi^{\frac{1}{2}} \Gamma(\frac{n-1}{2}) \sinh(\ell(g))} \int_{-\ell(g)}^{\ell(g)} \e^{-(i\lambda + \frac{n-3}{2})s} \Big( 1 - \Big( \frac{\cosh(\ell(g)) - \e^s}{\sinh(\ell(g))} \Big)^2 \Big)^{\frac{n-3}{2}} ds \\
&= \frac{\Gamma(\frac{n}{2})}{\pi^{\frac{1}{2}} \Gamma(\frac{n-1}{2}) \sinh(\ell(g))^{n-2}} \int_{-\ell(g)}^{\ell(g)} \e^{-i\lambda s} \Big( \e^{-s} \big( \sinh^2(\ell(g)) - (\cosh(\ell(g)) - \e^s )^2 \big) \Big)^{\frac{n-3}{2}} ds \, . 
\end{align*}
Note that
\begin{align*}
\e^{-s} \big( \sinh^2(\ell(g)) - (\cosh(\ell(g)) - \e^s )^2 \big) &= \e^{-s} \big( \sinh^2(\ell(g)) - \cosh^2(\ell(g))   + 2\cosh(\ell(g))\e^s - \e^{2s} \big) \\
&= 2 (\cosh(\ell(g)) - \cosh(s)) \, , 
\end{align*}
so the spherical function is
\begin{align*}
\omega_{\lambda}^{(n)}(g) &= \frac{2^{\frac{n-3}{2}}\Gamma(\frac{n}{2})}{\pi^{\frac{1}{2}} \Gamma(\frac{n-1}{2}) \sinh(\ell(g))^{n-2}} \int_{-\ell(g)}^{\ell(g)} (\cosh(\ell(g)) - \cosh(s))^{\frac{n-3}{2}} \e^{i\lambda s} ds \\
&= \frac{2^{\frac{n-1}{2}}\Gamma(\frac{n}{2})}{\pi^{\frac{1}{2}} \Gamma(\frac{n-1}{2}) \sinh(\ell(g))^{n-2}} \int_{0}^{\ell(g)}  (\cosh(\ell(g)) - \cosh(s))^{\frac{n-3}{2}} \cos(\lambda s) ds \, . 
\end{align*}
\end{proof}
The following result is a recollection of some standard useful properties of the spherical functions.
\begin{lemma}
\label{LemmaPropertiesofSphericalFunctions}
The spherical functions $\omega_{\lambda}^{(n)}$, $\lambda \in \C$, satisfy the following: 
\begin{enumerate}
\item $\omega_{\lambda}^{(n)}(e) = 1$ and $\omega_{-\lambda}^{(n)} = \omega_{\lambda}^{(n)}$ for all $\lambda \in \C$.
\item $\omega_{\lambda}^{(n)}$ is bounded if and only if $|\Im(\lambda)| \leq \frac{n-1}{2}$. Moreover, if $|\Im(\lambda)| < \frac{n-1}{2}$ we have that 
$$|\omega_{\lambda}^{(n)}(g)| <  1 $$
for all $g \neq e$ in $G_n$.
\item If $g \neq e$ and $\lambda \geq 0$ then $|\omega_{\lambda}^{(n)}(g)| \leq \omega_{0}^{(n)}(g)$.
\end{enumerate}
\end{lemma}

\begin{proof}
For (1) we have that 
\begin{align*}
\omega_{\lambda}^{(n)}(e) = \frac{\Gamma(\frac{n}{2})}{2\pi^{\frac{n}{2}}} \int_{\bS^{n-1}} [o, (1, v)]^{-i\lambda - \frac{n-1}{2}} d\sigma_{n-1}(v) = 1 \, ,
\end{align*}
and by Lemma \ref{LemmaSphericalFunctionIntegralFormula} it is easy to see that $\omega_{-\lambda}^{(n)} = \omega_{\lambda}^{(n)}$ for all $\lambda \in \C$ using that $\cos(-\lambda s) = \cos(\lambda s)$.

For (2), the first claim is obvious and for the second one we first assume that $|\Im(\lambda)| \leq \frac{n-1}{2}$ and use Lemma \ref{LemmaSphericalFunctionIntegralFormula} to find that
\begin{align*}
|\omega_{\lambda}^{(n)}(g)| &\leq \frac{2^{\frac{n-1}{2}}\Gamma(\frac{n}{2})}{\pi^{\frac{1}{2}} \Gamma(\frac{n-1}{2}) \sinh(\ell(g))^{n-2}} \int_{0}^{\ell(g)}  (\cosh(\ell(g)) - \cosh(s))^{\frac{n-3}{2}} \cosh(|\Im(\lambda)| s) ds \\
&\leq \frac{2^{\frac{n-1}{2}}\Gamma(\frac{n}{2})}{\pi^{\frac{1}{2}} \Gamma(\frac{n-1}{2}) \sinh(\ell(g))^{n-2}} \int_{0}^{\ell(g)}  (\cosh(\ell(g)) - \cosh(s))^{\frac{n-3}{2}} \cosh\Big(\frac{n-1}{2} s\Big) ds \\
&= \omega_{i \frac{n-1}{2}}^{(n)}(g) = 1
\end{align*}
for all $g \in G_n$. Moreover, it is not hard to see that the last inequality is strict if $|\Im(\lambda)| < \frac{n-1}{2}$ and $g \neq e$, so the second statement also holds.

(3) follows from a similar bound as in the proof of (2), using that $|\cos(\lambda s)| \leq 1$ for all $s \geq 0$.
\end{proof}

\subsection{A trigonometric expansion of spherical functions}
\label{A trigonometric expansion of spherical functions}

The following Proposition can be viewed as a alternate formulation of the classical asymptotic expansions of spherical functions due to Harish-Chandra, see \cite[Theorem 5.5, p.430]{HelgasonGeometricGroups}. For us however, we are in need of an expansion along the line $\lambda \geq 0$, which differs slightly from the referenced statement. We emphasize that our proof will rely entirely on elementary arguments.   

\begin{proposition}
\label{PropTrigonometricExpansionofSphericalFunctions}
Let $\lambda > 0$. Then there is an $r_o > 0$ and bounded continuous functions 
$$\alpha_n(\cdot, \lambda), \beta_n(\cdot, \lambda) : [r_o, +\infty) \rightarrow \R$$
such that
\begin{align*}
\omega_{\lambda}^{(n)}(a_r) = a_n \sinh(r)^{-\frac{n-1}{2}} \big(\alpha_n(r, \lambda) \cos(\lambda r) + \beta_n(r, \lambda) \sin(\lambda r) \big) \, ,
\end{align*}
with
\begin{align*}
\lim_{r \rightarrow +\infty} \alpha_n(r, \lambda) = \Re\Big(\rB\Big(i\lambda, \frac{n-1}{2}\Big)\Big) \quad \mbox{ and } \quad \lim_{r \rightarrow +\infty} \beta_n(r, \lambda) = \Im\Big(\rB\Big(i\lambda, \frac{n-1}{2}\Big)\Big) \, . 
\end{align*}
\end{proposition}

\subsubsection{Some identities for the Beta function}

For the proof of Proposition \ref{PropTrigonometricExpansionofSphericalFunctions} in even dimensions, it will be useful to consider the the power series expansion
\begin{align}
\label{EqExpansionofOneMinusXtotheMinusOneHalf}
(1 - x)^{-1/2} = \sum_{\ell = 0}^{\infty} E_{\ell} x^{\ell} \, , \quad E_{\ell} = \frac{(2\ell)!}{4^{\ell} (\ell!)^2} 
\end{align}
for $0 \leq x < 1$. Using Stirling approximation, one gets the asymptotic $E_{\ell} \sim \sqrt{\frac{2}{\pi \ell}}$ as $\ell \rightarrow +\infty$.
\begin{lemma}
\label{LemmaBetaFunctionFormulas}
Let $z \in \C \backslash \{0\}$ such that $\Re(z) \geq 0$ and $m \in \Z_{\geq 0}$. Then
\begin{align*}
\rB(z, m + 1) = \sum_{k = 0}^m \binom{m}{k} \frac{(-1)^k}{k + z} \quad \mbox{ and } \quad  \rB(z, m + 1/2) = \sum_{\ell = 0}^{\infty} E_{\ell} \, \rB(z + \ell, m + 1) \, . 
\end{align*}
\end{lemma}

\begin{proof}
The first identity follows from binomial expansion, 
\begin{align*}
\rB(z, m + 1) = \int_0^1 x^{z - 1} (1 - x)^m dx = \sum_{k = 0}^m \binom{m}{k} (-1)^k \int_0^1 x^{z + k - 1} dx = \sum_{k = 0}^m \binom{m}{k} \frac{(-1)^k}{k + z} \, . 
\end{align*}
For the second identity, we first note that
\begin{align*} 
|\rB(z + \ell, m + 1)| \leq \rB(\Re(z) + \ell, m + 1) \leq \rB(\ell, m + 1) = \frac{m!(\ell - 1)!}{(\ell + m)!} \, , 
\end{align*}
so
\begin{align*}
\sum_{\ell = 1}^{\infty} E_{\ell}|\rB(z + \ell, m + 1)| &\leq \sum_{\ell = 1}^{\infty} E_{\ell} \rB(\ell, m + 1) \\
&= m! \sum_{\ell = 1}^{\infty} E_{\ell} \frac{(\ell - 1)!}{(\ell + m)!} \leq m! \sum_{\ell = 1}^{\infty} \frac{E_{\ell}}{\ell^{m + 1}} < +\infty \, . 
\end{align*}
By dominated convergence, we get
\begin{align*}
\sum_{\ell = 1}^{\infty} E_{\ell} \, \rB(z + \ell, m + 1) &= \int_0^1 x^{z - 1} \Big( \sum_{\ell = 1}^{\infty} E_{\ell} x^{\ell} \Big) (1 - x)^{m} dx \\
&=\int_0^1 x^{z - 1} \big( (1 - x)^{-1/2} - 1 \big) (1 - x)^m dx \\
&= \rB(z, m + 1/2) - \rB(z, m + 1) \, . 
\end{align*}
Rearranging this, we finally have that 
\begin{align*}
\rB(z, m + 1/2) = \sum_{\ell = 0}^{\infty} E_{\ell} \, \rB(z + \ell, m + 1) \, . 
\end{align*}

\end{proof}

\subsubsection{Some formulas for intermediate coefficients}

\begin{lemma}
\label{LemmaCoshPowerIntegral}
For every $q \in \N$ and all $r \geq 0$,
\begin{align*}
\int_0^r \cosh(s)^q ds \leq 2^{q + \frac{1}{2}} + \frac{2 \cosh(r)^q}{q}   \, . 
\end{align*}
\end{lemma}
\begin{proof}
Setting $u = \cosh(s)$ we split our integral as 
\begin{align*}
\int_0^r \cosh(s)^q ds = \int_1^{\cosh(r)} \frac{u^q}{\sqrt{u^2 - 1}} du = \int_1^{2} \frac{u^q}{\sqrt{u^2 - 1}} du + \int_2^{\cosh(r)} \frac{u^q}{\sqrt{u^2 - 1}} du  \, . 
\end{align*}
For the first integral we make the estimate
\begin{align*}
\int_1^{2} \frac{u^q}{\sqrt{u^2 - 1}} du &= \int_1^2 \frac{u^q}{\sqrt{(u + 1)(u - 1)}} du \leq 2^{q - \frac{1}{2}} \int_1^2 \frac{du}{\sqrt{u - 1}} = 2^{q + \frac{1}{2}} \, . 
\end{align*}
For the second integral, we have that
\begin{align*}
\int_2^{\cosh(r)} \frac{u^q}{\sqrt{u^2 - 1}} du &= \int_2^{\cosh(r)} \frac{u^q}{\sqrt{(u + 1)(u - 1)}} du \leq \int_2^{\cosh(r)} \frac{u^q}{u - 1} du \\
&= \int_2^{\cosh(r)} \Big( 1 + \frac{1}{u-1} \Big) u^{q - 1} du \leq 2 \int_2^{\cosh(r)} u^{q - 1} du \leq \frac{2 \cosh(r)^q}{q} \, . 
\end{align*}
Gathering the two estimates we get 
\begin{align*}
\int_0^r \cosh(s)^q ds = \int_1^{2} \frac{u^q}{\sqrt{u^2 - 1}} du + \int_2^{\cosh(r)} \frac{u^q}{\sqrt{u^2 - 1}} du \leq 2^{q + \frac{1}{2}} + \frac{2 \cosh(r)^q}{q} \, . 
\end{align*}
\end{proof}

\begin{lemma}
\label{LemmaHyperbolicBeckab}
For every $k \in \Z_{\geq 0}$, $r \geq 0$, $\lambda > 0$, 
\begin{align}
\label{EqCoshIntegralExpansionEquationab}
\int_0^r \cosh(s)^k \cos(\lambda s) ds = a_k(r, \lambda) \cos(\lambda r) + b_k(r, \lambda) \sin(\lambda r) \, ,
\end{align}
where 
\begin{align*}
a_k(r, \lambda) = \frac{1}{2^{k}} \sum_{j = 0}^k \binom{k}{j} \frac{(2j - k) \sinh((2j - k)r)}{(2j - k)^2 + \lambda^2} \, , \,\, b_k(r, \lambda) = \frac{\lambda}{2^{k}} \sum_{j = 0}^k \binom{k}{j} \frac{\cosh((2j - k)r)}{(2j - k)^2 + \lambda^2} \, .
\end{align*}
Moreover, if $r > \arctanh(\frac{1}{2})$ these coefficients satisfy
\begin{align*}
a_k(r, \lambda)  \leq  \int_0^r \cosh(s)^k ds \, , \quad b_k(r, \lambda) \leq  \frac{1}{\lambda} + 2\lambda \int_0^r \cosh(s)^k ds   \, , 
\end{align*}
and
\begin{align*}
\lim_{r \rightarrow +\infty} \frac{a_k(r, \lambda)}{\cosh(r)^k} = \frac{k}{k^2 + \lambda^2} \, , \quad \lim_{r \rightarrow +\infty} \frac{b_k(r, \lambda)}{\cosh(r)^k} = \frac{\lambda}{k^2 + \lambda^2} \, . 
\end{align*}
\end{lemma}

\begin{proof}
A binomial expansion of $\cosh(s)^k$ yields
\begin{align*}
\int_0^r \cosh(s)^k \cos(\lambda s) ds = \frac{1}{2}\int_{-r}^r \cosh(s)^k \e^{i \lambda s} ds = \frac{1}{2^{k + 1}} \sum_{j = 0}^k \binom{k}{j} \int_{-r}^{r} \e^{(2j - k + i\lambda)s} ds \, .
\end{align*}
The latter integral is
\begin{align*}
\int_{-r}^{r} \e^{(2j - k + i\lambda)s} ds &= 2 \, \Re \Big( \frac{\sinh((2j - k + i\lambda)r)}{2j - k + i\lambda} \Big) \\
&= 2 \, \frac{(2j - k)\sinh((2j - k)r)\cos(\lambda r) + \lambda \cosh((2k - j)r) \sin(\lambda r)}{(2j - k)^2 + \lambda^2}
\end{align*}
which means that
\begin{align*}
\int_0^r \cosh(s)^k \cos(\lambda s) ds = &\Big(\frac{1}{2^{k}} \sum_{j = 0}^k \binom{k}{j} \frac{(2j - k) \sinh((2j - k)r)}{(2j - k)^2 + \lambda^2}  \Big) \cos(\lambda r) \\
&+ \Big( \frac{\lambda}{2^{k}} \sum_{j = 0}^k \binom{k}{j} \frac{\cosh((2j - k)r)}{(2j - k)^2 + \lambda^2}  \Big) \sin( \lambda r)
\end{align*}
as desired. 

For the upper bound of $a_k(r, \lambda)$, we have
\begin{align*}
a_k(r, \lambda) &= \frac{1}{2^{k}} \sum_{j = 0}^k \binom{k}{j} \frac{(2j - k) \sinh((2j - k)r)}{(2j - k)^2 + \lambda^2} \\
&= \frac{1}{2^{k}} \sum_{j = 0}^k \binom{k}{j} \frac{(2j - k)^2}{(2j - k)^2 + \lambda^2} \frac{\sinh((2j - k)r)}{2j - k} \\
&\leq \frac{1}{2^k} \sum_{j = 0}^k \binom{k}{j} \frac{\sinh((2j - k)r)}{2j - k} = \int_0^r \cosh(s)^k ds \, .
\end{align*}
For $b_k(r, \lambda)$, we isolate the worst case of $2j - k = 0$ by
\begin{align*}
b_k(r, \lambda) &= \frac{\lambda}{2^{k}} \sum_{j = 0}^k \binom{k}{j} \frac{\cosh((2j - k)r)}{(2j - k)^2 + \lambda^2} \leq \frac{1}{\lambda} + \frac{\lambda}{2^k} \sum_{2j - k \neq 0} \binom{k}{j} \frac{\cosh((2j - k)r)}{(2j - k)^2 + \lambda^2} \, . 
\end{align*}
In the above inequality, we used that $\binom{k}{j} \leq 2^k$ for all $j = 0, ..., k$ when isolating the term corresponding to $2j - k = 0$. Next we observe that 
$$ \frac{1}{2|2j - k|} \leq \frac{1}{2} \leq \tanh(r) \leq \tanh(|2j - k|r) $$
for all $j$ such that $2j - k \neq 0$ and all $r \geq \arctanh(1/2)$, so that
\begin{align*}
\cosh((2j - k)r) \leq 2 (2j - k) \sinh((2j - k)r)
\end{align*}
for the same $j$ and $r$. Thus
\begin{align*}
\frac{1}{2^k} \sum_{2j - k \neq 0} \binom{k}{j} \frac{\cosh((2j - k)r)}{(2j - k)^2 + \lambda^2} \leq  \frac{2}{2^k} \sum_{j = 0}^k \binom{k}{j} \frac{(2j - k)\sinh((2j - k)r)}{(2j - k)^2 + \lambda^2} \\
\leq  2 a_k(r, \lambda) \leq 2 \int_0^r \cosh(s)^k ds \, , 
\end{align*}
and so
\begin{align*}
b_k(r, \lambda) &\leq \frac{1}{\lambda} + \frac{\lambda}{2^k} \sum_{2j - k \neq 0} \binom{k}{j} \frac{\cosh((2j - k)r)}{(2j - k)^2 + \lambda^2} \leq \frac{1}{\lambda} + 2\lambda \int_0^r \cosh(s)^k ds 
\end{align*}
for all $r > \arctanh(1/2)$. Lastly, the limits in the statement follow easily. 
\end{proof}
\begin{lemma}
\label{LemmaHyperbolicBeckAB}
Suppose that $m \in \N$, $\ell \in \Z_{\geq 0}$. Then 
\begin{align*}
 \int_0^r \Big( \frac{\cosh(s)}{\cosh(r)} \Big)^{\ell}  \Big( 1 - \frac{\cosh(s)}{\cosh(r)} \Big)^{m-1}  \cos(\lambda s) ds = A_{m, \ell}(r, \lambda) \cos(\lambda r) + B_{m, \ell}(r, \lambda) \sin(\lambda r) \, , 
\end{align*}
where 
\begin{align*}
A_{m, \ell}(r, \lambda) = \sum_{k = 0}^{m-1} \binom{m-1}{k} \frac{(-1)^k a_{k + \ell}(r, \lambda)}{\cosh(r)^{k + \ell}} \, , \quad B_{m, \ell}(r, \lambda) = \sum_{k = 0}^{m-1} \binom{m-1}{k}\frac{(-1)^k b_{k + \ell}(r, \lambda)}{\cosh(r)^{k + \ell}} \, . 
\end{align*}
Moreover, for all $r > \arccosh(2)$ and all $\ell \geq 1$,
\begin{align*}
|A_{m, \ell}(r, \lambda)| \leq \frac{2^{m+1}}{\ell} \quad  \mbox{ and } \quad  |B_{m, \ell}(r, \lambda)| \leq \big( \frac{1}{\lambda} + 8\lambda \big) \frac{2^{m-1}}{\ell} 
\end{align*}
and for all $\ell \geq 0$,
\begin{align*}
 \lim_{r \rightarrow +\infty} A_{m, \ell}(r, \lambda) &= \sum_{k = 0}^{m-1} \binom{m-1}{k} \frac{(-1)^k (k + \ell)}{(k + \ell)^2 + \lambda^2} = \Re(\rB(i\lambda, m)) \, , \\
 \lim_{r \rightarrow +\infty} B_{m, \ell}(r, \lambda) &= \lambda \sum_{k = 0}^{m-1} \binom{m-1}{k} \frac{(-1)^k }{(k + \ell)^2 + \lambda^2} = \Im(\rB(i\lambda, m)) \, . 
\end{align*}
\begin{proof}[Proof of Lemma \ref{LemmaHyperbolicBeckAB}]
Using the binomial Theorem we can write 
\begin{align*}
\int_0^r \Big( \frac{\cosh(s)}{\cosh(r)} \Big)^{\ell}  &\Big( 1 - \frac{\cosh(s)}{\cosh(r)} \Big)^{m-1}  \cos(\lambda s) ds = \\
&= \sum_{k = 0}^{m-1} \binom{m-1}{k} \frac{(-1)^k}{\cosh(r)^{k + \ell}} \int_0^r \cosh(s)^{k + \ell} \cos(\lambda s) ds \, , 
\end{align*}
which by Lemma \ref{LemmaHyperbolicBeckab} is 
\begin{align*}
\sum_{k = 0}^{m-1} &\binom{m-1}{k} \frac{(-1)^k}{\cosh(r)^{k + \ell}}(a_{k + \ell}(r, \lambda) \cos(\lambda r) + b_{k + \ell}(r, \lambda) \sin(\lambda r) ) = \\
= \Big(\sum_{k = 0}^{m-1} &\binom{m-1}{k} \frac{(-1)^ka_{k + \ell}(r, \lambda)}{\cosh(r)^{k + \ell}} \Big) \cos(\lambda r)  + \Big(\sum_{k = 0}^{m-1} \binom{m-1}{k} \frac{(-1)^kb_{k + \ell}(r, \lambda)}{\cosh(r)^{k + \ell}} \Big) \sin(\lambda r)  \, . 
\end{align*}
For the upper bounds, we use the upper bounds in Lemma \ref{LemmaHyperbolicBeckab} for $a_k(r, \lambda)$ and $b_k(r, \lambda)$, yielding
\begin{align*}
|A_{m, \ell}(r, \lambda)| \leq \sum_{k = 0}^{m - 1} \binom{m-1}{k} \frac{a_{k + \ell}(r, \lambda)}{\cosh(r)^{k + \ell}} \leq \sum_{k = 0}^{m - 1} \binom{m-1}{k} \frac{1}{\cosh(r)^{k + \ell}} \int_0^r \cosh(s)^{k + \ell} ds \, . 
\end{align*}
Since $r > \arccosh(2)$, Lemma \ref{LemmaCoshPowerIntegral} gives us 
\begin{align*}
\int_0^r \cosh(s)^{k + \ell} ds \leq  \sqrt{2} \Big(\frac{2}{\cosh(r)}\Big)^{k + \ell} + \frac{2}{k + \ell} \leq \frac{4}{\ell}  
\end{align*}
for all $k \geq 0$ and all $\ell \geq 1$, so
\begin{align*}
\sum_{k = 0}^{m - 1} \binom{m-1}{k} \frac{1}{\cosh(r)^{k + \ell}} \int_0^r \cosh(s)^{k + \ell} ds &\leq \frac{4}{\ell} \sum_{k = 0}^{m - 1} \binom{m-1}{k} = \frac{2^{m + 1}}{\ell} \, . 
\end{align*}
Similarly,
\begin{align*}
|B_{m, \ell}(r, \lambda)| &\leq \sum_{k = 0}^{m - 1} \binom{m-1}{k} \frac{b_{k + \ell}(r, \lambda)}{\cosh(r)^{k + \ell}} \\
&\leq \sum_{k = 0}^{m - 1} \binom{m-1}{k} \frac{1}{\cosh(r)^{k + \ell}} \Big( \frac{1}{\lambda} + 2\lambda \int_0^r \cosh(s)^{k + \ell} ds \Big) \\
&\leq \sum_{k = 0}^{m - 1} \binom{m-1}{k}  \Big( \frac{1}{\lambda \cosh(r)^{k + \ell}} + \frac{8\lambda}{\ell} \Big) \\
&\leq 2^{m - 1} \Big( \frac{1}{\lambda \cosh(r)^{\ell}} + \frac{8\lambda}{\ell} \Big) \leq \Big( \frac{1}{\lambda} + 8\lambda \Big) \frac{2^{m-1}}{\ell} \, . 
\end{align*}
The limits are a straightforward consequence of Lemma \ref{LemmaBetaFunctionFormulas} and Lemma \ref{LemmaHyperbolicBeckab}. 
\end{proof}
\end{lemma}

\begin{corollary}
\label{CorollaryABAbsoluteConvergence}
Let $r \geq 0$ and $\lambda > 0$. Then the series 
\begin{align*}
\sum_{\ell = 0}^{\infty} E_{\ell} A_{m, \ell}(r, \lambda) \quad \mbox{ and } \quad  \sum_{\ell = 0}^{\infty} E_{\ell} B_{m, \ell}(r, \lambda)
\end{align*}
are absolutely convergent. 
\end{corollary}
\begin{proof}
Using the upper bounds provided in Lemma \ref{LemmaHyperbolicBeckAB}, we get that
\begin{align*}
\sum_{\ell = 0}^{\infty} E_{\ell} |A_{m, \ell}(r, \lambda)| \leq 2^{m + 1} \sum_{\ell = 0}^{\infty} \frac{E_{\ell}}{\ell} \quad \mbox{ and } \quad  \sum_{\ell = 0}^{\infty} E_{\ell} |B_{m, \ell}(r, \lambda)| \leq 2^{m - 1} \Big( \frac{1}{\lambda} + 8\lambda \Big) \sum_{\ell = 0}^{\infty} \frac{E_{\ell}}{\ell} \, . 
\end{align*}
From the asymptotic $E_{\ell} \sim \sqrt{\frac{2}{\pi \ell}}$ as $\ell \rightarrow +\infty$, we see that $\sum_{\ell \geq 0} E_{\ell}/\ell$ is convergent.
\end{proof}

\subsubsection{A Proof of Proposition \ref{PropTrigonometricExpansionofSphericalFunctions} for odd dimensions}
Let $n = 2m + 1$, $m \in \N$. We use the formula in Lemma \ref{LemmaSphericalFunctionIntegralFormula}, binomial expansion and Lemma \ref{LemmaHyperbolicBeckAB} to write the spherical function $\omega_{\lambda}^{(2m + 1)}$ as 
\begin{equation}
\begin{split}
\label{EqSphericalFunctionOddExpansion}
\omega^{(2m + 1)}_{\lambda}(a_r) &= \frac{a_{2m + 1}}{\sinh(r)^{2m - 1}} \int_0^r (\cosh(r) - \cosh(s))^{m - 1} \cos(\lambda s) ds \\ 
&= \frac{a_{2m + 1}}{\sinh(r)^{m}} \coth(r)^{m - 1} \int_0^r \Big(1 - \frac{\cosh(s)}{\cosh(r)} \Big)^{m - 1} \cos(\lambda s) ds \\
&= \frac{a_{2m + 1}}{\sinh(r)^{m}} \coth(r)^{m - 1} \big( A_{m, 0}(r, \lambda) \cos(\lambda r) + B_{m, 0}(r, \lambda) \sin(\lambda r) \big) \, . 
\end{split}
\end{equation}
If we define 
\begin{align*} 
\alpha_{2m + 1}(r, \lambda) = \coth(r)^{m - 1} A_{m, 0}(r, \lambda) \quad \mbox{ and } \quad \beta_{2m + 1}(r, \lambda) = \coth(r)^{m - 1} B_{m, 0}(r, \lambda)
\end{align*}
then 
$$\omega_{\lambda}^{(2m + 1)}(a_r) = a_n \sinh(r)^{-\frac{n-1}{2}} \big(\alpha_{2m + 1}(r, \lambda) \cos(\lambda r) + \beta_{2m + 1}(r, \lambda) \sin(\lambda r) \big) \, , $$
and Lemma \ref{LemmaHyperbolicBeckAB} tells us that
\begin{align*}
\lim_{r \rightarrow +\infty}\alpha_{2m + 1}(r, \lambda) = \Re(\rB(i\lambda, m)) \quad \mbox{ and } \quad \lim_{r \rightarrow +\infty} \beta_{2m + 1}(r, \lambda) = \Im(\rB(i\lambda, m)) \, .
\end{align*}

\subsubsection{A Proof of Proposition \ref{PropTrigonometricExpansionofSphericalFunctions} for even dimensions}
Let $n = 2m$, $m \in \N$. Using the Taylor expansion
\begin{align*}
(1 - x)^{-\frac{1}{2}} = \sum_{\ell = 0}^{\infty} E_{\ell} x^{\ell}
\end{align*}
for $0 \leq x < 1$, we rewrite the spherical function $\omega_{\lambda}^{(2m)}$ in the same way as in the odd dimensional case, 
\begin{equation}
\begin{split}
\label{EqSphericalFunctionEvenExpansion}
\omega^{(2m)}_{\lambda}(a_r) &= \frac{a_{2m}}{\sinh(r)^{2m - 2}} \int_0^r (\cosh(r) - \cosh(s))^{m - \frac{3}{2}} \cos(\lambda s) ds \\ 
&= \frac{a_{2m }}{\sinh(r)^{m - \frac{1}{2}}} \coth(r)^{m - \frac{3}{2}} \int_0^r \Big(1 - \frac{\cosh(s)}{\cosh(r)} \Big)^{m - \frac{3}{2}} \cos(\lambda s) ds \\
&= \frac{a_{2m }}{\sinh(r)^{m - \frac{1}{2}}} \coth(r)^{m - \frac{3}{2}} \int_0^r  \sum_{\ell = 0}^{\infty} E_{\ell} \Big( \frac{\cosh(s)}{\cosh(r)} \Big)^{\ell} \Big(1 - \frac{\cosh(s)}{\cosh(r)} \Big)^{m - 1} \cos(\lambda s) ds \, . 
\end{split}
\end{equation}
By Lemma \ref{LemmaHyperbolicBeckAB} and Corollary \ref{CorollaryABAbsoluteConvergence}, we know that 
\begin{align*}
\sum_{\ell = 0}^{\infty} E_{\ell} \int_0^r \Big( \frac{\cosh(s)}{\cosh(r)} \Big)^{\ell} &\Big(1 - \frac{\cosh(s)}{\cosh(r)} \Big)^{m - 1} \cos(\lambda s) ds = \\
&= \Big( \sum_{\ell = 0}^{\infty} E_{\ell} A_{m, \ell}(r, \lambda) \Big)\cos(\lambda r) + \Big( \sum_{\ell = 0}^{\infty} E_{\ell} B_{m, \ell}(r, \lambda) \Big) \sin(\lambda r) 
\end{align*}
is absolutely convergent. We may then apply Fubini's Theorem to write the spherical function as 
\begin{align*}
%
\omega^{(2m)}_{\lambda}(a_r) &= \frac{a_{2m }}{\sinh(r)^{m - \frac{1}{2}}} \coth(r)^{m - \frac{3}{2}} \sum_{\ell = 0}^{\infty} E_{\ell} \int_0^r \Big( \frac{\cosh(s)}{\cosh(r)} \Big)^{\ell} \Big(1 - \frac{\cosh(s)}{\cosh(r)} \Big)^{m - 1} \cos(\lambda s) ds \\
&= \frac{a_{2m }}{\sinh(r)^{m - \frac{1}{2}}} \coth(r)^{m - \frac{3}{2}}  \Big( \Big( \sum_{\ell = 0}^{\infty} E_{\ell} A_{m, \ell}(r, \lambda) \Big)\cos(\lambda r) \\
&\quad\quad\quad\quad\quad\quad\quad\quad\quad\quad\quad\quad\quad+ \Big( \sum_{\ell = 0}^{\infty} E_{\ell} B_{m, \ell}(r, \lambda) \Big) \sin(\lambda r) \Big) \, . 
\end{align*}
Thus we are inclined to define 
\begin{align*}
\alpha_{2m}(r, \lambda) = \coth(r)^{m - \frac{3}{2}}   \sum_{\ell = 0}^{\infty} E_{\ell} A_{m, \ell}(r, \lambda) \quad \mbox{ and } \quad \beta_{2m}(r, \lambda) = \coth(r)^{m - \frac{3}{2}}   \sum_{\ell = 0}^{\infty} E_{\ell} B_{m, \ell}(r, \lambda) \, . 
\end{align*}
Since these series are absolutely convergent for all $r > \arccosh(2)$ by Corollary \ref{CorollaryABAbsoluteConvergence}, we can compute their term-wise limit and use Lemma \ref{LemmaBetaFunctionFormulas} to see that
\begin{align*}
\lim_{r \rightarrow +\infty} \alpha_{2m}(r, \lambda) = \lim_{r \rightarrow +\infty} \sum_{\ell = 0}^{\infty} E_{\ell}  A_{m, \ell}(r, \lambda) = \sum_{\ell = 0}^{\infty} E_{\ell}  \Re(\rB(i\lambda + \ell, m)) =  \Re(\rB(i\lambda, m - 1/2))
\end{align*}
and 
\begin{align*}
\lim_{r \rightarrow +\infty} \beta_{2m}(r, \lambda) = \lim_{r \rightarrow +\infty} \sum_{\ell = 0}^{\infty} E_{\ell}  B_{m, \ell}(r, \lambda) =  \sum_{\ell = 0}^{\infty} E_{\ell}  \Im(\rB(i\lambda + \ell, m)) = \Im(\rB(i\lambda, m - 1/2)) \, . 
\end{align*}
This concludes the proof of Proposition \ref{PropTrigonometricExpansionofSphericalFunctions}.

\begin{remark}
The spherical functions $\omega_{\lambda}^{(n)}$ on real hyperbolic space $\H^n$ are special cases of \emph{Jacobi functions}.
For real rank one symmetric spaces, the spherical functions are essentially all Jacobi functions, admitting integral formulas similar to that of Lemma \ref{LemmaSphericalFunctionIntegralFormula}, see \cite[Eq. 5.10, p.40]{Koornwinder}. The explicit integral formulas for $n$-dimensional hyperbolic space over the real, complex and quarternionic field respectively are given by 
$$ \omega_{\lambda}^{(n)}(g) = \Vol(\bS^{n-1})^{-1} \int_{\bS^{n-1}} |[a_{\ell(g)}.o, (1, v)]|^{-i\lambda - \rho} d\sigma(v) $$
for $t \geq 0$, where $\bS^{n-1}$ is the unit sphere in the hyperbolic space with canonical surface measure $\sigma$ and $\rho$ is the half sum of positive roots with respect to a fixed root system. For the octonionic case, $n = 2$ and one must make some modifications for the integral formula, but we do not get into any details about this here. We expect that the proof of Proposition \ref{PropTrigonometricExpansionofSphericalFunctions} can for the complex and quarternionic cases be carried out in a similar fashion for such real rank one symmetric spaces.
\end{remark}


\subsection{Bounds and asymptotics of spherical functions}
\label{Bounds and asymptotics of spherical functions}
We will now derive the main asymptotic bounds of spherical functions that will lead us to the proof of Theorem \ref{Theorem2} in the later sections.
\begin{proposition}
\label{PropSphericalAverage}
The spherical functions satisfy the following:
\begin{enumerate}
\item For every $r > 0$ and every $\lambda > 0$,
\begin{align*}
|\omega_{\lambda}^{(n)}(a_r)| \leq \frac{a_n}{\lambda \sinh(r)^{\frac{n - 1}{2}}} \, . 
\end{align*}
\item For every $\lambda > 0$ we have
\begin{align*}
\lim_{R \rightarrow +\infty} \frac{1}{R} \int_0^R |\omega^{(n)}_{\lambda}(a_r)|^2 \sinh(r)^{n - 1} dr = 2^{n - 2}  \frac{ |\rB(i\lambda, \frac{n-1}{2})|^2}{\rB(\frac{1}{2}, \frac{n-1}{2})^2} \, . 
\end{align*}
In particular, the convergence is uniform in $\lambda \geq \lambda_o$ for every $\lambda_o > 0$. 
\item Moreover, if $\lambda \in [0, \frac{n-1}{2}]$ and $r > 1$ then there is a constant $C = C(n) > 0$ such that
\begin{align*}
\omega_{i \lambda}^{(n)}(a_r) \geq C \frac{\sinh(\lambda r)}{\lambda \sinh(r)^{\frac{n-1}{2}}} \, . 
\end{align*}
\end{enumerate}
\end{proposition}
\begin{remark}
The limit formula in item (2) of the above Proposition has previously been computed by Strichartz in \cite[Eq.4.20, p.84]{Strichartz1989HarmonicAA} by viewing the spherical functions on $\H^n$ as generalized Legendre functions. We present another proof of this using the trigonometric expansion given in Proposition \ref{PropTrigonometricExpansionofSphericalFunctions}.
\end{remark}
To prove item (2) in Proposition \ref{PropSphericalAverage}, we use of the following elementary Lemma. 

\begin{lemma}
\label{LemmaAverageLimits}
Let $R_o \geq 0$ and suppose that $\alpha, \beta : [R_o, +\infty) \rightarrow \R$ are bounded continuous functions such that $\alpha(R) \rightarrow a$ and $\beta(R) \rightarrow b$ as $R \rightarrow +\infty$. Then for every $\lambda > 0$,
\begin{align*}
\frac{1}{R}\int_{R_o}^{R} (\alpha(r) \cos(\lambda r) + \beta(r) \sin(\lambda r))^2 dr \longrightarrow \frac{a^2 + b^2}{2} \, , \quad R \rightarrow +\infty \, . 
\end{align*}
\end{lemma}

\begin{proof}[Proof of Proposition \ref{PropSphericalAverage}]
To prove (1) we use that $\cosh(r) - 1 \leq \sinh(r)$ for all $ r \geq 0$ to find that
\begin{align*}
|\omega_{\lambda}^{(n)}(a_r)| &= \frac{a_n}{\sinh(r)^{n-2}} \Big| \int_0^r (\cosh(r) - \cosh(s))^{\frac{n-3}{2}} \cos(\lambda s) ds \Big| \\
&\leq a_n \frac{(\cosh(r) - 1)^{\frac{n-3}{2}}}{\sinh(r)^{n-2}} \Big| \int_0^r \cos(\lambda s) ds\Big| \leq \frac{a_n}{\lambda \sinh(r)^{\frac{n-1}{2}}} \, . 
\end{align*}
For (2) it is enough to show that 
\begin{align*}
\lim_{R \rightarrow +\infty} \frac{1}{R} \int_{R_o}^R |\omega^{(n)}_{\lambda}(a_r)|^2 \sinh(r)^{n - 1} dr = 2^{n - 2} \frac{ |\rB(i\lambda, \frac{n-1}{2})|^2}{\rB(\frac{1}{2}, \frac{n-1}{2})^2}  
\end{align*}
for some choice of $R_o \geq 0$. Choose $R_o$ sufficiently large so that we can apply Proposition \ref{PropTrigonometricExpansionofSphericalFunctions} and Lemma \ref{LemmaAverageLimits} to find that 
\begin{align*}
\frac{1}{R} \int_{R_o}^R &|\omega^{(n)}_{\lambda}(a_r)|^2 \sinh(r)^{n - 1} dr = \frac{a_n^2}{R} \int_{R_o}^R (\alpha_n(r, \lambda)\cos(\lambda r) + \beta_n(r, \lambda) \sin(\lambda r))^2 dr  \\
&\longrightarrow a_n^2 \frac{\Re(\rB(i\lambda, \frac{n-1}{2}))^2 + \Im(\rB(i\lambda, \frac{n-1}{2}))^2}{2} = 2^{n - 2} \frac{ |\rB(i\lambda, \frac{n-1}{2})|^2}{\rB(\frac{1}{2}, \frac{n-1}{2})^2}  
\end{align*}
as $R \rightarrow +\infty$. 

For (3), we bound from below by
\begin{align*}
\omega_{i\lambda}(a_r) &= \frac{a_n}{\sinh(r)^{n-2}} \int_0^r (\cosh(r) - \cosh(s))^{\frac{n-3}{2}} \cosh(\lambda s) ds \\
&\geq  \frac{a_n}{\sinh(r)^{n-2}} \int_0^{r - 1} (\cosh(r) - \cosh(s))^{\frac{n-3}{2}} \cosh(\lambda s) ds \\
&\geq a_n \frac{(\cosh(r) - \cosh(r - 1))^{\frac{n-3}{2}}}{\sinh(r)^{n-2}} \int_0^{r-1} \cosh(\lambda s) ds \\
&= a_n \frac{\cosh(r)^{\frac{n-3}{2}}}{\sinh(r)^{n-2}} \Big( 1 - \frac{\cosh(r - 1)}{\cosh(r)} \Big)^{\frac{n-3}{2}} \frac{\sinh(\lambda(r - 1))}{\lambda} \\
&\geq \frac{a_n}{\sinh(r)^{\frac{n-1}{2}}} \Big( 1 - \frac{\cosh(r - 1)}{\cosh(r)} \Big)^{\frac{n-3}{2}} \frac{\sinh(\lambda(r - 1))}{\lambda} \, . 
\end{align*}
Using that $\lambda \in [0, \frac{n-1}{2}]$ and $r > 1$, elementary calculations yield the lower bounds 
\begin{align*}
\sinh(\lambda(r - 1)) \geq \e^{-\frac{n-1}{2}} \sinh(\lambda r) \quad \mbox{ and } \quad  1 - \frac{\cosh(r - 1)}{\cosh(r)}  \geq 1 - \frac{2}{\e}
\end{align*}
so that 
\begin{align*}
\omega_{i\lambda}(a_r) \geq C \frac{\sinh(\lambda r)}{\lambda \sinh(r)^{\frac{n-1}{2}}} \, , \quad C = a_n \e^{-\frac{n-1}{2}} (1 - 2\e^{-1})^{\frac{n-3}{2}} \, . 
\end{align*}
\end{proof}

\subsection{The spherical transform on $\H^n$}
\label{The spherical transform on H^n}

Given a function $\varphi \in \Borelbndinfty(G_n)$, its \emph{Fourier-Helgason transform} is the function $\sF\varphi : \C \times \bS^{n-1} \rightarrow \C$ given by
\begin{align*}
\sF\varphi(\lambda, u) = \int_{G_n} \varphi(g) \overline{\xi_{\lambda, u}(g)} dm_{G_n}(g) \, . 
\end{align*}
If $\varphi$ happens to be bi-$K_n$-invariant, then 
\begin{align*}
\sF\varphi(\lambda, u) = \int_{G_n} \varphi(g) \overline{\omega_{\lambda}(g)} dm_{G_n}(g) 
\end{align*}
for any choice of $u \in \bS^{n-1}$.
\begin{definition}[Hyperbolic spherical transform]
The \emph{spherical transform} of a function $\varphi \in \Borelbndinfty(G_n, K_n)$ is the function $\hat{\varphi} : \C \rightarrow \C$ given by 
\begin{align*}
\hat{\varphi}(\lambda) = \int_{G_n} \varphi(g) \overline{\omega_{\lambda}(g)} dm_{G_n}(g) = \int_0^{\infty} \varphi(a_t) \overline{\omega_{\lambda}(a_t)} \sinh(t)^{n-1} dt \, . 
\end{align*}
\end{definition}
One shows that $\hat{\varphi}$ is even and analytic on $\C$ for every $\varphi \in \Borelbndinfty(G_n, K_n)$ and that $\hat{\varphi_1^* * \varphi_2}(\lambda) = \overline{\hat{\varphi}_1(\overline{\lambda})} \hat{\varphi}_2(\lambda)$. In particular, if $\lambda \in \R \cup i\R$ then $\hat{\varphi^* * \varphi}(\lambda) = |\hat{\varphi}(\lambda)|^2$. 

As for the $L^2$-theory of the spherical transform, we again have a Plancherel formula. To state it, we introduce the \emph{Harish-Chandra $c$-function}
\begin{align*}
c_n(\lambda) = \rho_n 2^{\frac{n - 2}{2}} \frac{\rB(i\lambda, \frac{n-1}{2})}{\rB(\frac{1}{2}, \frac{n-1}{2})} = \rho_n 2^{\frac{n - 2}{2}}  \frac{\Gamma(\frac{n}{2}) \Gamma(i\lambda)}{\Gamma(\frac{n-1}{2}) \Gamma(i \lambda + \frac{n-1}{2})}  \, , \quad \lambda \in \C \backslash i \Z_{\geq 0} \, ,
\end{align*}
where $\rho_n > 0$ is a normalizing constant such that the Plancherel formula below holds. The precise value of $\rho_n$ is not crucial to us in proving the remaining main results. A computation of the $c$-function with an exact value of $\rho_n$ can be found in \cite[Section 4, Thm 4.2. p.334]{Takahashi1963SurLR}, but we note that the value depends on the normalization of the Haar measure $m_{G_n}$. Note that $c_n$ is meromorphic and does not have any zeros in $\C$. In particular, $|c_n|$ is a strictly positive function.  
\begin{lemma}[Spherical Plancherel formula]
For every $\varphi \in \Borelbndinfty(G_n, K_n)$,
\begin{align*}
\int_{G_n} |\varphi(g)|^2 dm_{G_n}(g) = \int_0^{\infty} |\hat{\varphi}(\lambda)|^2 \frac{d\lambda}{|c_n(\lambda)|^2}
\end{align*}
Equivalently, the spherical transform extends to a unitary map 
$$L^2(G_n, K_n) \rightarrow L^2(\R_{\geq 0}, |c_n(\lambda)|^{-2} d\lambda) \, . $$
\end{lemma}

Before moving on, we will compute the spherical transform of indicator functions on balls.
\begin{example}[Spherical transform of the indicator function on a centered Cartan ball]
\label{ExHyperbolicIndicatorFT}
The computation of the spherical transform for the indicator function $\chi_{B_r}$ is similar to the Euclidean case. Using the formula 
$$ \omega^{(n)}_{\lambda}(t) = a_n \sinh(t)^{2 - n}\int_0^t (\cosh(t) - \cosh(s))^{\frac{n-3}{2}} \cos(\lambda s) ds \, , \quad a_n =\frac{2^{\frac{n-1}{2}}}{\rB(\frac{1}{2},\frac{n-1}{2})}  $$
for the spherical functions, Fubini's Theorem yields
\begin{align*}
\hat{\chi}_{B_r}(\lambda) &=   \int_0^r \omega^{(n)}_{\lambda}(t) \sinh(t)^{n-1} dt \\
&= a_n \int_0^r \Big( \int_0^t (\cosh(t) - \cosh(s))^{\frac{n-3}{2}} \cos(\lambda s) ds \Big) t  dt \\
&= a_n \int_0^r \Big( \int_s^r (\cosh(t) - \cosh(s))^{\frac{n-3}{2}} \sinh(t) dt \Big) \cos(\lambda s) ds \\
&= \frac{2 a_n}{n-1} \int_0^r (\cosh(r) - \cosh(s))^{\frac{n-1}{2}} \cos(\lambda s) ds \, . 
\end{align*}
In terms of spherical functions, we have 
\begin{align}
\label{EqHyperbolicFTIndincatorFunction}
\hat{\chi}_{B_r}(\lambda) =  b_n \sinh(r)^{n} \omega_{\lambda}^{(n + 2)}(r) \, , \quad b_n = \frac{2^{\frac{n+1}{2}}}{n} \, . 
\end{align}
\end{example}

%% file: HyperbolicDiffraction.tex
In this section we will introduce the diffraction measure associated to a random measure on $\H^n$. Furthermore, we provide an asymptotic upper bound on the diffraction measures of large intervals, similar to the Euclidean case. The support such diffraction measures will be confined to the subset
\begin{align*}
\Omega_n^+ = (0, +\infty) \cup i[0, (n - 1)/2] \subset \C \, . 
\end{align*}
The following Lemma ensures existence and uniqueness of the diffraction measure of a random measure on $\H^n$, which we will prove in larger generality in upcoming work.
\begin{lemma}
\label{LemmaHyperbolicTemperedDiffraction}
Let $\mu$ be a locally square-integrable invariant random measure on $\H^n$ with autocorrelation measure $\eta_{\mu}$ on $G_n$. Then there is a unique positive Radon measure $\hat{\eta}_{\mu}$ on $\Omega_n^+$ such that
\begin{align*}
\int_{G_n} \varphi(g) d\eta_{\mu}(g) = \int_{\Omega_n^+} \hat{\varphi}(\lambda) d\hat{\eta}_{\mu}(\lambda) \, , \quad \forall \, \varphi \in \Ccinfty(G_n, K_n) \, . 
\end{align*}
Moreover, $\hat{\eta}_{\mu}(\{i \frac{n-1}{2} \}) = 0$ and the measures $\eta_{\mu}, \hat{\eta}_{\mu}$ are tempered in the sense that there are constants $a, b \in \Z_{\geq 0}$ such that
\begin{align*}
\int_{G_n} (1 + \ell(g))^{-a} \omega_0^{(n)}(g)^{- \frac{n+1}{2}} d\eta_{\mu}(g) < +\infty \quad \mbox{ and } \quad \int_{\Omega_n^+} (1 + |\lambda|)^{-b} d\hat{\eta}_{\mu}(\lambda) < +\infty \, . 
\end{align*}
\end{lemma}
Since the spherical function $\omega_0^{(n)}$ decays exponentially, the following family of functions are integrable with respect to the autocorrelation measure.
\begin{corollary}
\label{CorollarySuperExponentiallyDecayingFunctionIntegrableAutocorrelation}
If $\varphi : G_n \rightarrow \C$ is a measurable function such that
$$ |\varphi(g)|\e^{-\alpha d(g.o, o)} \longrightarrow 0 $$
as $g \rightarrow +\infty$ for every $\alpha > 0$, then $\varphi$ is integrable with respect to $\eta_{\mu}$.
\end{corollary}
In fact, $\eta_{\mu}$ extends to a continuous linear functional on a certain \emph{Harish-Chandra $L^p$-space} of smooth functions that are rapidly decaying with respect to a power of $\omega_0$, but for our purposes we will not need this here. 

We define the \emph{principal diffraction measure} $\hat{\eta}_{\mu}^{(p)}$ and the \emph{complementary diffraction measure $\hat{\eta}_{\mu}^{(c)}$} to be the restrictions of $\hat{\eta}_{\mu}$ to $(0, +\infty)$ and $i[0, \frac{n-1}{2})$ respectively, so that
\begin{align*}
\int_{\Omega_n^+} \hat{\varphi}(\lambda) d\hat{\eta}_{\mu}(\lambda) = \int_0^{\infty} \hat{\varphi}(\lambda) d\hat{\eta}_{\mu}^{(p)}(\lambda) + \int_0^{\frac{n-1}{2}} \hat{\varphi}(i\lambda) d\hat{\eta}_{\mu}^{(c)}(\lambda) \, , \quad \varphi \in \Ccinfty(G_n, K_n) \, . 
\end{align*}
Both measures $\hat{\eta}_{\mu}^{(p)}, \hat{\eta}_{\mu}^{(c)}$ are positive Radon measures, and $\hat{\eta}_{\mu}^{(c)}$ is necessarily a finite measure. As in the Euclidean case, we can extend the defining identity
\begin{align*}
\Var_{\mu}(Sf) = \eta_{\mu}(\varphi_f^* * \varphi_f) = \int_{0}^{\infty} |\hat{\varphi}_f(\lambda)|^2 d\hat{\eta}_{\mu}^{(p)}(\lambda) + \int_0^{\frac{n-1}{2}} |\hat{\varphi}_f(i\lambda)|^2 d\hat{\eta}_{\mu}^{(c)}(\lambda) 
\end{align*}
to all radial functions $f \in \Borelbndinfty(\H^n)$, in particular the indicator functions $f = \chi_{B_r}$, $r > 0$.
\begin{lemma}
Let $\mu$ be an invariant locally square-integrable random measure on $\H^n$ and $\varphi \in \Borelbndinfty(G_n, K_n)$. Then
\begin{align*}
\eta_{\mu}(\varphi^* * \varphi) = \int_{\Omega_n^+} |\hat{\varphi}(\lambda)|^2 d\hat{\eta}_{\mu}(\lambda) \, .
\end{align*}
In particular, $\hat{\varphi} \in L^2(\Omega_n^+, \hat{\eta}_{\mu})$.
\end{lemma}
\begin{proof}
Same proof as that of Lemma \ref{LemmaExtensionOfDiffractionFormulaToMeasurableFunctions} by letting $\beta_{\varepsilon}$ be the bi-$K_n$-invariant rapidly decaying smooth function on $G_n$ satisfying $\hat{\beta}_{\varepsilon}(\lambda) = \e^{-\varepsilon^2(\lambda^2 + (\frac{n-1}{2})^2)}$.
\end{proof}
\begin{example}[Diffraction of the $m_{\H^n}$-Poisson point process]
Recall that the autocorrelation measure $\eta_{\Poi}$ of the $m_{\H^n}$-Poisson point process is evaluation at the identity $e \in G_n$ when considered as a linear functional on $\Borelbndinfty(G_n, K_n)$. Therefore, by the Plancherel formula for the spherical transform,
\begin{align*}
\eta_{\Poi}(\varphi^* * \varphi) = (\varphi^* * \varphi)(e) = \int_{G_n} |\varphi(g)|^2 dm_{G_n}(g) = \int_0^{\infty} |\hat{\varphi}(\lambda)|^2 \frac{d\lambda}{|c_n(\lambda)|^2} \, . 
\end{align*}
Thus the Poisson diffraction measure is 
\begin{align*} 
d\hat{\eta}_{\Poi}^{(p)}(\lambda) = |c_n(\lambda)|^{-2} d\lambda \quad \mbox{ and } \quad d\hat{\eta}^{(c)}_{\Poi}(\lambda) = 0\, .
\end{align*}
\end{example}
We will next prove the following analogue of Lemma \ref{LemmaEuclideanBergFrost}.
\begin{proposition}
\label{PropHyperbolicBergFrost}
Let $\mu$ be a locally square-integrable invariant random measure on $\H^n$. Then the principal part $\hat{\eta}_{\mu}^{(p)}$ of the diffraction measure $\hat{\eta}_{\mu}$ of $\mu$ satisfies 
\begin{align*}
\hat{\eta}_{\mu}^{(p)}([0, L]) \ll_{n, \mu} L^n  
\end{align*}
for all sufficiently large $L > 0$.
\end{proposition}
\begin{remark}
For the $m_{\H^n}$-Poisson point process on $\H^n$ we have that $d\hat{\eta}^{(c)}_{\Poi}(\lambda) = 0$ and $d\hat{\eta}^{(p)}_{\Poi}(\lambda) = |c_n(\lambda)|^{-2} d\lambda$, and the $c$-function has the asymptotic bounds
\begin{align}
\label{EqCFunctionAsymptotics}
c_n(\lambda) =  \rho_n 2^{\frac{n - 2}{2}} \frac{\Gamma(\frac{n}{2})\Gamma(i\lambda)}{\Gamma(\frac{n-1}{2}) \Gamma(i \lambda + \frac{n-1}{2})} \asymp \begin{cases} \lambda^{-\frac{n - 1}{2}} &\mbox{ if } \lambda \gg 0 \\ \lambda^{-1} &\mbox{ if } \lambda \approx 0 \, . \end{cases}
\end{align}
For sufficiently large $L > 0$ we then have
\begin{align*}
\hat{\eta}_{\mu}^{(p)}([0, L]) = \int_0^{L} \frac{d\lambda}{|c_n(\lambda)|^2} \asymp \int_0^L \lambda^{n-1} d\lambda \asymp L^n \, .  
\end{align*}
The upper bound in Proposition \ref{PropHyperbolicBergFrost} is thus, up to constant, a strict upper bound in general and which is achieved by the invariant Poisson point process. 
\end{remark}
In order to prove Proposition \ref{PropHyperbolicBergFrost} we make use of the \emph{heat kernel} $h_{\tau}$ on $G_n$ with $\tau > 0$. It is the bi-$K_n$-invariant function
\begin{align*}
h_{\tau}(g) = \int_0^{\infty} \e^{-\tau((\frac{n-1}{2})^2 + \lambda^2)} \omega_{\lambda}(g) \frac{d\lambda}{|c_n(\lambda)|^2} \, . 
\end{align*}
Equivalently, $\hat{h}_{\tau}(\lambda) = \e^{-\tau((\frac{n-1}{2})^2 + \lambda^2)}$. For the heat kernel in even dimensions $n$, there is no closed formula in terms of elementary functions, but it will be sufficient for us to consider the following uniform bounds in \cite[Theorem 3.1, p.186]{Davies1988HeatKB}.
\begin{lemma}[Davies-Mandouvalos]
\label{LemmaDaviesMandouvalos}
The heat kernel satisfies
\begin{align*}
h_{\tau}(a_r) \asymp \tau^{-\frac{n}{2}} (1 + \tau + r)^{\frac{n-3}{2}}(1 + r) \e^{-(\frac{n-1}{2})^2\tau - \frac{n-1}{2}r - \frac{1}{4\tau} r^2} 
\end{align*}
uniformly in $r \in [0, +\infty)$ and $\tau \in (0, +\infty)$.
\end{lemma}

In particular, the heat kernel $h_{\tau}$ is tempered and hence integrable with respect to the autocorrelation measure $\eta_{\mu}$ of a locally square-integrable random measure $\mu$ on $\H^n$ by Lemma \ref{LemmaHyperbolicTemperedDiffraction}. 
\begin{proof}[Proof of Proposition \ref{PropHyperbolicBergFrost}]
Denote by $\chi_{[0, L]}$ the indicator function of the interval $[0, L]$ in $ \R_{\geq 0}$ and bound it from above by $\chi_{[0, L]}(\lambda) \leq \e^{1 - \lambda^2 / L^2}$, so that 
\begin{align*}
\hat{\eta}_{\mu}^{(p)}([0, L]) \leq \int_0^{\infty} \e^{1 - \frac{\lambda^2}{L^2}} d\hat{\eta}_{\mu}^{(p)}(\lambda) &\leq \int_{\Omega_n^+} \e^{1 - \frac{\lambda^2}{L^2}} d\hat{\eta}_{\mu}(\lambda) \\%
&= \e^{1 + (\frac{n-1}{2L})^2} \int_{\Omega_n^+} \e^{- \frac{1}{L^2}((\frac{n-1}{2})^2 + \lambda^2)} d\hat{\eta}_{\mu}(\lambda) \, . 
\end{align*}
Note that the latter integral is finite since $\hat{\eta}_{\mu}^{(p)}$ is tempered by Lemma \ref{LemmaHyperbolicTemperedDiffraction}. Since 
$$ \hat{h}_{\tau}(\lambda) = \e^{- \frac{1}{L^2}((\frac{n-1}{2})^2 + \lambda^2)} $$
then from the definition of the diffraction measure $\hat{\eta}_{\mu}$ we see that
\begin{align*}
\int_{\Omega_n^+} \e^{- \frac{1}{L^2}((\frac{n-1}{2})^2 + \lambda^2)} d\hat{\eta}_{\mu}(\lambda) = \int_{G_n} h_{L^{-2}}(g) d\eta_{\mu}(g) \, , 
\end{align*}
so that $\hat{\eta}_{\mu}^{(p)}([0, L]) \leq \e^{1 + (\frac{n-1}{2L})^2} \eta_{\mu}(h_{L^{-2}})$. Consider the Radon measure $\eta_{\mu}^+ = \eta_{\mu} + \iota_{\mu} m_{G_n}$, which is a positive tempered measure satisfying $\eta_{\mu}(h_{\tau}) \leq \eta_{\mu}^+(h_{\tau})$ for all $\tau > 0$. By Lemma \ref{LemmaDaviesMandouvalos}, we find in particular that there is a constant $C > 0$ such that
\begin{align*}
0 < h_{L^{-2}}(a_r) \leq C L^n (1 + L^{-2} + r)^{\frac{n-3}{2}} (1 + r) \e^{- \frac{(Lr)^2}{4}} \leq C L^n (2 + r)^{\frac{n-1}{2}} \e^{-\frac{r^2}{4}}
\end{align*}
for all $L > 1$ and all $r \in [0, +\infty)$. We moreover bound $\e^{1 + (\frac{n-1}{2L})^2} \leq \e^{1 + (\frac{n-1}{2})^2}$ for $L > 1$, so that the final bound becomes 
\begin{align*}
\hat{\eta}_{\mu}^{(p)}([0, L]) \leq \e^{1 + (\frac{n-1}{2})^2} C L^n \int_{G_n} (2 + \ell(g))^{\frac{n-1}{2}} \e^{-\frac{\ell(g)^2}{4}} d\eta_{\mu}^+(g) \, .  
\end{align*}
Since $\eta_{\mu}^+$ is tempered, the latter integral is finite by Corollary \ref{CorollarySuperExponentiallyDecayingFunctionIntegrableAutocorrelation} and thus $\hat{\eta}_{\mu}^{(p)}([0, L]) \ll_{n, \mu} L^n$. 
\end{proof}

\begin{remark}
It should be noted that, there are uniform estimates of the heat kernel by Anker and Ostellari in the Main Theorem of \cite{Anker2003TheHK}, and using those the above proof can be approached in the same way for general non-compact symmetric spaces $X = G/K$ with $G$ a non-compact connected real semisimple Lie group with finite center and maximal compact subgroup $K < G$. 
\end{remark}

%% file: HyperbolicBeck.tex
We prove Theorem \ref{Theorem2} using the following hyperbolic analogue of Theorem \ref{ThmBeck}. 
\begin{theorem}
\label{TheoremHyperbolicBeck}
Let $\mu$ be an invariant locally square-integrable random measure on $\H^n$ with diffraction measure $\hat{\eta}_{\mu}$. The following holds: 
\begin{enumerate}
\item if $\hat{\eta}_{\mu}^{(c)}([0, \frac{n-1}{2})) > 0$, then for every $R_o > 1$ there is a constant $C = C(n, R_o) > 0$ such that 
\begin{align*}
\frac{\NVmu(R)}{\cosh(R)^{n-1}} \geq C \int_{0}^{\frac{n-1}{2}} \frac{\sinh^2(R\lambda)}{ \lambda^2 } d\hat{\eta}_{\mu}^{(c)}(\lambda) \, , \quad \forall \, R \geq R_o \, . 
\end{align*}
\item if $\hat{\eta}_{\mu}^{(p)}((0, +\infty)) > 0$ then for every $\lambda_o > 0$ there is $R_o = R_o(\lambda_o) > 0$ and a constant $C(n, \lambda_o) > 0$ such that 
\begin{align*}
\frac{1}{R} \int_0^R \frac{\NVmu(r)}{\cosh(r)^{n-1}} dr \geq C \int_{\lambda_o}^{\infty} |c_{n + 2}(\lambda)|^2 d\hat{\eta}_{\mu}^{(p)}(\lambda) \, , \quad \forall \, R \geq R_o \, .  
\end{align*} 
\end{enumerate}
\end{theorem}

\begin{proof}
By definition of the diffraction measure $\hat{\eta}_{\mu}$ and Equation \eqref{EqHyperbolicFTIndincatorFunction} in Example \ref{ExHyperbolicIndicatorFT},
\begin{align*}
\frac{\NVmu(r)}{\cosh(r)^{n-1}} &= \int_0^{\infty} \frac{|\hat{\chi}_{B_r}(\lambda)|^2}{\cosh(r)^{n-1}} d\hat{\eta}_{\mu}^{(p)}(\lambda) + \int_0^{\frac{n-1}{2}} \frac{|\hat{\chi}_{B_r}(i\lambda)|^2}{\cosh(r)^{n-1}} d\hat{\eta}_{\mu}^{(c)}(\lambda) \\
&= \frac{2^{n+1}}{n^2} \tanh(r)^{n-1} \sinh(r)^{n+1} \Big( \int_0^{\infty} |\omega_{\lambda}^{(n+2)}(a_r)|^2 d\hat{\eta}_{\mu}^{(p)}(\lambda) \\
&\qquad\qquad\qquad\qquad\qquad\qquad\quad + \int_0^{\frac{n-1}{2}} |\omega_{i\lambda}^{(n+2)}(a_r)|^2 d\hat{\eta}_{\mu}^{(c)}(\lambda) \Big) \, . 
\end{align*}
To prove (1), we use item (3) in Proposition \ref{PropSphericalAverage} to find a constant $C_o = C_o(n) > 0$ for $R_o > 1$ such that 
\begin{align*}
\omega_{i \lambda}^{(n + 2)}(a_R) \geq C_o \frac{\sinh(\lambda R)}{\lambda \sinh(R)^{\frac{n+1}{2}}} \, \quad \forall \, R \geq R_o \, . 
\end{align*}
Then 
\begin{align*}
\frac{\NVmu(R)}{\cosh(R)^{n-1}} &\geq \frac{2^{n+1}}{n^2} \tanh(R)^{n-1} \sinh(R)^{n+1} \int_0^{\frac{n-1}{2}} |\omega_{i\lambda}^{(n+2)}(a_R)|^2 d\hat{\eta}_{\mu}^{(c)}(\lambda) \\
&\geq \frac{2^{n+1}}{n^2} C_o \tanh(R_o)^{n-1} \int_0^{\frac{n-1}{2}} \frac{\sinh^2(R\lambda)}{\lambda^2} d\hat{\eta}_{\mu}^{(c)}(\lambda) \, . 
\end{align*}
If we let $C(n, R_o) =  2^{n+1} n^{-2} C_o(n) \tanh(R_o)^{n-1}$ then we are done. 

For (2), we assume without loss of generality that $\hat{\eta}_{\mu}^{(c)} = 0$. An application of Fubini yields 
\begin{align*}
\frac{1}{R} \int_0^R \frac{\NVmu(r)}{\cosh(r)^{n-1}} dr &\geq \frac{2^{n+1}}{n^2} \int_{\lambda_o}^{\infty} \Big( \frac{1}{R} \int_{0}^R |\omega_{\lambda}^{(n+2)}(a_r)|^2 \tanh(r)^{n-1} \sinh(r)^{n+1} dr \Big) d\hat{\eta}_{\mu}^{(p)}(\lambda) \, . 
\end{align*}
We claim that  
\begin{align*}
\lim_{R \rightarrow +\infty}  \frac{1}{R} \int_{0}^R |\omega_{\lambda}^{(n+2)}(a_r)|^2 \tanh(r)^{n-1} \sinh(r)^{n+1} dr = \rho_{n + 2}^{-2} |c_{n+2}(\lambda)|^2  
\end{align*}
uniformly for all $\lambda \geq \lambda_o$, which will finish the proof by dominated convergence. To compute this limit, write
\begin{align*}
\int_0^R |\omega_{\lambda}^{(n+2)}(a_r)|^2 \tanh(r)^{n-1} &\sinh(r)^{n+1} dr = \int_0^R |\omega_{\lambda}^{(n+2)}(a_r)|^2 \sinh(r)^{n+1} dr \\
&\quad + \int_0^R |\omega_{\lambda}^{(n+2)}(a_r)|^2 (\tanh(r)^{n-1} - 1) \sinh(r)^{n+1} dr \, . 
\end{align*}
By item (1) in Proposition \ref{PropSphericalAverage}, $\omega_{\lambda}^{(n + 2)}(a_r) \leq \lambda^{-1} \sinh(r)^{-\frac{n+1}{2}}$, and since $1 - \tanh(r)^{n-1} \leq n \e^{-r}$ we get 
\begin{align*}
\frac{1}{R} \int_0^R |\omega_{\lambda}^{(n+2)}(a_r)|^2 |\tanh(r)^{n-1} - 1| \sinh(r)^{n+1} dr \leq \frac{n}{R\lambda^2} \int_0^R \e^{-r} dr \leq \frac{n}{R\lambda^2} \underset{R \rightarrow +\infty}{\longrightarrow}  0  \, . 
\end{align*}
By item (2) in Proposition \ref{PropSphericalAverage} we finally have that
\begin{align*}
 \frac{1}{R} \int_{0}^R |\omega_{\lambda}^{(n+2)}(a_r)|^2 \tanh(r)^{n-1} \sinh(r)^{n+1} dr =  &\, \frac{1}{R} \int_{0}^R |\omega_{\lambda}^{(n+2)}(a_r)|^2 \sinh(r)^{n+1} dr \\
\underset{R \rightarrow +\infty}{\longrightarrow}  &\, 2^{n}   \frac{ |\rB(i\lambda, \frac{n+1}{2})|^2}{\rB(\frac{1}{2}, \frac{n+1}{2})^2} = \rho_{n+2}^{-2}|c_{n+2}(\lambda)|^2 \, . 
\end{align*}
\end{proof}
\begin{remark}
\label{RemarkHyperbolicDiffractionIntegrability}
The lower bounds in Theorem \ref{TheoremHyperbolicBeck} are finite for fixed $R > 0$. For the complementary diffraction bound the statement is trivial, and for the principal diffraction bound we write  
\begin{align*}
\int_{\lambda_o}^{\infty} |c_{n + 2}(\lambda)|^2 d\hat{\eta}_{\mu}^{(p)}(\lambda) \ll_n \int_{\lambda_o}^{\infty} \hat{\eta}_{\mu}^{(p)}([0, u]) \frac{d}{du} |c_{n + 2}(u)|^2 du  \, . 
\end{align*}
Using the standard asymptotic
\begin{align*}
\frac{\Gamma(iu)}{\Gamma(iu + \frac{n-1}{2})} \sim (iu)^{-\frac{n - 1}{2}}
\end{align*}
as $u \rightarrow  +\infty$ one can see that $\frac{d}{du}|c_{n + 2}(u)|^2 \ll_{n, \lambda_o} u^{-(n + 2)}$ for large $u \geq \lambda_o$, and since $\hat{\eta}_{\mu}^{(p)}$ satisfies $\hat{\eta}_{\mu}^{(p)}([0, L]) \ll_{n,\mu} L^n$ for sufficiently large $L > 0$ by Lemma \ref{PropHyperbolicBergFrost}, then
\begin{align*}
\int_{\lambda_o}^{\infty} |c_{n + 2}(\lambda)|^2 d\hat{\eta}_{\mu}^{(p)}(\lambda) \ll_{n, \lambda_o} \int_{\lambda_o}^{\infty} \frac{\hat{\eta}_{\mu}^{(p)}([0, u])}{u^{n + 2}} du  \ll_{n, \mu, \lambda_o} \int_{\lambda_o}^{\infty} \frac{du}{u^{2}} = \lambda_o^{-1} < + \infty \, . 
\end{align*}
Thus the right hand side in item (2) of Theorem \ref{TheoremHyperbolicBeck} is indeed finite.
\end{remark}
\begin{proof}[Proof of Theorem \ref{Theorem2}]
The first statement of the Theorem follows from item (2) of Theorem \ref{TheoremHyperbolicBeck},
\begin{align*}
\limsup_{R \rightarrow +\infty} \frac{\NVmu(R)}{\cosh(R)^{n-1}} \geq \limsup_{R \rightarrow +\infty} \frac{1}{R} \int_0^R \frac{\NVmu(r)}{\cosh(r)^{n-1}} dr \geq C \int_0^{\infty} |c_{n + 2}(\lambda)|^2 d\hat{\eta}_{\mu}^{(p)}(\lambda)\, . 
\end{align*}
Since $|c_{n + 2}(\lambda)|^2 > 0$ for all $\lambda > 0$ then the right hand side is positive, possibly infinite.

Item (1) of Theorem \ref{Theorem2} follows from item (1) of Theorem \ref{TheoremHyperbolicBeck}, since 
\begin{align*}
\liminf_{R \rightarrow +\infty} \frac{\NVmu(R)}{R^2\cosh(R)^{n-1}} \geq \liminf_{R \rightarrow +\infty} \frac{C}{R^2} \int_0^{\frac{n-1}{2}} \frac{\sinh^2(R\lambda)}{\lambda^2} d\hat{\eta}^{(c)}_{\mu}(\lambda) \geq C \hat{\eta}^{(c)}_{\mu}(\{ 0 \}) > 0 \, . 
\end{align*}
Lastly, for any $\delta \in (0, 1]$ with $\hat{\eta}_{\mu}^{(c)}([\delta\frac{n-1}{2}, \frac{n-1}{2})) > 0$ we have that
\begin{align*}
\liminf_{R \rightarrow +\infty} \frac{\NVmu(R)}{\cosh(R)^{(n-1)(1 + \delta)}} &\geq \liminf_{R \rightarrow +\infty} \frac{C}{\cosh(R)^{(n-1)\delta}} \int_{\delta\frac{n-1}{2}}^{\frac{n-1}{2}} \frac{\sinh^2(R\lambda)}{\lambda^2} d \hat{\eta}_{\mu}^{(c)}(\lambda) \\
&\geq C \hat{\eta}_{\mu}^{(c)}\Big(\Big[\delta \frac{n-1}{2}, \frac{n-1}{2}\Big)\Big) \liminf_{R \rightarrow \infty} \frac{\sinh^2(R\delta \frac{n-1}{2})}{\delta^2 \cosh(R)^{(n-1)\delta}} \\
&= C 2^{(n-1)\delta - 2} \delta^{-2} \hat{\eta}_{\mu}^{(c)}\Big(\Big[\delta \frac{n-1}{2}, \frac{n-1}{2}\Big)\Big) > 0 \, . 
\end{align*}
Thus item (2) of Theorem \ref{Theorem2} follows.
\end{proof}

%% file: SpectralHyperuniformity.tex
We define spectral hyperuniformity of an invariant locally square-integrable random measure on $\H^n$ in Subsection \ref{Spectral hyperuniformity} and define the relative notion of stealth in Subsection \ref{Stealthy random measures}.

\subsection{Spectral hyperuniformity}
\label{Spectral hyperuniformity}
By Theorem \ref{Theorem2}, we know that random measures $\mu$ on $\H^n$ are never geometrically hyperuniform in the sense that 
\begin{align}
\limsup_{R \rightarrow +\infty} \frac{\NVmu(R)}{\cosh(R)^{n-1}} > 0 \, ,
\end{align}
and if $\mu$ admits non-trivial complementary spectrum, then this upper limit is infinite. A useful definition of spectral hyperuniformity would then necessarily have to differ from this geometric definition. Having understood the large scale behaviour of diffraction measures, one is inclined to study the small scale/local behaviour of the diffraction measure in order to obtain information on the large scale behaviour of the number variance. We take the stance that one should interpret spectral hyperuniformity as comparing the diffraction measure of a given random measure to the diffraction measure of the invariant Poisson point process around $\lambda = 0$. For the definition we shall use that
\begin{align*}
\hat{\eta}_{\Poi}((0, \varepsilon]) = \int_0^{\varepsilon} \frac{d\lambda}{|c_n(\lambda)|^2} \asymp \int_0^{\varepsilon} \lambda^2 d\lambda \asymp \varepsilon^3
\end{align*}
for sufficiently small $\varepsilon > 0$. Note that the exponent $3$ is the same in all dimensions $n$. 

\begin{definition}[Hyperbolic spectral hyperuniformity]
A random measure $\mu$ on $\H^n$ is \emph{spectrally hyperuniform} if $\hat{\eta}^{(c)}_{\mu} = 0$ and 
\begin{align*}
\limsup_{\varepsilon \rightarrow 0^+}\frac{\hat{\eta}_{\mu}^{(p)}((0, \varepsilon])}{\varepsilon^3} = 0 \, . 
\end{align*}
\end{definition}
\begin{remark}
There are many examples of point processes $\mu$ on $\H^n$ such that $\hat{\eta}_{\mu}^{(c)} \neq 0$, even for many lattices, including cocompact ones. These will be hyperfluctuating in the sense of Theorem \ref{Theorem2} and not spectrally hyperuniform according to our definition. 
\end{remark}
Having defined spectral hyperuniformity, we have a point of reference for defining more rigid properties of random measures on $\H^n$.

\subsection{Stealthy random measures}
\label{Stealthy random measures}
\begin{definition}[Stealthy random measure on $\H^n$]
A random measure $\mu$ on $\H^n$ is \emph{stealthy} if $\hat{\eta}_{\mu}^{(c)} = 0$ and there is a $\lambda_o > 0$ such that $\hat{\eta}_{\mu}^{(p)}((0, \lambda_o)) = 0$.
\end{definition}
It is clear from this definition that every stealthy random measure on $\H^n$ is spectrally hyperuniform. As mentioned in the introduction, in \cite[Section 1.2., Eq.B]{Jenni1984UeberDE} Jenni provides a lattice $\Gamma < G_2 \cong \SL_2(\R)$ such that the random lattice orbit $\mu_{\Gamma}$ on $\H^2$ is a stealthy point process, in particular spectrally hyperuniform.

An almost immediate consequence of Theorem \ref{TheoremHyperbolicBeck} is that we can compute the average asymptotic number variance of a given stealthy random measure in terms of the diffraction measure. 
\begin{corollary}
Let $\mu$ be a stealthy random measure on $\H^n$. Then 
\begin{align*}
\lim_{R \rightarrow +\infty} \frac{1}{R} \int_0^R \frac{\NVmu(r)}{\cosh(r)^{n-1}} dr =  \frac{2^{n+1}}{n^2 \rho_{n + 2}^2} \int_{0}^{\infty} |c_{n + 2}(\lambda)|^2 d\hat{\eta}_{\mu}^{(p)}(\lambda) \, , 
\end{align*}
which is strictly positive. 
\end{corollary}
\begin{proof}
If $\mu$ is stealthy there is a $\lambda_o > 0$ such that the diffraction measure $\hat{\eta}_{\mu}$ is supported on $[\lambda_o, +\infty) \subset \Omega_n^+$. Thus 
\begin{align*}
\frac{1}{R} \int_0^R \frac{\NVmu(r)}{\cosh(r)^{n-1}} dr &= \frac{2^{n+1}}{n^2} \int_{\lambda_o}^{\infty} \Big( \frac{1}{R} \int_{0}^R |\omega_{\lambda}^{(n+2)}(a_r)|^2 \tanh(r)^{n-1} \sinh(r)^{n+1} dr \Big) d\hat{\eta}_{\mu}^{(p)}(\lambda) \, .
\end{align*}
As shown in the proof of item (2) of Theorem \ref{TheoremHyperbolicBeck}, we have 
\begin{align*}
\lim_{R \rightarrow +\infty} \frac{1}{R} \int_{0}^R |\omega_{\lambda}^{(n+2)}(a_r)|^2 \tanh(r)^{n-1} \sinh(r)^{n+1} dr = \rho_{n + 2}^{-2} |c_{n+2}(\lambda)|^2
\end{align*}
uniformly in $\lambda \geq \lambda_o$, so by dominated convergence we're done.
\end{proof}

%% file: HyperbolicRandomPerturbedLatticeOrbits.tex
In Subsection \ref{A criterion for local square-integrability} we provide examples of non-compact perturbations of random lattice orbits that are locally square-integrable. In Subsection \ref{A formula for the diffraction measure} we compute the diffraction measure of random perturbed lattice orbits that are locally square-integrable and lastly we show that random perturbed lattice orbits in $\H^n$ are not spectrally hyperuniform in Subsection \ref{Random perturbed lattices on Hn are not spectrally hyperuniform}.

\subsection{A criterion for local square-integrability}
\label{A criterion for local square-integrability}

We first extend Lemma \ref{LemmaPerturbedLatticeisLocallySquareIntegrable} to include right-$K_n$-invariant perturbations $\nu \in \Prob(G_n)$ that are absolutely continuous with respect to the Haar measure $m_{G_n}$ with sufficent exponential decay of the Radon-Nikodym derivative at infinity. To do this we make use of the fact that the critical exponent of a lattice $\Gamma < G_n$ is that of the volume growth. The following Lemma is a special case of Theorem C in \cite{AlbaquerquePattersonSullivan}, which includes higher rank analogues as well.
\begin{lemma}
\label{LemmaHyperbolicLatticeCriticalExponent}
Let $\Gamma < G_n$ be a lattice. Then for every $\varepsilon > 0$,
\begin{align*}
\sum_{\gamma \in \Gamma} \e^{-(n - 1 + \varepsilon)d(\gamma.o, o)} < +\infty \, . 
\end{align*}
\end{lemma}
\begin{lemma}
\label{LemmaHyperbolicPerturbedLatticeSquareIntegrabilityCriterion}
Let $\beta : G_n \rightarrow [0, +\infty)$ be a right-$K_n$-invariant measurable function with $\int \beta \, dm_{G_n} = 1$ and assume that there are $C > 0$ and $\varepsilon > 0$ such that
\begin{align*}
\beta(g) \leq C \e^{-2(n - 1 + \varepsilon) d(g.o, o)}
\end{align*}
for all $g \in G_n$. Then for $d\nu(g) = \beta(g) dm_{G_n}(g)$ and any lattice $\Gamma < G_n$, the random perturbed lattice orbit $\mu_{\Gamma, \nu}$ in $\H^n$ is locally square-integrable.
\end{lemma}

By the same argument in the proof in Lemma \ref{LemmaPerturbedLatticeisLocallySquareIntegrable}, it suffices to bound the expression in Equation \ref{EqLemma2point5SecondTerm}, that is,
\begin{align*}
I_{\Gamma, \nu} := \int_{G_n/\Gamma} \Big( \sum_{\gamma \in \Gamma} \nu(\gamma^{-1}g^{-1}Q) \Big)^2 dm_{G_n/\Gamma}(g\Gamma)  \, .  
\end{align*}
\begin{lemma}
\label{LemmaHyperbolicPerturbedLatticeSecondCorrelationFormula}
Let $Q \subset G_n$ be a pre-compact measurable set and let $\nu \in \Prob(G_n)$. Then 
\begin{align*}
I_{\Gamma, \nu} = \frac{1}{\covol(\Gamma)} \sum_{\gamma \in \Gamma} \int_{G_n} \int_{G_n} m_{G_n}(Qh_1^{-1} \gamma h_2 \cap Q) d\nu(h_1) d\nu(h_2) \, .  
\end{align*}
\end{lemma}

\begin{proof}
By Fubini,
\begin{align*}
\int_{G_n/\Gamma} \Big( &\sum_{\gamma \in \Gamma} \nu(\gamma^{-1}g^{-1}Q) \Big)^2 dm_{G_n/\Gamma}(g\Gamma) = \\
&= \int_{G_n} \int_{G_n} \Big( \int_{G_n/\Gamma} |g \Gamma h_1 \cap Q| |g \Gamma h_2 \cap Q| dm_{G_n/\Gamma}(g\Gamma) \Big) d\nu(h_1) d\nu(h_2) \, . 
\end{align*}
From the computation for the second moment of a random lattice in example \ref{ExampleLatticeAutocorrelation}, we have that
\begin{align*}
\int_{G_n/\Gamma} |g \Gamma h_1 \cap Q| |g \Gamma h_2 \cap Q| dm_{G_n/\Gamma}(g\Gamma) &= \frac{1}{\covol(\Gamma)} \sum_{\gamma \in \Gamma} (\chi^*_{Qh_1^{-1}} * \chi_{Qh_2^{-1}})(\gamma) \\
&= \frac{1}{\covol(\Gamma)} \sum_{\gamma \in \Gamma} m_{G_n}(Qh_1^{-1} \gamma h_2 \cap Q) \, , 
\end{align*}
and Fubini then gives us that
\begin{align*}
\int_{G_n/\Gamma} \Big( \sum_{\gamma \in \Gamma} &\nu(\gamma^{-1}g^{-1}Q) \Big)^2 dm_{G_n/\Gamma}(g\Gamma) = \\
&\frac{1}{\covol(\Gamma)}  \sum_{\gamma \in \Gamma} \int_{G_n} \int_{G_n} m_{G_n}(Qh_1^{-1} \gamma h_2 \cap Q) d\nu(h_1) d\nu(h_2) \, .  
\end{align*}
\end{proof}
\begin{proof}[Proof of Lemma \ref{LemmaHyperbolicPerturbedLatticeSquareIntegrabilityCriterion}]
For convenience of notation we set $\alpha_{\varepsilon} = n - 1 + \varepsilon$. Let $Q = \pi^{-1}(B) = \varsigma(B)K$, so that $Q \subset G_n$ is pre-compact, measurable and $Q.o = B$. By Lemma \ref{LemmaHyperbolicPerturbedLatticeSecondCorrelationFormula}, it suffices to show that
\begin{align*}
\sum_{\gamma \in \Gamma} \int_{G_n} \int_{G_n} m_{G_n}(Qh_1^{-1} \gamma h_2 \cap Q) \beta(h_1) \beta(h_2) dm_{G_n}(h_1) dm_{G_n}(h_2)
\end{align*}
is finite. We first rewrite this expression as
\begin{align*}
\int_{G_n} \int_{G_n} &m_{G_n}(Qh_1^{-1} \gamma h_2 \cap Q) \beta(h_1) \beta(h_2) dm_{G_n}(h_1) dm_{G_n}(h_2) = \\
&= \int_{G_n} \Big( \int_{G_n} m_{G_n}(Qh_2 \cap Q) \beta(\gamma^{-1}h_1h_2) dm_{G_n}(h_2)\Big) \beta(h_1) dm_{G_n}(h_1) \, . 
\end{align*}
If $m_{G_n}(Qh_2 \cap Q) > 0$ then $h_2 \in Q^{-1}Q$, and using the assumed upper bound on $\beta$ we get that
\begin{align*}
\int_{G_n} m_{G_n}(Qh_2 \cap Q) &\beta(\gamma^{-1}h_1h_2) dm_{G_n}(h_2) \leq \\
&\leq C \int_{G_n} m_{G_n}(Qh_2 \cap Q) \e^{-2\alpha_{\varepsilon} d(h_2.o, h_1^{-1}\gamma.o)} dm_{G_n}(h_2) \\
&\leq C m_{G_n}(Q) \e^{-2\alpha_{\varepsilon}d(Q^{-1}Q.o, h_1^{-1}\gamma.o)} \, , 
\end{align*}
where $d(Q^{-1}Q.o, h_1^{-1}\gamma.o)$ denotes the minimal distance between $h_1^{-1}\gamma.o$ and $Q^{-1}Q.o$ in $\H^n$. Since $Q^{-1}Q$ is pre-compact, we use the reverse triangle inequality to find that 
$$\e^{-2\alpha_{\varepsilon}d(Q^{-1}Q.o, h_1^{-1}\gamma.o)} \leq \sup_{q \in Q^{-1}Q} \e^{2\alpha_{\varepsilon}d(q.o, o)} \e^{-2\alpha_{\varepsilon}d(h_1.o, \gamma.o)}$$
for all $h_1 \in G, \gamma \in \Gamma$. Thus it remains to prove finiteness of 
\begin{align*}
\sum_{\gamma \in \Gamma} \int_{G_n} \e^{-2\alpha_{\varepsilon}d(h_1.o, \gamma.o)} &\beta(h_1) dm_{G_n}(h_1) \leq C \sum_{\gamma \in \Gamma} \int_{G_n} \e^{-2\alpha_{\varepsilon}(d(h_1.o, \gamma.o) + d(h_1.o, o))} dm_{G_n}(h_1) \, . 
\end{align*}
To see this this we use the triangle inequality to find that
\begin{align*}
\e^{-2\alpha_{\varepsilon}(d(h_1.o, \gamma.o) + d(h_1.o, o))} &= \e^{-\alpha_{\varepsilon}(d(h_1.o, \gamma.o) + d(h_1.o, o))} \e^{-\alpha_{\varepsilon}(d(h_1.o, \gamma.o) + d(h_1.o, o))} \\
&\leq \e^{-\alpha_{\varepsilon}d(\gamma.o, o)} \e^{-\alpha_{\varepsilon}d(h_1.o, o)} \, . 
\end{align*}
Finally,
\begin{align*}
\sum_{\gamma \in \Gamma} \int_{G_n} \e^{-2\alpha_{\varepsilon}(d(h_1.o, \gamma.o) + d(h_1.o, o))} dm_{G_n}(h_1) \leq \sum_{\gamma \in \Gamma} \e^{-\alpha_{\varepsilon}d(\gamma.o, o)} \int_{G_n} \e^{-\alpha_{\varepsilon}d(h_1.o, o)} dm_{G_n}(h_1) \, ,
\end{align*}
which is finite by Lemma \ref{LemmaHyperbolicLatticeCriticalExponent}.
\end{proof}

\subsection{A formula for the diffraction measure}
\label{A formula for the diffraction measure}

Let $\Gamma < G_n$ be a lattice and $\nu \in \Prob(G_n)$ be right-$K_n$-invariant and absolutely continuous with respect to the Haar measure satisfying the conditions in Lemma \ref{LemmaHyperbolicPerturbedLatticeSquareIntegrabilityCriterion}, so that the random perturbed lattice $\mu_{\Gamma, \nu}$ is locally square-integrable. Then from the computations in example \ref{ExamplePerturbedLatticeAutocorrelation}, the autocorrelation measure of the random perturbed lattice orbit $\mu_{\Gamma, \nu}$ is 
\begin{align*}
\eta_{\Gamma, \nu} = |\Gamma_o|^2 \check{\nu} * \eta_{\Gamma} * \nu + \frac{1}{\covol(\Gamma)}(\delta_e - \check{\nu} * \nu)  
\end{align*}
as a linear functional on $\Borelbndinfty(G_n, K_n)$, where $\Gamma_o = \Stab_{\Gamma}(o) = \Gamma \cap K_n$. If we let $\hat{\eta}_{\Gamma}$ denote the diffraction measure of the random lattice orbit $\mu_{\Gamma}$ and $\hat{\eta}_{\Gamma, \nu}$ that of $\mu_{\Gamma, \nu}$, then
\begin{align*}
\int_{\Omega_n^+} \hat{\varphi}(\lambda) d\hat{\eta}_{\Gamma, \nu}(\lambda) = |\Gamma_o|^2 \int_{\Omega_n^+} \hat{\varphi}_{\nu}(\lambda) d\hat{\eta}_{\Gamma}(\lambda) + \frac{1}{\covol(\Gamma)} \int_0^{\infty} \big(\hat{\varphi}(\lambda) - \hat{\varphi}_{\nu}(\lambda)\big) \frac{d\lambda}{|c_n(\lambda)|^2} 
\end{align*}
for all $\varphi \in \Borelbndinfty(G_n, K_n)$, where $\varphi_{\nu} = \nu * \varphi * \check{\nu}$. If $\nu$ happens to be bi-$K_n$-invariant, then by Lemma \ref{LemmaGelfandsTrick} we can write
\begin{align*}
\eta_{\Gamma, \nu} = |\Gamma_o|^2 \check{\nu} * \nu * \eta_{\Gamma} + \frac{1}{\covol(\Gamma)}(\delta_e - \check{\nu} * \nu)  \, . 
\end{align*}
Considering the spherical transform of $\nu$ given by
$$ \hat{\nu}(\lambda) = \int_{G_n} \omega_{\lambda}(g) d\nu(g) $$
we can write the diffraction measure $\hat{\eta}_{\Gamma, \nu}$ as 
\begin{align}
\label{EqHyperbolicPerturbedLatticeDiffraction}
d\hat{\eta}_{\Gamma, \nu}(\lambda) = |\Gamma_o|^2 |\hat{\nu}(\lambda)|^2 d\hat{\eta}_{\Gamma}(\lambda) + \frac{1}{\covol(\Gamma)} (1 - |\hat{\nu}(\lambda)|^2) \frac{d\lambda}{|c_n(\lambda)|^2} \, .
\end{align}

\subsection{Random perturbed lattice orbits in $\H^n$ are not spectrally hyperuniform}
\label{Random perturbed lattices on Hn are not spectrally hyperuniform}
Let $\Gamma < G_n$ be a lattice and $\nu \in \Prob(G_n)$ a bi-$K_n$-invariant probability measure that is absolutely continuous with respect to the Haar measure $m_{G_n}$, satisfying the exponential decay condition in Lemma \ref{LemmaHyperbolicPerturbedLatticeSquareIntegrabilityCriterion}.  
\begin{proof}[Proof of Proposition \ref{PropHyperbolicLatticePerturbationDestroysHyperuniformity}]
From the formula for the diffraction measure $\hat{\eta}_{\Gamma, \nu}$ in Equation \ref{EqHyperbolicPerturbedLatticeDiffraction} we cut off the first term to get the lower bound
\begin{align*}
\limsup_{\varepsilon \rightarrow 0^+} \frac{\hat{\eta}_{\Gamma, \nu}((0, \varepsilon])}{\varepsilon^3} &\geq \frac{1}{\covol(\Gamma)} \limsup_{\varepsilon \rightarrow 0^+} \frac{1}{\varepsilon^3} \int_0^{\varepsilon} (1 - |\hat{\nu}(\lambda)|^2)\frac{d\lambda}{|c_n(\lambda)|^2}  \, . 
\end{align*}
By item (3) of Lemma \ref{LemmaPropertiesofSphericalFunctions} we have that
\begin{align*}
\hat{\nu}(\lambda) = \int_{G_n} \omega_{\lambda}(g) d\nu(g) \leq \int_{G_n} \omega_{0}(g) d\nu(g) = \hat{\nu}(0) \, , 
\end{align*}
so
\begin{align*}
\limsup_{\varepsilon \rightarrow 0^+} \frac{\hat{\eta}_{\Gamma, \nu}((0, \varepsilon])}{\varepsilon^3} &\geq \frac{1}{\covol(\Gamma)} (1 - |\hat{\nu}(0)|^2)\limsup_{\varepsilon \rightarrow 0^+} \frac{1}{\varepsilon^3} \int_0^{\varepsilon} \frac{d\lambda}{|c_n(\lambda)|^2}  \, . 
\end{align*}
Since $|c_n(\lambda)|^{-2} \asymp \lambda^{2}$ when $\lambda \rightarrow 0^+$ then there is a constant $c_o > 0$ such that $|c(\lambda)|^{-2} \geq c_o \lambda^2$ for all sufficiently small $\lambda \geq 0$, so that
$$ \limsup_{\varepsilon \rightarrow 0^+} \frac{1}{\varepsilon^3} \int_0^{\varepsilon} \frac{d\lambda}{|c_n(\lambda)|^2} \geq \frac{c_o}{3} \, .  $$
Moreover, by item (2) in Lemma \ref{LemmaPropertiesofSphericalFunctions}, $\omega_0(g) < 1 $ for all $g \neq e$, and since $m_{G_n}(K_n) = 0$ then $\nu(K_n) = 0$ by absolute continuity, so
\begin{align*}
\hat{\nu}(0) = \int_{G_n} \omega_0(g) d\nu(g) < \nu(G_n) = 1  \, . 
\end{align*}
Gathering these lower bounds,
\begin{align*}
\limsup_{\varepsilon \rightarrow 0^+} \frac{\hat{\eta}_{\Gamma, \nu}((0, \varepsilon])}{\varepsilon^3} \geq \frac{c_o}{3 \,  \covol(\Gamma)} (1 - |\hat{\nu}(0)|^2) > 0 \, . 
\end{align*}
\end{proof}
\begin{remark}
Although we have not introduced diffraction measures for random measures on other homogeneous spaces, we do so in upcoming work for general commutative metric spaces $X = G/K$ with $K$ compact. This includes semisimple Lie groups $G$ with finite center and maximal compact subgroup $K$, and there the same argument as in the above proof holds, so no perturbation of a stealthy lattice yields a hyperuniform point process.  
\end{remark}

As a consequence of Proposition \ref{PropHyperbolicLatticePerturbationDestroysHyperuniformity}, we prove that the number variance of a perturbed lattice orbit grows as fast as the volume along \emph{all} subsequences of radii.
\begin{corollary}
Let $\Gamma < G_n$ be a lattice and $\nu \in \Prob(G_n)$ be bi-$K_n$-invariant as in Lemma \ref{LemmaHyperbolicPerturbedLatticeSquareIntegrabilityCriterion}. Then
\begin{align*}
\liminf_{R \rightarrow +\infty} \frac{\NV_{\mu_{\Gamma, \nu}}(R)}{\Vol_{\H^n}(B_R)} > 0 \, . 
\end{align*}
\begin{proof}
By the formula for the diffraction measure of $\mu_{\Gamma, \nu}$ in equation \ref{EqHyperbolicPerturbedLatticeDiffraction} we can bound the number variance form below using the spherical Plancherel formula as
\begin{align*}
\NV_{\mu_{\Gamma, \nu}}(R) &\geq  \frac{1}{\covol(\Gamma)} \int_0^{\infty} |\hat{\chi}_{B_R}(\lambda)|^2 (1 - |\hat{\nu}(\lambda)|^2) \frac{d\lambda}{|c_n(\lambda)|^2} \\
&\geq \frac{1}{\covol(\Gamma)}(1 - |\hat{\nu}(0)|^2) \int_0^{\infty} |\hat{\chi}_{B_R}(\lambda)|^2 \frac{d\lambda}{|c_n(\lambda)|^2} \\
&=  \frac{1}{\covol(\Gamma)} (1 - |\hat{\nu}(0)|^2) \int_G |\chi_{B_R}(g)|^2 dm_{G_n}(g) \\
&= \frac{\Gamma(\frac{n}{2})}{\covol(\Gamma) 2\pi^{\frac{n}{2}}} (1 - |\hat{\nu}(0)|^2) \Vol_{\H^n}(B_R) \, . 
\end{align*}
Since $\hat{\nu}(0) < 1$ the right hand side is strictly positive.
\end{proof}
\end{corollary}